\RequirePackage{silence}
\documentclass[a4paper,12pt]{amsart}
\usepackage[colorlinks, linkcolor=blue,anchorcolor=Periwinkle,
urlcolor=Emerald]{hyperref}
\usepackage{amsmath, amscd, amssymb,amsfonts,amsthm,amscd,graphicx}
\usepackage{array, float, mathrsfs, mathtools, stmaryrd, youngtab, enumitem}
\usepackage{cleveref}
\usepackage{comment}
\usepackage[rgb]{xcolor} 
\definecolor{mygreen}{rgb}{0,0.7,0.3}
\definecolor{myblue}{rgb}{0,0.50,1.20}
\definecolor{myorange}{rgb}{1,0.5,0.1}
\definecolor{fillred}{rgb}{1,0.9,0.9}
\definecolor{fillgreen}{rgb}{0.9,1,0.9}
\newcommand{\red}[1]{#1} 
\newcommand{\redv}[1]{#1}
\usepackage{quiver}
\usepackage{psfrag}
\usepackage{mathtools}
\usepackage{adjustbox}
\usepackage{extarrows}
\usepackage[all]{xy} 
\usepackage{tikz}
\usepackage{tikz-cd}

\usetikzlibrary{arrows,chains,shapes,matrix,positioning,scopes} 
\usetikzlibrary{decorations.pathreplacing}
\usetikzlibrary{matrix,arrows}
\usetikzlibrary{decorations.markings}
\tikzstyle arrowstyle=[scale=1.1]
\tikzstyle directed=[postaction={decorate,decoration={markings,
		mark=at position 1 with {\arrow[arrowstyle]{latex}}}}]
\tikzstyle reverse directed=[postaction={decorate,decoration={markings,
		mark=at position .45 with {\arrowreversed[arrowstyle]{latex};}}}]

\newtheorem{theorem}{Theorem}[section]

\newtheorem{lemma}[theorem]{Lemma}
\newtheorem{proposition}[theorem]{Proposition}
\newtheorem{prop}[theorem]{Proposition}
\newtheorem{corollary}[theorem]{Corollary}

\newtheorem{conjecture}[theorem]{Conjecture}
\newtheorem{theorem-definition}[theorem]{Theorem-Definition}
\newtheorem{theorem-construction}[theorem]{Theorem-Construction}
\newtheorem{lemma-definition}[theorem]{Lemma--Definition}
\newtheorem{lemma-construction}[theorem]{Lemma--construction}
\newtheorem{definition}[theorem]{Definition}

\newtheorem{remark}[theorem]{Remark}
\theoremstyle{definition}

\newtheorem{example}[theorem]{Example}

\newenvironment{redtext}
  {}

\newcommand{\bg}{\begin{equation}\begin{gathered}}
		\newcommand{\eg}{\end{gathered}\end{equation}}

\newcommand{\N}{\mathbb{N}}
\newcommand{\Z}{\mathbb{Z}}

\newcommand{\old}[1]{}

\newcommand{\be}{\begin{equation}}
	\newcommand{\ee}{\end{equation}}
\newcommand{\bt}{\begin{theorem}}
	\newcommand{\et}{\end{theorem}}
\newcommand{\bd}{\begin{definition}}
	\newcommand{\ed}{\end{definition}}
\newcommand{\bp}{\begin{proposition}}
	\newcommand{\ep}{\end{proposition}}

\newcommand{\bl}{\begin{lemma}}
	\newcommand{\el}{\end{lemma}}
\newcommand{\bc}{\begin{corollary}}
	\newcommand{\ec}{\end{corollary}}
\newcommand{\bcon}{\begin{conjecture}}
	\newcommand{\econ}{\end{conjecture}}

\newcommand{\ie}{{i.e.}\ }

\newcommand{\ko}{\: , \;}
\newcommand{\ul}[1]{\underline{#1}}

\numberwithin{equation}{subsection}

\newtheorem{classification-theorem}[subsection]{Classification Theorem}
\newtheorem{decomposition-theorem}[subsection]{Decomposition Theorem}
\newtheorem{proposition-definition}[subsection]{Proposition-Definition}
\newtheorem{periodicity-conjecture}[subsection]{Periodicity Conjecture}

\numberwithin{theorem}{subsection}

\newcommand{\reminder}[1]{}

\newcommand{\opname}[1]{\operatorname{\mathsf{#1}}}

\renewcommand{\mod}{\mathrm{mod}\,}
\newcommand{\gpr}{\mathrm{gpr}\,}
\newcommand{\dgp}{\mathrm{dgp}\,}
\newcommand{\mpr}{\mathrm{mpr}\,}
\newcommand{\tria}{\mathrm{tria}\,}
\newcommand{\mor}{\mathrm{mor}\,}

\newcommand{\cosg}{\mathrm{cosg}\,}
\newcommand{\Cosg}{\mathrm{Cosg}\,}

\newcommand{\rep}{\opname{rep}\nolimits}

\newcommand{\Mod}{\mathrm{Mod}\,}
\newcommand{\proj}{\mathrm{proj}\,}

\newcommand{\per}{\mathrm{per}\,}
\newcommand{\pvd}{\mathrm{pvd}\,}
\newcommand{\Pvd}{\mathrm{Pvd}\,}
\newcommand{\add}{\mathrm{add}\,}

\renewcommand{\Im}{\opname{Im}\nolimits}

\newcommand{\holim}{\mathrm{holim}}

\newcommand{\tw}{\mathrm{tw}}

\newcommand{\we}{\wedge}

\newcommand{\Se}{\mathbb{S}}
\newcommand{\can}{\mathrm{can}}

\renewcommand{\rep}{\opname{rep}\nolimits}

\newcommand{\res}{\mathrm{res}}

\newcommand{\Hmo}{\mathrm{Hmo}}

\newcommand{\colim}{\mathrm{colim}}
\newcommand{\cok}{\mathrm{cok}\,}

\renewcommand{\ker}{\mathrm{ker}\,}
\newcommand{\fib}{\mathrm{fib}\,}

\newcommand{\cone}{\mathrm{cone}}

\newcommand{\iso}{\xrightarrow{_\sim}}

\newcommand{\id}{\mathbf{1}}

\newcommand{\Hom}{\mathrm{Hom}}

\newcommand{\RHom}{\mathrm{RHom}}

\newcommand{\Ext}{\mathrm{Ext}}

\newcommand{\ten}{\otimes}
\newcommand{\lten}{\overset{\boldmath{L}}{\ten}}

\newcommand{\ca}{{\mathcal A}}
\newcommand{\cb}{{\mathcal B}}
\newcommand{\cc}{{\mathcal C}}
\newcommand{\cd}{{\mathcal D}}
\newcommand{\ce}{{\mathcal E}}
\newcommand{\cF}{{\mathcal F}}

\newcommand{\ch}{{\mathcal H}}
\newcommand{\ci}{{\mathcal I}}

\newcommand{\bk}{k}

\newcommand{\cm}{{\mathcal M}}
\newcommand{\cn}{{\mathcal N}}

\newcommand{\cp}{{\mathcal P}}
\newcommand{\cR}{{\mathcal R}}

\newcommand{\cs}{{\mathcal S}}
\newcommand{\ct}{{\mathcal T}}

\newcommand{\cv}{{\mathcal V}}

\newcommand{\cx}{{\mathcal X}}
\newcommand{\cy}{{\mathcal Y}}

\newcommand{\Ga}{\Gamma}
\newcommand{\La}{\Lambda}
\newcommand{\Si}{\Sigma}
\newcommand{\Om}{\Omega}
\newcommand{\si}{\sigma}

\renewcommand{\phi}{\varphi}

\renewcommand{\tilde}[1]{\widetilde{#1}}

\begin{document}

\title[A Higgs category for triples of flags]{A Higgs category for\\
the cluster variety of triples of flags}

\author{Bernhard Keller}
\address{Universit\'e Paris Cit\'e and Sorbonne Université, CNRS, IMJ-PRG, F-75013 Paris, France}
\email{bernhard.keller@imj-prg.fr}
\urladdr{https://webusers.imj-prg.fr/~bernhard.keller/}

\author{Miantao Liu}
\address{Universit\'e Paris Cit\'e and Sorbonne Université, CNRS, IMJ-PRG, F-75013 Paris, France. Now: 
Yau Mathematical Sciences Center, Tsinghua University, Beijing, 100084, China}
\email{miantaoliu@mail.tsinghua.edu.cn}
\urladdr{https://ymsc.tsinghua.edu.cn/en/info/1035/3748.htm}


\keywords{Higher Teichm\"uller space}
\subjclass{Primary 35XX; secondary 35B33, 35A24}


\dedicatory{To Dmitri Orlov on the occasion of his 60th birthday}


\begin{abstract}
The cluster variety of triples of flags (associated with a split simple Lie group of Dynkin type $\Delta$) 
plays a key role in higher Teichm\"uller theory as developed by Fock--Goncharov, Jiarui Fei, Ian Le, \ldots\ 
and Goncharov--Shen. We refer to it as the basic triangle associated with $\Delta$.
In this paper, for simply laced $\Delta$, we construct and study a Higgs category 
(in the sense of Yilin Wu) which we expect to categorify the basic triangle. 
This category is a certain exact dg category  (in the sense of Xiaofa Chen) 
which is Frobenius and stably $2$-Calabi--Yau.  We show that it has
indeed the expected cyclic group symmetry and that its derived category 
has the expected braid group symmetry. A key ingredient in our construction is 
a conjecture by Merlin Christ, whose proof occupies most of this paper. 
The proof is based on a new description of the Higgs category in terms of 
Gorenstein projective dg modules. Our techniques are in the spirit of
Orlov in his work on triangulated categories of graded $B$-branes.
\end{abstract}

\maketitle
\tableofcontents

\section{Introduction}

Cluster algebras are certain commutative algebras endowed with
a rich combinatorial structure. They were invented by Fomin--Zelevinsky \cite{FominZelevinsky02}  
motivated by questions from Lie theory, and more precisely from the study
of canonical bases in quantum groups and total positivity in algebraic groups.  
However, it has turned out that cluster algebras are also relevant in a large array of other subjects, cf.~for example 
Fomin's cluster algebras portal \cite{FominPortal} or the introduction to
\cite{CasalsKellerWilliams23} and the references given there. The {\em (additive) categorification} of cluster algebras
via  $2$-Calabi--Yau {\em triangulated categories} was initiated in \cite{BuanMarshReinekeReitenTodorov06} for
acyclic quivers and generalized to arbitrary (cluster) quivers with (non-degenerate) potential
in \cite{Palu08a, Amiot09, Plamondon11, Plamondon11a} with crucial input from
\cite{DerksenWeymanZelevinsky08, DerksenWeymanZelevinsky10}. In parallel,
Geiss--Leclerc--Schr\"oer developed additive categorification of 
classes of cluster algebras arising in Lie theory using stably 
$2$-Calabi--Yau {\em Frobenius exact categories}, cf.~for example 
\cite{GeissLeclercSchroeer05,
GeissLeclercSchroeer06,
GeissLeclercSchroeer08a,
Leclerc10,
GeissLeclercSchroer12,
GeissLeclercSchroeer13}. 
These two types of additive categorification, via triangulated categories on the
one hand and via Frobenius exact categories on the other,  were recently unified and
generalized in the work of Yilin Wu \cite{Wu21, Wu23, Wu23a}, cf.~also \cite{KellerWu23}.
With a given ice quiver with potential, Wu associates its (relative) cluster category,
which is algebraic triangulated, and its Higgs category, which is a certain
extension closed subcategory of the cluster category. It is therefore an
extriangulated category in the sense of Nakaoka--Palu \cite{NakaokaPalu19} and
has a canonical enhancement to an exact dg category in the sense of Xiaofa Chen \cite{Chen23}. 
When the frozen part of the given ice quiver is empty, the Higgs category and the
relative cluster category both specialize
to Amiot's cluster category \cite{Amiot09} associated to the quiver with potential.
When the ice quiver with potential is associated with Geiss--Leclerc--Schr\"oer's
cluster structure on a maximal unipotent subgroup of a simple algebraic
group \cite{GeissLeclercSchroeer08a}, then Yilin Wu's Higgs category specializes to the category
of finite-dimensional modules over the corresponding preprojective
algebra and its derived category is canonically equivalent to
Wu's (relative) cluster category, cf.~Example~8.19 in \cite{Wu23a}. 
It seems very likely that these equivalences
generalize to all the other classes of examples considered by
Geiss--Leclerc--Schr\"oer. Another important example of a Higgs
category is Jensen--King--Su's Grassmannian cluster category
\cite{JensenKingSu16}, as explained in section~7 of \cite{KellerWu23}.

In this paper, we propose an ice quiver with potential whose associated
Higgs category is expected to categorify the $K_2$-cluster variety of triples of
flags as introduced by Fock--Goncharov \cite{FockGoncharov06a}, and
further studied by Jiarui Fei \cite{Fei17, Fei17a, Fei17b, Fei19, Fei21}, 
Ian Le  \cite{Le19, Le19a}, \ldots\  and Goncharov--Shen \cite{GoncharovShen19}.
To construct this ice quiver with potential, we fix a field $k$ and a quiver $Q$ whose
underlying graph is a given simply laced Dynkin diagram $\Delta$.
Following Wu \cite[Ex.~8.19]{Wu23}, we now use a relative $3$-Calabi--Yau 
completion \cite{Yeung16} (cf.~\cite{Keller11b} for the absolute case)
to construct the desired ice quiver with potential. 
Namely, expanding on an idea of Fei \cite{Fei17}, we consider the relative
$3$-Calabi--Yau completion of the functor
\[
\begin{tikzcd}
\cb = \cb_{-1} \coprod \cb_0 \coprod \cb_1 \arrow{r}  & \ca
\end{tikzcd}
\]
where  $\ca$ is the category $\mpr(kQ)$ of morphisms $P_1 \to P_0$
between finitely generated projective modules over the path algebra
$kQ$ of $Q$ and the functor $\cb_i \to \ca$ is the inclusion of
the subcategory of the objects $\id_P: P \to P$ (for $i=0$),
respectively the objects $P \to 0$
(for $i=1$), respectively the objects $0 \to P$ (for $i=-1$).
The relative $3$-Calabi--Yau completion is well-defined since
$\cb$ and $\ca$ are smooth when considered as dg categories
(this is clear for $\cb$ and easy for $\ca$, cf. 
section~\ref{ss: Passage to the derived categories and functors}). 
The output of the (reduced \cite[3.6]{Wu23a}) relative Calabi--Yau completion  is a dg functor
\begin{equation} \label{eq: Pi2 -> Pi3}
\begin{tikzcd}
\Pi_2(\cb) \arrow{r} & \Pi_3(\ca,\cb) \ko
\end{tikzcd}
\end{equation}
endowed with a (left) relative $3$-Calabi--Yau structure in the 
sense of Brav--Dyckerhoff \cite{BravDyckerhoff19}, where $\Pi_2(\cb)$
is the absolute $2$-Calabi--Yau completion of $\cb$ and
$\Ga=\Pi_3(\ca,\cb)$ the relative derived $3$-preprojective dg
algebra of $\ca$ over $\cb$.  A variant of Theorem~6.10 of
\cite{Keller11b} shows that $\Ga$ is isomorphic to the
relative $3$-dimensional Ginzburg dg algebra of an ice quiver with
potential $(R_Q, F_Q, W_Q)$, cf.~Remark~\ref{rk: ice quiver with potential}. 
Here, the inclusion of the frozen subquiver $F_Q \subseteq R_Q$
extends to the functor (\ref{eq: Pi2 -> Pi3}).
We will show elsewhere that $(R_Q, F_Q)$ is indeed isomorphic to an ice quiver describing
the cluster structure on the variety of triples of flags, namely
the quiver associated by Goncharov--Shen \cite[sect.~10]{GoncharovShen19}
to a triangle with three colored boundary components and any reduced expression for
the longest element $w_0 \in W_\Delta$ which is adapted to
the orientation $Q$ of $\Delta$. In other words, the (upper) cluster
algebra $\ca$ associated with $(R_Q, F_Q)$ is isomorphic
to the (homogeneous) coordinate algebra of the variety of
triples of flags. 

Let us write $\cp_{dg} \subseteq \Ga$ for the full dg subcategory whose
objects are in the image of $\Pi_2(\cb)$ under the functor (\ref{eq: Pi2 -> Pi3}).
Starting from this functor, Wu \cite{Wu23a} defines two categories:
\begin{itemize}
\item[1)] the {\em cluster category $\cc$}: It is expected to categorify the 
cluster algebra $\ca$ with {\em coefficients made invertible} and is defined 
as the quotient of the perfect derived category of $\Ga$
 by the thick subcategory generated by the semi-simple quotients
 of the dg modules $H^0(X)$, where $X\in H^0(\Ga)$ does not have a
 non zero summand in $\cp=H^0(\cp_{dg})$;
\item[2)] the {\em Higgs category $\ch \subseteq \cc$}: It is expected to 
categorify the cluster algebra $\ca$ with {\em non-invertible coefficients} 
and is defined as the full subcategory of the cluster category on the objects $X$
satisfying the `Gorenstein condition'
\[
\Ext^p_\cc(X,P)=0=\Ext^p_\cc(P,X)
\]
for all $p>0$ and all objects $P\in \cp_{dg}$.
\end{itemize}
In the case where the quiver $Q$ is of type $A_1$, there are no non-frozen
vertices so that the cluster category is equivalent to the perfect derived
category of $\Ga$ and the Higgs category identifies with the category
of finitely generated projective modules over the algebra $H^0(\Ga)$
(which is $6$-dimensional and selfinjective). It is not hard to check
that in this case, the quotient functor
\[
\begin{tikzcd} 
\cc \arrow{r} & \cosg(\Ga)
\end{tikzcd}
\]
induces an equivalence of $k$-linear categories
\[
\begin{tikzcd} 
\ch \arrow{r}{_\sim} & \cosg(\Ga) \ko
\end{tikzcd}
\]
where $\cosg(\Ga)$ is the {\em cosingularity category}
\[
\cosg(\Ga) = \per(\Ga)/\pvd(\Ga)
\]
obtained by quotienting the perfect derived category by the
thick subcategory generated by {\em all} simple dg $\Ga$-modules
$S_i$. In fact, this functor \red{even induces} isomorphisms in
$\tau_{\leq 0} \RHom$. As shown by Christ \cite{Christ22a},
this fact generalizes to all relative Ginzburg algebras
associated with triangulations of marked surfaces
(without punctures). The fact \red{that, for} surfaces, the
cosingularity category carries so much information came
as a surprise to the experts since this category 
{\em vanishes} for the examples arising from
Geiss--Leclerc--Schr\"oer's work.  Indeed, in those
examples, the relative Ginzburg algebra is smooth
and proper (and concentrated in degree $0$). Nevertheless,
Christ conjectured \cite{Christ24} that for the Ginzburg algebras
$\Ga=\Pi_3(\ca,\cb)$ of (\ref{eq: Pi2 -> Pi3}), the
Higgs category should be equivalent to the
cosingularity category, cf.~Conjecture~\ref{conj: Christ}. 
This is an important ingredient in Christ's approach
to the categorification of the Goncharov--Shen moduli
space associated with an arbitrary colored decorated
surface, cf. \cite{Christ25b, Christ25a}. 
Our first main result is a proof
of Christ's conjecture, cf.~Theorem~\ref{conj true}.
The proof of the essential surjectivity of the canonical functor
\[
\begin{tikzcd}
\Phi: \ch \arrow{r} & \cosg(\Ga)
\end{tikzcd}
\]
in section~\ref{ss: Essential surjectivity of Phi} is not difficult
in view of the computations done in section~\ref{section4}.
However, the full faithfulness of $\Phi$ is much harder
to come by. We obtain it in section~\ref{ss: Full faithfulness of Phi}
using a new description of the Higgs
category via {\em Gorenstein projective dg modules}
over the category $\cp_{dg}$ of projective-injectives of $\ch_{dg}$,
where we call an object `projective-injective' if it is so as an
object of the extriangulated category $\ch=H^0(\ch_{dg})$.
The fact that such a description is possible is suggested
by Iyama--Kalck--Wemyss--Yang's `Morita theorem' for
Frobenius categories `admitting a non commutative resolution'
in \cite{IyamaKalckWemyssYang15}. However, technically,
our situation is quite different. Indeed, the basic fact
needed for the `Morita theorem' is the existence of
projective resolutions. However, in a general extriangulated
category with enough projectives, such resolutions do not
carry enough information as one sees by considering
the case of a triangulated category, which is extriangulated
with enough projectives but where all projectives are
zero objects. We therefore need to introduce a
stronger notion, that of {\em projective domination}
(which depends on the datum of a dg enhancement),
cf.~section~\ref{ss: Projective domination}. In 
Cor.~\ref{cor: The Higgs category is projectively dominated},
we prove that the Higgs category is projectively dominated,
which implies in particular that the Yoneda functor 
composed with restriction to $\cp_{dg}$ yields a fully
faithful functor
\[
\begin{tikzcd}
\ch \arrow{r} & \cd(\cp_{dg}).
\end{tikzcd}
\]
It remains to characterize its image. In Theorem~\ref{thm: Higgs=gpr},
we show that the image consists precisely of the {\em Gorenstein
projective dg modules} as defined by Z.~Ding in his ongoing
Ph.~D.~thesis \cite{Ding25}, 
cf.~section~\ref{ss: Gorenstein projective modules}.
Let us point out that this definition makes sense for arbitrary
connective (small) dg categories and, for dg algebras concentrated
in degree $0$, coincides with the definition via the existence of
complete projective resolutions. We expect that the dg
algebras we consider (namely any dg algebra which is derived Morita
equivalent to $\cp_{dg}$) are in fact Gorenstein in the sense of
Frankild--\red{J\o rgensen}  \cite{FrankildJoergensen03}, cf.~also
\cite{FrankildIyengarJoergensen03}. However,
the dg modules $M$ which they consider (which, in our situation, have
their homology of finite total dimension) are not suitable for
our purposes since we do need dg modules which have
non vanishing homology in infinitely many (non positive)
degrees, for example the representable dg modules $\cp_{dg}(?, P)$.
Similarly, we cannot use the elegant theory of proper Gorenstein dg 
algebras and modules developed by Haibo Jin in \cite{Jin20}
because the dg categories occurring in our examples are
not proper. For the same reason, our dg algebras do not fit into
the setting of Brown--Sridhar's noncommutative  generalization \cite{BrownSridhar25}
of Orlov's theorem.

In section~\ref{s: Comparison with the cosingularity category}, 
we prove Christ's conjecture using the equivalence
\[
\begin{tikzcd}
\ch \arrow{r}{_\sim} & \gpr(\cp_{dg})
\end{tikzcd}
\]
between the Higgs category and the category of Gorenstein projective
dg modules over $\cp_{dg}$. The idea in the proof of the 
key Proposition~\ref{prop: bijection in Hom} is in the spirit
of Orlov in his proof of Theorem~16 in \cite{Orlov09} but technically different 
because we do not have a second grading and
we consider the category of Gorenstein projective
dg modules itself instead of the associated singularity category. 

Our second main result is the construction of the
expected group actions:
\begin{itemize}
\item[1)] the cyclic group of order $6$ acts naturally on
the Higgs category $\ch$ and the cluster category $\cc$,
cf.~section~\ref{ss: The cyclic group action};
\item[2)] the braid subgroup $B^*_\Delta$ acts naturally
on the cluster category $\cc$ (but this action {\em does not}
stabilize the Higgs category $\ch\subset \cc$), 
cf.~section~\ref{ss: The braid group action}.
\end{itemize}
The cyclic action on the Higgs category in 1) corresponds to Goncharov--Shen's
action by {\em cluster automorphisms} whereas the braid action in 2) on the
cluster category (i.e. the derived category of the exact dg category
enhancing the Higgs category) corresponds to Goncharov--Shen's
action by {\em quasi-cluster automorphisms}. 

Our construction of the cyclic group action in section~\ref{ss: The cyclic group action} 
is fairly straightforward and based on general facts about relative Calabi--Yau
structures. On the other hand, our construction of the braid
action in section~\ref{ss: The braid group action} strongly relies 
on our proof of Christ's conjecture in sections~\ref{section4} to 
\ref{s: Comparison with the cosingularity category} as well as
on Mizuno--Yang's recent classification \cite{MizunoYang24} of silting objects
in the perfect derived category of $\Pi_2(kQ)$.

\subsection*{Acknowledgment}
\red{The authors are deeply grateful to an anonymous referee for carefully reading the manuscript
and making numerous helpful suggestions.}
This article is part of the second-named author's  Ph.~D.~thesis. He would like to thank his Ph.~D.~supervisor, the
first-named author, for his guidance, patience, and kindness. He acknowledges financial support from the National 
Key R\&D Program of China 2024YFA1013802 and the Chinese Scholarship Council (grant number 202206190153). 
Both authors thank Merlin Christ for inspiring
discussions and for his comments on a previous version of this paper. They are grateful to Zhenhui Ding
for allowing them to use ideas and results from his ongoing thesis \cite{Ding25} in 
section~\ref{s: The Higgs category via Gorenstein projective dg modules}.
\red{They also thank Xiaofa Chen for correcting mistakes in previous versions of 
Lemma \ref{lemma: Serre functor equivalence}
and section~\ref{ss: Gorenstein projective modules}}.


\section{Preliminaries}

\begin{redtext}
\subsection{Notations and terminology} \label{ss: notations}
Throughout this article, we write $k$ for a fixed ground field (of arbitrary characteristic). 
We write $\ten$ for $\ten_k$ and 
\[
DV=\Hom_k(V,k)
\] 
for the $k$-dual of a (dg) vector space $V$ over $k$. A complex of $k$-modules $M$ is {\em connective} 
if its homology $H^p(M)$ vanishes in all degrees $p>0$. 

A $k$-category $\ca$ is {\em Deligne finite} if it is Morita equivalent to a finite-dimensional
$k$-algebra. For example, if $A$ is a finite-dimensional $k$-algebra,
the category $\proj(A)$ of finitely generated projective $A$-modules
is Deligne finite since it is Morita equivalent to $A$.

By a {\em dg category} (cf.~\cite{Keller06d}), we always mean a dg $k$-category.
A dg category is {\em connective} if all of its
morphism complexes are connective. 

Let $\ca$ be a dg category. We write $\cd(\ca)$ for its 
derived category, whose objects are all right dg $\ca$-modules. We write 
$\per(\ca)$ for the {\em perfect derived category of $\ca$}, i.e. the full subcategory
of $\cd(\ca)$ whose objects are the perfect (equivalently: compact) objects.
We write $\pvd(\ca)$ for the {\em perfectly valued derived category of $\ca$},
i.e. the full subcategory of $\cd(\ca)$ whose objects are the dg functors
$M$ such that $MX$ is perfect in $\cd(k)$ for all objects $X$ of $\ca$. 
For an object $M$ of $\cd(\ca)$, we write
\[
M^\vee=\RHom_\ca(M, \ca)
\]
for its dual over $\ca$. This is canonically a left dg  $\ca$-module, i.e.~an object of $\cd(\ca^{op})$
and its value at an object $A$ is $\RHom_\ca(M, A^\vee)$. 
The object $M$ is {\em reflexive} if the canonical morphism
\[
\begin{tikzcd}
M \arrow{r} & (M^\vee)^\vee
\end{tikzcd}
\]
is invertible. In the
particular case where $\ca$ is the envelopping dg category $\cb^e=\cb^{op}\ten \cb$, we identify
$\cd(\ca)$ with $\cd(\ca^{op})$ using the flip isomorphism $(\cb^{op}\ten \cb)^{op} \iso \cb^{op}\ten \cb$
and consider $M^\vee$ again as an object of $\cd(\ca)$.

\end{redtext}

\subsection{Adjoints from Serre functors} \label{ss: Adjoints from Serre functors}
Let $k$ be a field.
\red{Let $\ct$} be a $\Hom$-finite triangulated $\bk$-category.
Let us assume that $\ct$ {\em admits a Serre functor \red{$\Se$}}, \ie 
$\red{\Se}$ is a triangle autoequivalence $\ct \iso \ct$ 
together with bifunctorial isomorphisms
\[
D\ct(X,Y) \iso  \ct (Y,\Se X) ,
\]
where $D=\Hom_\bk(?,\bk)$ is the duality over the ground field. 
Let $\cs$ be another $\Hom$-finite triangulated category admitting
a Serre functor (which we will also denote by \red{$\Se$}) and let 
$L: \cs \to \ct$ be a triangle functor admitting a right
adjoint $R$. Let us recall the following well-known Lemma.

\begin{lemma}  \label{Lemma: Adjoints from Serre functors}
The triangle functor \red{$\Se^{-1}R\,\Se$} is left adjoint to $L$ 
and \red{$\Se L\Se^{-1}$} is right adjoint to $R$.
\end{lemma}
\begin{proof} Let $X$ be an object of $\cs$ and $Y$ an object of $\ct$. 
We have the bifunctorial isomorphisms
 \begin{align}
 	\Hom(Y,\Se LX)\simeq D\Hom(LX,Y)\simeq D\Hom(X,RY)\simeq \Hom(RY,\Se X).
 \end{align}
If we replace $X$ with $\Se^{-1} X$, we obtain the \red{second} assertion. We get
the \red{first} assertion by replacing $Y$ with $\Se Y$.
\end{proof}

\subsection{Smooth and proper dg categories} \label{ss: reminder smooth and proper}
Let $\ca$ be a dg $k$-category.  Recall that $\ca$ is {\em proper} if all of its morphism 
complexes are perfect in $\cd(k)$ and that $\ca$ is {\em smooth} if the dg $\ca$-bimodule 
$(X,Y) \mapsto \ca(X,Y)$ is perfect as an object of the derived category of dg 
$\ca$-bimodules $\cd(\ca^{op}\ten \ca)$.

By the following well-known lemma, if $\ca$ is smooth, then $\pvd(\ca)$ is contained in $\per(\ca)$
and we define the {\em cosingularity category of $\ca$} as the Verdier quotient
\[
\cosg(\ca) = \per(\ca)/\pvd(\ca).
\]

\begin{lemma} \label{lemma: finite global dimension} 
\begin{itemize}
\item[a)] If $\ca$ is smooth, then $\pvd(\ca)$ is included in $\per(\ca)$.
\item[b)] If $\ca$ is smooth and connective,  there is an integer $N$ such that 
\[
\Hom_{\cd(\ca)}(L,M)=0
\]
if $M$ belongs to $\cd^{\leq 0}$ and $L$ to $\cd^{\geq N}$, where $\cd^{\leq 0}$ is the left 
aisle of the canonical $t$-structure on $\cd(\ca)$.
\end{itemize}
\end{lemma}

\begin{remark} Part b) generalizes the observation that a smooth algebra
concentrated in degree $0$ is of finite global dimension.
\end{remark}

\begin{proof} Since $\ca$ is smooth, the dg bimodule $\ca\in \cd(\ca^e)$ is in the closure
under finite direct sums, direct summands and extensions of finitely many objects
of the form $\Si^{p_i} \ca(Y_i, -) \ten_k \ca(?,X_i)$, where $p_i\in \Z$, $X_i,Y_i \in \ca$, $1\leq i\leq N$. 
It follows that any object $L \iso L \lten_\ca \ca$ belongs to the closure under finite direct sums,
direct summands and extensions of the objects $L(Y_i)\ten_k \Si^{p_i} \ca(?, X_i)$. This shows
that if $L$ is perfectly valued, it is perfect, which proves a). Let us prove b): 
Since $\ca$ is connective, the object $\Si^{p_i} \ca(?,X_i)$ is left orthogonal to $\cd^{\leq -p_i}$. 
So it suffices to choose $N$ greater than the maximum of the $p_i$.
\end{proof}

\begin{lemma}[Kalck--Yang] \label{lemma: KY generalized}
If $\ca$ is smooth and connective and the $k$-category $H^0(\ca)$ is
$\Hom$-finite, then $\per(\ca)$ is $\Hom$-finite.
\end{lemma}
\red{
\begin{remark} \label{rk: finite-dimensional homologies}
The lemma is proved for dg algebras in Prop.~2.5 of \cite{KalckYang16}. 
\end{remark}
}

\begin{proof} Let $P$ be a connective perfect dg $\ca$-module.
It suffices to prove that the $H^0(\ca)$-module $H^n(P)$ takes values in
finite-dimensional vector spaces for each $n\in \Z$. Consider the triangle
\[
\begin{tikzcd}
\tau_{<0}(P) \arrow{r} & P \arrow{r} & H^0(P) \arrow{r} & \Si\tau_{<0}(P).
\end{tikzcd}
\]
It follows from our assumption that $H^0(P)$ belongs to $\pvd(\ca)$. Since $\ca$ is
smooth, by part a) of Lemma~\ref{lemma: finite global dimension}, the dg module
$H^0(P)$ is perfect. Looking at the above triangle we see that $\tau_{<0}(P)$ is
an extension of $P$ by $\Si^{-1}H^0(P)$ and thus perfect. Therefore, the
object $\Si^{-1} \tau_{<0}(P)$ is both perfect and connective and so the
dg $\ca$-module 
\[
H^{-1}(P) = H^0(\Si^{-1} \tau_{<0}(P))
\]
is perfectly valued. By induction, we deduce the claim.
\end{proof}

We denote by $\Hmo$ the Morita homotopy category of small dg categories, cf.~\cite{Keller06d}.
It becomes a symmetric monoidal category when endowed with the tensor product over $k$.
Recall from \cite{Toen07} that $\Hmo$ is in fact a closed monoidal category and that
its inner $\Hom$ functor is given by 
\[
\ul{\Hom}(\ca,\cb) = \rep_{dg}(\ca,\cb).
\]
The following lemma and its proof are well-known. 

\begin{lemma} \label{lemma: inner Hom} Let $\ca$ be a smooth and proper dg $k$-category.
Then $\ca$ is dualizable in $\Hmo$ and its dual is $\ca^{op}$. In particular, we have a
canonical isomorphism of endofunctors of $\Hmo$
\[
? \ten_k \ca^{op} \iso \rep_{dg}(\ca,?).
\]
\end{lemma} 

\begin{proof}[Sketch of proof] Since $\ca$ is smooth, we have the morphism
\[
\Hom_\ca : \ca^{op} \ten \ca \to \per_{dg}(k)
\]
of $\Hmo$ and since $\ca$ is proper, we have the morphism
\[
\per_{dg}(k) \to \per_{dg}(\ca^{op}\ten\ca) \iso \per_{dg}(\ca^{op}) \ten_k \per_{dg}(\ca)
\]
taking $k$ to the identity bimodule $\ca(?,-)$ in $\Hmo$. These two morphisms
are the evaluation and the coevaluation morphisms which make $\ca^{op}$ into
the dual object of $\ca$. 
\end{proof}

\begin{lemma} \label{lemma: cosg of tensor product} Let
$\ca$ be a smooth, proper and connective dg category and let
$\cb$ be a smooth and connective dg category.
Suppose that each simple $H^0(\ca\ten\cb)$-module is isomorphic to the
tensor product of a simple $H^0(\ca)$-module with a simple $H^0(\cb)$-module.
Then we have a canonical isomorphism in $\Hmo$
\[
\cosg_{dg}(\ca\ten \cb) \iso \ca \ten \cosg_{dg}(\cb).
\]
\end{lemma}

\begin{proof} By definition, we have a short exact sequence of dg categories
\[
\begin{tikzcd} 
0 \ar{r} & \pvd_{dg}(\cb) \ar{r} & \per_{dg}(\cb) \ar{r} & \cosg_{dg}(\cb) \ar{r} & 0.
\end{tikzcd}
\]
It induces a short exact sequence
\[
\begin{tikzcd}[scale=0.8]
0 \ar{r} & \cp(\ca)\ten\pvd_{dg}(\cb) \ar{r} & \cp(\ca)\ten\cp(\cb) \ar{r} 
&  \cp(\ca)\ten\cosg_{dg}(\cb) \ar{r} & 0\ko
\end{tikzcd}
\]
where we abbreviate $\cp(\ca)=\per_{dg}(\ca)$ and similarly for $\cb$. The canonical isomorphism
of $\Hmo$
\[
\cp(\ca)\ten\cp(\cb)  \to \cp(\ca\ten\cb)
\]
induces a functor
\[
\pvd_{dg}(\ca) \ten \pvd_{dg}(\cb) \to \pvd_{dg}(\ca\ten\cb).
\]
By our assumption on the simple $H^0(\ca\ten\cb)$-modules, this functor
is also an isomorphism in $\Hmo$. The claim follows because $\pvd_{dg}(\ca)=\cp(\ca)$
since $\ca$ is proper.
\end{proof}

\subsection{The inverse dualizing bimodule} 
Let $\ca$ be a smooth dg $k$-catego\-ry.  Using the notations introduced in
section~\ref{ss: notations}, we define the {\em inverse dualizing bimodule of $\ca$} 
as the object
\[
\Omega_\ca = \ca^\vee = \RHom_{\ca^e}(\ca, \ca^e)
\]
of $\cd(\ca^e)$. 

\begin{proposition} \label{keylemma} \cite[Lemma~4.1]{Keller08d}
Suppose $\ca$ is smooth. For $L\in \pvd(\ca)$ and $M\in \cd(\ca)$, we have a canonical isomorphism
\[
\begin{tikzcd}
\Hom_{\cd \ca}(M\lten_\ca \Omega_\ca, L)\arrow{r}{\sim} & D\Hom_{\cd \ca}(L,M).
\end{tikzcd}
\]
\end{proposition}

\begin{corollary} For a smooth dg $k$-category $\ca$, the functor 
$?\lten_\ca \Omega_\ca$ takes $\pvd(\ca)$ to $\per(\ca)$ and its restriction
to $\pvd(\ca)$ is fully faithful. Moreover, if it takes $\pvd(\ca)$ to itself and induces 
an essentially surjective functor $\red{\Se'}: \pvd(\ca) \to \pvd(\ca)$, then $\red{\Se'}$ is an 
{\em inverse Serre functor} on $\pvd(\ca)$, \ie an autoequivalence such
that we have isomorphisms
\[
D\Hom_{\cd \ca}(L,M) \iso \Hom_{\cd \ca}(\red{\Se'} M,L)
\]
which are bifunctorial in $L,M\in \pvd(\ca)$.
\end{corollary}

\red{
\begin{remark}
Notice that the functor $?\lten_\ca\Omega_\ca$ does not always
take $\pvd(\ca)$ to itself. For example, if $V$ is a vector space of dimension at least $2$ and
$A=TV$ is the tensor algebra on $V$, then $?\lten_A \Omega_A$ takes the trivial module $k$ to 
the augmentation ideal of $TV$. 
\end{remark}
}

\begin{lemma} \label{lemma: Serre functor equivalence} Suppose that 
\begin{itemize}
\item[a)] $\ca$ is smooth and connective, 
\item[b)] the $k$-category $H^0(\ca)$ is $\Hom$-finite and 
\item[c)] $\ca^\vee$ is right perfect and the functor $?\lten_\ca \ca^\vee$ takes
$\pvd(\ca)$ to itself.
\end{itemize}
Then the functor $?\lten_\ca \ca^\vee: \cd(\ca) \to \cd(\ca)$ is an equivalence.
\end{lemma}

\begin{proof} Since $\ca^\vee$ is right perfect, the right adjoint $R=\RHom_\ca(\ca^\vee, ?)$ also takes
$\pvd(\ca)$ to itself. Moreover, both functors are fully faithful when restricted to $\pvd(\ca)$.
Thus, they induce quasi-inverse autoequivalences in $\pvd(\ca)$. Now consider the
adjunction morphism
\[
\begin{tikzcd}
\phi \ca: \red{L}R \ca\arrow{r} & \ca.
\end{tikzcd}
\]
Notice that under our hypotheses, all homologies $H^p(\ca)$, $p\in \Z$, are perfectly
valued $H^0(\ca)$-modules, cf.~Lemma~\ref{lemma: KY generalized}.
Thus, the truncations $\tau_{\leq p} \ca$ belong to $\pvd(\ca)$. 
Moreover, the canonical morphism
\[
\begin{tikzcd}
\ca \arrow{r} & \holim \tau_{\leq p} \ca
\end{tikzcd}
\]
is invertible. Since $\ca^\vee$ is left perfect, the functor $L$ 
commutes with homotopy limits. We deduce that the adjunction
morphism $LR \ca\to \ca$ is invertible so that the restriction
of $L$ to $\per(\ca)$ is fully faithful. By the same argument,
we see that the adjunction morphism $\ca \to RL\ca$ is invertible.
We deduce the claim.
\end{proof}

\subsection{Hom-finite Higgs categories} \label{ss: Higgs categories}
Let $\ca$ and $\cb$ be smooth, connective dg categories such that the categories
$H^0(\ca)$ and $H^0(\cb)$ are Deligne finite (cf.~section~\ref{ss: notations}).
Recall from Prop.~2.5 of \cite{KalckYang16} that this implies that
the categories $\per(\ca)$ and $\per(\cb)$ are $\Hom$-finite. 

Let $G: \cb\to\ca$ be a dg functor endowed with a left relative  $3$-Calabi--Yau structure
in the sense of \cite{BravDyckerhoff19}. We denote by $\pvd(\ca,\cb)$ the kernel of
the restriction functor $\pvd(\ca) \to \pvd(\cb)$.  Following Wu \cite{Wu23}, we
define the {\em relative cluster category $\cc_{\ca,\cb}$} associated with $G: \cb\to\ca$ 
as the Verdier quotient
\[
\per(\ca)/\pvd(\ca,\cb).
\]
Notice that since $H^0(\ca)$ is $\Hom$-finite, this quotient is idempotent complete
by Cor.~4.15 of \cite{Wu23}. Moreover, it is $\Hom$-finite by Cor.~4.18 of \cite{Wu23}.
The composition of the Yoneda functor $H^0(\cb) \to \per(\cb)$ with the
extension of scalars $\per(\cb) \to \per(\ca)$ and the projection
$\per(\ca)\to \cc_{\ca,\cb}$ is fully faithful. We denote by $\cp=\cp_{\ca,\cb}$ the
closure of its image under finite direct sums and retracts. 
We define the {\em Higgs category $\ch_{\ca,\cb}$} as the
full subcategory of the relative cluster category $\cc_{\ca,\cb}$
whose objects are the dg $\ca$-modules $X$ such that we
have
\[
\Hom(X, \Si^i P) = 0 = \Hom(P, \Si^i X)
\]
for all $P\in \cp$ and all $i>0$. By Theorem~4.14 of 
\cite{KellerWu23}, this definition is equivalent to Wu's original
definition in Def.~5.22 of \cite{Wu23}.  Notice that the
Higgs category is $\Hom$-finite as a full subcategory of
the relative cluster category. The above definition also shows
that the Higgs category is an extension closed subcategory
of the triangulated category $\cc_{\ca,\cb}$ and thus carries
a canonical extriangulated structure in the sense of
Nakaoka--Palu \cite{NakaokaPalu19}.

\begin{theorem}[Wu \cite{Wu23}] 
\begin{itemize}
\item[a)] The Higgs category $\ch_{\ca,\cb}$ is a $\Hom$-finite Frobenius extriangulated category.
\item[b)] Its associated stable category is canonically $2$-Calabi--Yau.
\item[c)] The closure under finite direct sums and retracts of the image
of $H^0(\ca)$ in $\ch_{\ca,\cb}$ is a $2$-cluster tilting subcategory.
\end{itemize}
\end{theorem}

\subsection{Relative Calabi--Yau completions and dg localizations} \label{ss: Relative Calabi-Yau completions and quotients} 
Let $F: \cb \to \ca$ be a dg functor between smooth dg categories. Let $d$ be an integer. 
\red{Let us recall the construction of the relative $d$-Calabi--Yau completion of $F$. 
The {\em inverse dualizing bimodule $\Omega_\cb$} is (a cofibrant replacement of) the
bimodule dual $\cb^\vee$ of $\cb$. 
The {\em relative inverse dualizing bimodule $\Omega_{\ca,\cb}$} is (a cofibrant replacement of) the \redv{cocone} over
the canonical morphism}
\[
\begin{tikzcd} 
\ca^\vee  \arrow{r}  & \cb^\vee\lten_{\cb^e} \ca^e 
\end{tikzcd}
\]
\red{(where $\ca^\vee=\RHom_{\ca^e}(\ca, \ca^e)$). The {\em relative $d$-Calabi--Yau
completion} of the dg functor $F: \cb \to \ca$ is the dg functor}
\[
\Pi_{d-1}(\cb) \to \Pi_d(\ca,\cb) \ko
\]
\red{where  
\begin{itemize}
\item[a)] $\Pi_{d-1}(\cb)$ is the (absolute) $(d-1)$-Calabi--Yau completion of $\cb$, i.e.~the 
(derived) tensor category over $\cb$ of $\Si^{d-2} \cb^\vee$,
\item[b)] $\Pi_d(\ca,\cb)$ is the (derived) tensor category over $\ca$
of the bimodule $\Si^{d-1} \Omega_{\ca,\cb}$.
\item[c)] the functor $\Pi_{d-1}(\cb) \to \Pi_d(\ca,\cb)$ is induced by the canonical
morphism of dg $\cb$-bimodules $\cb^\vee \to \cb^\vee\lten_{\cb^e} \ca^e$.
\end{itemize}
}

Now let $T$ be a subset of $H^0(\cb)$ and $S$ a subset of $H^0(\ca)$ such that $FT \subseteq S$. 
For example, if $\cb$ is pretriangulated, we can take for $T$ the set of morphisms whose
cone lies in a given (small) pretriangulated subcategory $\cn$ of $\cb$. 
We write $\cb[T^{-1}]$ and $\ca[S^{-1}]$ for the corresponding dg localizations. 
\red{Notice that, by part c) of Prop.~3.10 of \cite{Keller11b}, a localization of
a smooth dg category is still smooth.} Clearly, the dg functor $F$ induces a 
dg functor $\cb[T^{-1}]\to \ca[S^{-1}]$. We still denote by $S$ the 
image of $S$ under the natural functor from $H^0(\ca)$ to $H^0(\Pi_d(\ca,\cb))$.

\begin{proposition} \label{prop: Relative CY-completions and localizations}
We have a canonical quasi-equivalence
\[
\begin{tikzcd}
\Pi_d(\ca, \cb)[S^{-1}]  \arrow{r}{_\sim} & \Pi_d(\ca[S^{-1}], \cb[T^{-1}]).
\end{tikzcd}
\]
\end{proposition}

\begin{proof} The construction of the
relative inverse dualizing bimodule is compatible with localizations (the proof is analogous
to the one in the absolute case in Prop.~3.10 of \cite{Keller11b}) and forming the
tensor algebra is compatible with localizations  as one sees using Cor.~3.10 of \cite{FanKellerQiu24}.
\end{proof}

As an application, suppose that $F: \cb \to \ca$ is a dg functor between smooth connective dg categories
and $\cn\subset \cb$ a full (small) dg subcategory. Let $\cm\subseteq \ca$ be its image under $F$.
Let  \red{
\[
\Ga=\Pi_3(\ca,\cb)
\] 
}
be the relative $3$-Calabi--Yau completion. For simplicity, let
us assume that $H^0(\ca)$, $H^0(\cb)$ and $H^0(\Ga)$ are Deligne finite
(cf.~section~\ref{ss: notations}). Let $\cc$ be the cluster category associated
with $\Ga$ and $\ch\subseteq\cc$ the Higgs category. The image $\cp_0$ of $H^0(\cn)$ in $\ch$ 
is a small subcategory of the category $\cp\subseteq\ch$ of projective-injectives of $\ch$. 
By Prop.~3.32 of \cite{Chen24b}, the inclusion $\ch\subset \cc$ extends to a
canonical equivalence from $\cd^b(\ch_{dg})$ to $\cc$, where $\ch_{dg}$ is the
canonical exact dg category enhancing $\ch$. By Theorem~B of \cite{Chen24b},
the quotient $\ch_{dg}/ \cp_0$ inherits a canonical structure of exact
dg category and its bounded dg derived category is canonically equivalent
to the Verdier quotient $\cc/\langle \cp_0 \rangle$, where $\langle \cp_0 \rangle$ denotes
the thick subcategory generated by $\cp_0$. Moreover, the category $H^0(\ch_{dg}/ \cp_0)$ is
the quotient of $\ch=H^0(\ch_{dg})$ by the ideal of morphisms generated by the identities of the
objects of $\cp_0$. Let 
\red{
\[
\Ga'=\Pi_3(\ca/\cm, \cb/\cn)
\]
}
and let $\cc'$, $\ch'=H^0(\ch'_{dg})$, $\cp'=H^0(\cp'_{dg})$ be the (dg) categories associated with $\Ga'$. From the
above Proposition, we deduce the following Corollary.

\begin{corollary} \label{cor: localized cluster category} We have a canonical equivalence 
\[
\begin{tikzcd}
\per(\Ga)/\langle \cp_0 \rangle  \arrow{r}{_\sim} & \per(\Ga')
\end{tikzcd}
\]
inducing equivalences $\cc/\cp_0 \iso \cc'$,
$\ch_{dg}/ \cp_0 \iso \ch'$ and $\ch/(\cp_0) \iso \ch'$.
\end{corollary}

\red{\subsection{Exact dg categories} \label{ss: reminder exact} 
Let $\ca$ be a connective dg category.
We define $\ca$ to be {\em additive} if $H^0(\ca)$ is additive. We then define $\ca$ to
{\em have split retractions} or to be {\em weakly idempotent complete} if this
holds for $H^0(\ca)$. If $\ca$ is an additive dg category, an
{\em exact structure} on $\ca$ is the datum of a distinguished class of
homotopy short exact sequences of $\ca$ satisfying certain axioms, cf.~\cite{Chen24a, Chen24b}.
For example, each exact category in the sense of Quillen can be viewed as
an exact dg category concentrated in degree~$0$.}

\red{Now suppose that $\ca$ is an exact dg category. Then its homotopy category
$H^0(\ca)$ canonically becomes \cite{Chen24a} an extriangulated category in the sense of
Nakaoka--Palu \cite{NakaokaPalu19}. An object of $\ca$ is defined to be {\em projective}
resp. {\em injective} if it is so as an object of $H^0(\ca)$. The exact dg category $\ca$ has
{\em enough projectives} if this holds for $H^0(\ca)$ and similarly for injectives.
It is {\em Frobenius} if $H^0(\ca)$ is Frobenius. }

\red{
Let $\ca$ be a (connective) exact dg category. We write $\ca_{add}$ for the underlying
additive dg category of $\ca$. The {\em unbounded derived category of the second
kind} of $\ca$ is defined \cite{Chen24b} as the quotient of the derived category of $\ca_{add}$ 
by the localizing subcategory generated by the totalizations of the conflations
of $\ca$. The {\em bounded derived category $\cd^b(\ca)$ of $\ca$} is
defined as the triangulated subcategory of its unbounded derived category
of the second kind generated by the image of $\ca$ under the Yoneda functor. 
}

\subsection{Characterizing objects in degree \texorpdfstring{$0$}{}} 
Let $\ca$ be an extriangulated category \cite{NakaokaPalu19} with
enough injectives where each retraction is a deflation (i.e.~$\ca$ is
`weakly idempotent complete'). A variant of the following lemma for (Quillen) exact categories 
where retractions are not necessarily deflations appears as Lemma~2.6 in \cite{IyamaKalckWemyssYang15}.

\begin{lemma} \label{lemma: Characterizing inflations} 
Let $i: X \to Y$ be a morphism of $\ca$. Then $i$ is an inflation
if (and only if) it induces a surjection
\[
\begin{tikzcd}
\Hom(Y,I) \arrow{r} & \Hom(X,I)
\end{tikzcd}
\]
for each injective $I$ of $\ca$. 
\end{lemma}

\begin{proof}  The necessity of the condition is clear. To prove its sufficiency, 
we choose an inflation $j: X \to I$ with injective $I$. By our assumption,
the inflation $j$ factors as $j' \circ i$ for some morphism $j': Y \to I$.
By Prop.~2.7 of \cite{Klapproth22}, it follows that $i$ is an inflation.
\end{proof}

Let $\ce$ be a connective exact dg category \cite{Chen24a, Chen24b} where all retractions are deflations. 
Since $\ce$ is connective, the derived category \red{$\cd(\ce_{add})$} of the underlying additive 
dg category \red{$\ce_{add}$} of $\ce$ has a canonical weight structure, \red{cf.~Appendix~1 of \cite{KellerNicolas12}}. 
The bounded derived category $\cd^b(\ce)$ is defined in \cite{Chen24b} as the full subcategory
of a Verdier quotient of $\cd(\ce_{add})$ generated by the image of $\ce$ under the
Yoneda functor. When we apply the truncation operations 
to an object $X$ of $\cd^b(\ce)$, we always first choose a preimage of $X$ in $\cd(\ce_{add})$.

\begin{lemma}  \label{lemma: Characterizing objects in degree 0} 
Suppose that the extriangulated category $H^0(\ce)$ has enough injectives. 
Let $X$ be an object of $\cd^b(\ce)$ lying in 
\[
\ce * \Si \ce * \Si^2 \ce * \ldots * \Si^N \ce
\]
for some $N\geq 0$. Then $X$ lies in $\ce\subset \cd^b(\ce)$ if and 
only if we have
\[
\Ext^p_\ce(X,I)=0
\]
for all $p>0$ and all injectives $I$ of $H^0(\ce)$.
\end{lemma} 

\begin{proof} The necessity of the condition is clear since
higher extensions computed in $H^0(\ce)$ coincide with those
computed in $\cd^b(\ce)$ by Prop.~6.24 in \cite{Chen23}.
Let us show that it is sufficient. 
We proceed by induction on $N$. If $N=0$, there is
nothing to be shown. Suppose $N>0$. \red{Recall that
$\ce_{add}$ is the connective additive dg category underlying the
exact dg category $\ce$. Let $C$ be the tensor product $X \lten_{\ce_{add}} H^0(\ce_{add})$.} Then
$C$ identifies with a complex 
\[
\begin{tikzcd}
0 \arrow{r} & C_N \arrow{r} & C_{N-1} \arrow{r} & \cdots \arrow{r} & C_1 \arrow{r} & C_0 \arrow{r} & 0
\end{tikzcd}
\]
of objects $C_i$ in \red{the $k$-linear category $H^0(\ce_{add})$, which equals the
underlying $k$-linear category of the exact dg category $H^0(\ce)$}. Moreover, the
assumption that we have $\Ext^N_\ce(X,I)=0$ implies that the morphism $C_N \to C_{N-1}$
induces a surjection
\[
\begin{tikzcd}
\Hom_{H^0(\ce)}(C_{N-1}, I) \arrow{r} & \Hom_{H^0(\ce)}(C_N, I)
\end{tikzcd}
\]
for all injective $I$ in $H^0(\ce)$ (equivalently: in $\ce$). By Lemma~\ref{lemma: Characterizing inflations},
it follows that there is a conflation
\[
\begin{tikzcd}
C_N \arrow{r} & C_{N-1} \arrow{r} & B
\end{tikzcd}
\]
in $H^0(\ce)$. On the other hand, in $\cd^b(\ce)$, we have a triangle
\[
\begin{tikzcd}
\Si^{N-1} C_N \arrow{r} & \Si^{N-1} C_{N-1} \arrow{r} & \sigma^{\leq -N+1}(X) \arrow{r} & \Si C_N.
\end{tikzcd}
\]
We see that in $\cd^b(\ce)$, the object $\sigma^{\leq -N+1}(X)$ becomes isomorphic to $\Si^{N-1} B$. 
Thus, we have an isomorphism of triangles in $\cd^b(\ce)$
\[
\begin{tikzcd}
\sigma^{\geq -N+2} X  \arrow{r} \arrow[d, equals] & X \arrow{r} \arrow{d} & \sigma^{\leq -N+1}(X) \arrow{r} 
\arrow{d} & \Si \sigma^{\geq -N+2} X \arrow{d}{\id} \\
\sigma^{\geq -N+2} X \arrow{r} & X' \arrow{r} & \Si^{N-1} B \arrow{r} & \Si \sigma^{\geq -N+2} X.
\end{tikzcd}
\]
By construction, the object $X'$ lies in
\[
\ce * \Si \ce * \Si^2 \ce * \ldots * \Si^{N-1} \ce \ko
\]
and by the induction hypothesis we obtain that $X'$ and hence $X$ lie in $\ce$.
\end{proof}

\subsection{Representations up to homotopy and the dg category of triangles} 
\label{ss: Representations up to homotopy}
Let $\ca$ and $\cb$ be dg categories. 
Recall that a \emph{dg $\ca$-$\cb$-bimodule} is a right $\ca^{op}\otimes\cb$-module. 
We denote by $\rep(\ca,\cb)$ the full subcategory of $\cd(\ca^{op}\ten\cb)$ whose objects
are the dg bimodules $X$ which are {\em right perfect}, i.e. such that $X(?,A)$ is 
perfect for each object $A$ of $\ca$. Objects of $\rep(\ca,\cb)$ are sometimes
called {\em representations up to homotopy of $\ca$ in $\per(\cb)$}.

\begin{lemma} \label{lemma: tensor-rep}
If $\ca$ is smooth and proper, the functor
\[
\per_{dg}(\ca^{op}) \ten \per_{dg}(\cb) \to \rep_{dg}(\ca,\cb)
\]
taking $(M,P)$ to the dg bimodule $M\ten_k P$ is a derived Morita equivalence.
\end{lemma} 

\begin{proof} The functor is well defined since $\ca$ is proper and clearly,
it is fully faithful. Its image clearly generates $\per(\ca^{op}\ten \cb)$ and
this is a subcategory of $\rep(\ca,\cb)$. Thus, it suffices to show that
$\rep(\ca,\cb)$ is contained in $\per(\ca^{op}\ten \cb)$. Let us write
$\ca_0$ for the discrete $k$-category with the same objects as $\ca$
and $\res$ for the restriction along the canonical functor $\ca_0^{op} \ten \cb \to \ca^{op}\ten \cb$.
Then for an arbitrary dg $\ca$-$\cb$-bimodule, we have
\[
\RHom_{\ca^{op}\ten \cb}(U, V) = \RHom_{\ca^{e}}(\ca, \RHom_{\ca_0^{op}\ten\cb}(\res(U), \res(V))) \ko
\]
where $\res$ is the restriction functor from $\cd(\ca^{op}\ten \cb)$ to $\cd(\cb)$. Since
$\ca$ is smooth, the dg bimodule $\ca$ lies in the thick subcategory generated 
by the dg bimodules $\ca(Y,-)\ten \ca(?,X)$, where $(X,Y)$ runs through a finite
subset $I$ of the set of pairs of objects of $\ca$. Now we have
\[
\RHom_{\ca^e}( \ca(Y,-)\ten \ca(?,X), \RHom_{\ca_0^{op}\ten\cb}(\res(U), \res(V))) 
\iso \RHom_\cb(U(?, X), V(?,Y))
\]
and if $U$ belongs to $\rep(\ca,\cb)$, then this functor clearly commutes with
arbitrary coproducts in the argument $V$.
\end{proof}

Let us now suppose that $\cb$ is pretriangulated.  \red{Following Christ \cite{Christ24}},
we define the \emph{dg category of triangles} of $\cb$ as the dg category
\[
\rep_{dg}(\bk A_2, \cb_{}) \ko
\]
where $k A_2$ is the path $k$-category of the quiver
\[
\begin{tikzcd}
1 \arrow{r}{i}& 0.
\end{tikzcd}
\]
\red{To justify the terminology, let us recall that} we 
have a canonical functor $C_{210}$ from $\rep(\bk A_2,  \cb)$ to the category of triangles
$\tria(H^0(\cb))$ of $H^{0}(\cb)$ 
which sends a bimodule $X$ to the triangle
\[
\begin{tikzcd}
	X(?,2) \arrow{r} & X(?,1)\arrow{r}{X(?,i)} & X(?,0)\arrow{r} & \Si X(?, 2) \ko
\end{tikzcd}
\]
where $X(?,2)$ is defined as the cocone over the morphism $X(?,i)$.
This functor induces a bijection at the level of isomorphism classes of objects.
\red{Thus, up to homotopy equivalence, the objects of $\rep_{dg}(\bk A_2, \cb_{})$
can indeed be viewed as triangles in the category $H^0(\cb)$.
Notice that the functor $C_{210}$ is {\em not full} in general (the morphisms
in its image could be called `good morphisms of triangles' \cite{Neeman91}). However,
we obtain a full functor by composing $C_{210}$ with the forgetful
functor $F_0: \tria(H^0(\cb)) \to \mor(H^0(\cb))$ taking a triangle
to its first morphism, where $\mor(H^0(\cb))$ is the category of
morphisms of $H^0(\cb)$. In the resulting commutative triangle}
\[
\begin{tikzcd}
\rep(kA_2, \cb) \arrow{r}{C_{210}} \ar[dr, "C_{21}"'] & \tria(H^0(\cb)) \arrow{d}{F_0} \\
 & \mor(H^0(\cb))
 \end{tikzcd}
 \]
 \red{
 the functors $C_{21}$ and $F_0$ are even {\em epivalences}, i.e.~they are full,
 essentially surjective and detect isomorphisms}.

\section {The morphism category  and its relative Calabi--Yau completion}
\subsection{The morphism category} \label{ss: The morphism category}
Let $Q$ be a Dynkin quiver (i.e.~a quiver whose underlying graph is an ADE Dynkin diagram) 
and $\bk Q$ its path algebra.
Let $\proj(\bk Q)$ be the category of finitely generated \red{projective} (right) $\bk Q$-modules. 
We define $\mpr(\bk Q)$ to be the category of morphisms 
\[
\begin{tikzcd}
P_1 \arrow{r} & P_0
\end{tikzcd}
\]
between finitely generated projective $\bk Q$-modules. Thus, a morphism 
$f: P \to P'$ of  $\mpr(\bk Q)$ is a commutative square
\[
\begin{tikzcd} 
P_1 \arrow{r} \arrow[d, "{f_1}"' ] & P_0 \arrow{d}{f_0} \\
P'_1 \arrow{r} & P'_0.
\end{tikzcd}
\]
Let $D_0: \proj(\bk Q) \to \mpr(\bk Q)$ be the functor taking a finitely
generated projective module $P$ to the identity morphism $P \to P$.
This functor fits into a chain of adjoint functors
\begin{equation}\label{cdc}
D_{-1} \dashv  C_0 \dashv  D_0 \dashv  C_1 \dashv  D_1.
\end{equation}
Explicitly, the functors $C_0$ and $C_1$ take a morphism
$P_1 \to P_0$ to $P_0$ respectively $P_1$, and the functors
$D_{-1}$ and $D_1$ take a projective module $P$ to the
morphism $0 \to P$ respectively $P \to 0$.

Clearly, the category $\mpr(\bk Q)$ is an extension closed subcategory
of the abelian category of morphisms $M_1 \to M_0$ between arbitrary
(right) $\bk Q$-modules. Thus, it carries a canonical structure of
exact category in the sense of Quillen. It is not hard to check
(cf.~section~3 of \cite{Bautista04}) that 
\begin{itemize}
\item[a)] the functor $D_{-1}$ induces an equivalence from $\proj(\bk Q)$
onto the full {\em subcategory $\cb_{-1}$} of projective objects of $\mpr(\bk Q)$ which do not
have a non-zero projective-injective summand;
\item[b)] the functor $D_0$ induces an equivalence from $\proj(\bk Q)$
onto the full {\em subcategory $\cb_0$} of projective-injective objects of $\mpr(\bk Q)$;
\item[c)] the functor $D_1$ induces an equivalence from $\proj(\bk Q)$
onto the full {\em subcategory $\cb_1$} of injective objects of $\mpr(\bk Q)$ which do
not have a non-zero projective-injective summand.
\end{itemize}

The functor $\mpr(\bk Q) \to \mod(\bk Q)$ taking a morphism
$P_1 \to P_0$ to its cokernel induces an equivalence from
the quotient of $\mpr(\bk Q)$ by the ideal of morphisms
factoring through an injective object onto the category
$\mod (\bk Q)$, cf.~\red{\cite[Chapter~III, Section~1]{Auslander71}} and \cite[Prop. 3.3]{Bautista04}. 
It follows that the isoclasses of non-injective indecomposables
of $\mpr(\bk Q)$ are in bijection with the isoclasses of indecomposable
$\bk Q$-modules. Thus, by Gabriel's theorem, since $Q$ is a Dynkin quiver, the category
$\mpr(\bk Q)$ only has finitely many isoclasses of
indecomposables. By theorem~5.1 of \cite{Bautista04}, for an arbitrary
acyclic finite quiver $Q$, the category $\mpr(\bk Q)$ has Auslander--Reiten sequences.

\begin{example} \label{ex: A3} 
Let $Q$ be the quiver 
\[
1\rightarrow 2 \rightarrow  3.
\]
We write $P_i$ \red{for} the indecomposable projective $kQ$-module whose head
is the simple $S_i$ concentrated at $i$, $1\leq i\leq 3$. We will write $I_i$ for the
injective hull of $S_i$. The Auslander--Reiten quiver of $\mpr(\bk Q)$ is given by the
diagram

\begin{equation} \label{eq: AR-quiver of mprA3}
\begin{tikzcd}[column sep=0.5em, row sep=3.2ex]
  & & {0 \to P_3} \arrow[rd] \arrow[rr, dashed, no head] && {P_1 \to 0} \arrow[rd] \\
  & {0 \to P_2} \arrow[ru] \arrow[rd] \arrow[rr, dashed, no head] && {P_1 \to P_3} \arrow[rd] \arrow[ru] \arrow[rr, dashed, no head] && {P_2 \to 0} \arrow[rd] \\
  {0 \to P_1} \arrow[ru] \arrow[rd] \arrow[rr, dashed, no head] && {P_1 \to P_2} \arrow[ru] \arrow[rd] \arrow[rr, dashed, no head] && {P_2 \to P_3} \arrow[ru] \arrow[rd] \arrow[rr, dashed, no head] && {P_3 \to 0} \\
  & {P_1 \to P_1} \arrow[ru] & & {P_2 \to P_2} \arrow[ru] & & {P_3 \to P_3} \arrow[ru] \\
\end{tikzcd}
\end{equation}
\end{example}

The triple of functors $(D_{-1}, D_0, D_1)$ defines a $\bk$-linear functor
\[
\begin{tikzcd}
\cb \arrow{r}{\Phi} & \ca ,
\end{tikzcd}
\]
where $\cb$ is the coproduct (in the category of $\bk$-linear categories)
of three copies of $\proj(\bk Q)$ and $\ca$ is the category $\mpr(\bk Q)$.
In the following sections, we will study the relative $3$-Calabi--Yau completion
of $\Phi$ in the sense of \cite{Yeung16}, cf.~also section~3.6  of \cite{Wu23}.

Clearly, the category $\cb$ is smooth and proper and we show 
in section~\ref{ss: Passage to the derived categories and functors} that
this also holds for $\ca$. It follows that the dg categories 
$\Pi_2(\cb)$ and $\Ga=\Pi_3(\ca,\cb)$ are smooth (but they are
never proper!) and that
$H^0(\Pi_2)$ and $H^0(\Ga)$ are Morita-equivalent to finite-dimensional
algebras. Moreover, both $\Pi_2(\cb)$ and $\Ga$ are connective.
Thus, the general theory developed by Wu in \cite{Wu23a} applies
and, to the dg functor
\[
\begin{tikzcd}
 \Pi_2(\cb) \arrow{r} &  \Pi_3(\ca,\cb) \ko
 \end{tikzcd}
 \]
we can associate the (relative) cluster category $\cc$, which is $\Hom$-finite
and triangulated, and the Higgs category $\ch\subseteq \cc$,
which is extension-closed in $\cc$ and therefore carries
a canonical extriangulated structure \cite{NakaokaPalu19}.
In more detail, let us call an object of $\Ga$ frozen if
it is a sum of images of objects in $\cb$. 
Then the cluster category $\cc$ is obtained
as the quotient of the perfect derived category $\per(\Ga)$
by the thick subcategory generated by \redv{the simple $H^0(\Ga)$-modules 
which vanish on the frozen objects $P$ of $\Ga$.}
An object $X\in\cc$ belongs to the Higgs category
$\ch\subseteq \cc$ if and only if we have
\begin{equation} \label{eq: gp-cond}
\Ext^i_\cc(X,P)=0=\Ext^i_\cc(P,X)
\end{equation}
for all $i>0$ and all objects represented by frozen
objects of $\Ga$. These are in fact precisely the
\red{projective-injective} objects of the Higgs category
(considered as an extriangulated category) and the
condition~\ref{eq: gp-cond} is necessary because
$\cc$ identifies with the derived category of 
(the dg enhancement) of $\ch$ by Prop.~3.32
of \cite{Chen24b}. 

\subsection{Passage to the derived categories and functors} \label{ss: Passage to the derived categories and functors}
We keep the notations and assumptions of the previous section. 
In particular, the symbol $\ca$ denotes the category $\mpr(\bk Q)$.
By an {\em $\ca$-module}, we mean a $\bk$-linear functor 
\[
\begin{tikzcd}
\ca^{op} \arrow{r} &  \Mod \bk.
\end{tikzcd}
\]
We write $\mod \ca$ for the category of finitely presented right $\ca$-modules. For example,
for each object $A$ of $\ca$, the {\em representable} $\ca$-module $A^\wedge=\ca(?,A)$
belongs to $\mod \ca$ and, \red{if $A$ is an additive generator of $\ca$, then $A^\wedge$ is
a projective generator for $\mod \ca$}.
Since submodules of projective $\bk Q$-modules are projective,
the category $\ca$ has kernels (computed as kernels in the category of morphisms of all $\bk Q$-modules).
Therefore, if a finitely presented $\ca$-module $M$ has a projective presentation
\[
\begin{tikzcd}
A_1^\wedge \arrow{r}{f^\wedge} & A_0^\wedge \arrow{r} & M \arrow{r} & 0
\end{tikzcd}
\]
for a morphism $f: A_1 \to A_0$ of $\ca$, then we obtain a projective resolution
\[
\begin{tikzcd}
0 \arrow{r} & (\ker f)^\wedge \arrow{r} & A_1^\wedge \arrow{r}{f^\wedge} & A_0^\wedge \arrow{r} & M \arrow{r} & 0.
\end{tikzcd}
\]
Thus, each object of $\mod \ca$ is of projective dimension at most $2$. 
Since $Q$ is a directed quiver, the simples over $\ca^{op}\ten \ca$ are tensor products of
simples over $\ca^{op}$ with simples over $\ca$ so that the category
$\mod(\ca^{op}\ten \ca)$ is of projective dimension at most $4$. We conclude 
that the category $\ca$ is (homologically) smooth when considered as a dg category
concentrated in degree $0$.  Thus, the category $\ca$ is smooth and proper so that the perfectly
valued derived category $\pvd(\ca)$ equals the perfect
derived category $\per(\ca)$.

If $F: \cx \to \cy$ is a dg functor between dg categories,
we still denote by 
\[
F: \cd(\cx) \to \cd(\cy)
\]
the left adjoint of the restriction along $F$. In this way, the functor
$D_0: \proj(\bk Q) \to \ca$ yields a functor
$D_0: \cd(\bk Q) \to \cd(\ca)$ which clearly
induces a functor 
\[
D_0: \per(\bk Q) \to \per(\ca).
\]
This functor fits into an infinite chain of adjoints
\begin{equation}\label{CIDI}
\ldots \dashv C_{-1} \dashv D_{-1} \dashv  C_0 \dashv  D_0 \dashv  C_1 \dashv  D_1 \dashv C_2 \dashv \ldots \ ,
\end{equation}
where the five functors $D_{-1}$, \ldots, $D_1$ are induced by the
corresponding functors between $\proj(\bk Q)$ and $\ca$,
cf.~section~\ref{ss: The morphism category}. From Lemma~\ref{Lemma: Adjoints from Serre functors},
we obtain natural isomorphisms
\begin{equation} \label{eq: adjoints}
\red{\Se^{-1} D_i \Se \iso D_{i-1} }\quad\mbox{and}\quad \red{\Se^{-1} C_i \Se \iso C_{i-1}}
\end{equation}
for all integers $i$, where we denote all Serre functors by $\Se$.

\subsection {The relative Calabi--Yau completion} \label{ss: Relative 3-CY-completion}
Clearly, the category $\cb$ is also smooth
($\cb$ was defined at the end of section~\ref{ss: The morphism category}).
Thus, the functor
\[
\begin{tikzcd}
\Phi: \cb \arrow{r} & \ca ,
\end{tikzcd}
\]
has a well-defined (relative) $3$-Calabi--Yau completion \cite{Yeung16} given by 
the corresponding Ginzburg functor
\[
\begin{tikzcd}
G: \Pi_2(\cb) \arrow{r} & \Pi_3(\ca,\cb) \ko
\end{tikzcd}
\]
\red{where $\Pi_2(\cb)$ is the absolute $2$-Calabi--Yau completion of $\cb$ and $\Pi_3(\ca,\cb)$
the $3$-dimensional relative derived preprojective category, cf.~\cite{Yeung16} and \cite{KellerWang23}.}
By definition, the {\em boundary dg category $\cp_{dg}$} is the
full dg subcategory of $\Pi_3(\ca,\cb)$ whose objects are the images $GX$ of objects $X$ of $\Pi_2(\cb)$.
We describe the morphism complexes of $\cp_{dg}$ in \red{Theorem~\ref{theorem1} below}. For this,
we introduce the following notation:
Clearly, the dg category $\Pi_2(\cb)$ is quasi-equivalent
to a coproduct of three copies of $\Pi_2(\proj(\bk Q))$ and the dg 
functor $G$ is given by a triple of dg functors 
\[
\begin{tikzcd}
G_i:\Pi_2(\proj(\bk Q)) \arrow{r} & \Pi_3(\ca,\cb) \ko  
\end{tikzcd}
\]
where $-1\leq i \leq 1$.  

\begin{remark} \label{rk: ice quiver with potential} \red{In the Morita homotopy category of dg $k$-categories
$\mathrm{Hmo}_k$, cf.~\cite{Keller06},  the dg category} $ \Pi_3(\ca,\cb)$ is isomorphic to the relative $3$-dimensional 
Ginzburg dg algebra of an ice quiver with potential $(R_Q, F_Q, W_Q)$. 
In Example~\ref{ex: A3}, the quiver $R_Q$ is
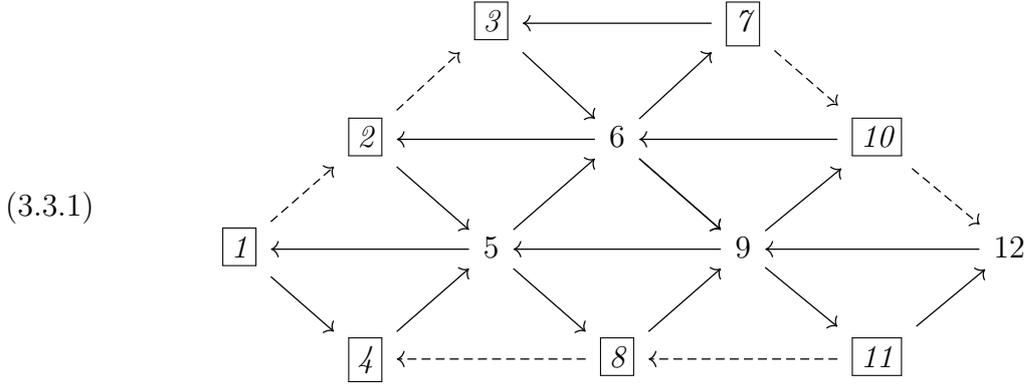
\begin{equation} \label{eq: quiver R_Q}
\begin{tikzcd}[column sep=2em, row sep=3.5ex]
  & & \framebox{3} \arrow[rd]  && 
  \framebox{7} \arrow[rd, dashed]  \arrow[ll]\\
  & \framebox{2} \arrow[ru, dashed] \arrow[rd] && 
  6 \arrow[ll]  \arrow[rd] \arrow[ru] \arrow[rd] && 
  \framebox{10} \arrow[rd, dashed] \arrow[ll] \\
  \framebox{1} \arrow[rd] \arrow[ru, dashed]  && 
  5 \arrow[ru] \arrow[ll] \arrow[rd] && 
  9 \arrow[ru] \arrow[rd] \arrow[ll] & &
  12 \arrow[ll] \\
  & \framebox{4} \arrow[ru] & & \framebox{8} \arrow[ru]  \arrow[ll, dashed] & & \framebox{11} \arrow[ru]   \arrow[ll, dashed] \\
\end{tikzcd}
\end{equation}
Here the framed vertices are frozen and so are the dashed arrows.
This quiver is obtained from the Auslander--Reiten quiver of Example~\ref{ex: A3} by the
following general rules:
\begin{itemize}
\item[1)] adding a `reverse' unfrozen arrow $j \to i$ in each mesh going from $i$ to $j$;
\item[2)] adding a `reverse' frozen arrow from $j$ to $i$ whenever the arrow $i \to j$ in
the quiver of $\proj(kQ)$ is mapped to a composition of arrows $i \to l \to j$ under
the functor $P \mapsto (\id_P: P \to P)$. 
\end{itemize}
This ice quiver is endowed with a potential obtained by summing contributions
from the meshes of the Auslander--Reiten quiver of $\mpr(kQ)$. Namely,
if in this quiver we have a mesh with associated mesh relation
\[
\sum_{i=1}^s \alpha_i \beta_i =0 \ko
\]
then we add the sum
\[
\rho \sum_{i=1}^s \alpha_i \beta_i
\]
to the potential, where $\rho$ is the reverse arrow associated with the mesh. The fact that
this ice quiver with potential has its relative Ginzburg algebra isomorphic to the
relative derived preprojective algebra $\Pi_3(\ca,\cb)$ will be proved elsewhere.
It follows from the relative variant of Theorem~6.10 of \cite{Keller11b} and the fact that
the category $\mpr(kQ)$ is {\em standard}, i.e. its subcategory of indecomposables
is canonically equivalent to the mesh category of its Auslander--Reiten quiver. 
\end{remark}

\begin{theorem} \label{theorem1}
For all objects $X$ and $Y$ of $\Pi_2(\proj(kQ))$ and integers $-1\leq i\leq 1$, we have canonical
isomorphisms
\begin{align}
\RHom(X,Y)  &\iso \RHom(G_i X, G_i Y) \\
\RHom(X,Y)  &\iso \RHom(G_{-1} X, G_0 Y ) \\
\RHom(X,Y)  &\iso \RHom(G_0 X, G_1 Y) 
\end{align}
and the spaces $\RHom(G_i X, G_{i-1} Y)$ vanish for $-1 \leq i  \leq 1$,
where we put $G_{-2}=G_1$. Moreover, we have a canonical isomorphism
\begin{align}
\tau_{\leq 0} \Si^{-1}\RHom(X,Y)  &\iso \RHom(G_1 X, G_{-1} Y).
\end{align}
Here, we write $\RHom$ for the morphism
complex in the canonical dg enhancement of the derived category
of $\Pi_2(\proj(kQ))$ respectively $\Pi_3(\ca,\cb)$.
\end{theorem}
We will prove the theorem in section \ref{ss: Proof of Thm 1} after preparing the ground in
sections \ref{ss: easy isomorphisms}  and \ref{ss: study of the functor F}. To establish
the link between the boundary dg category and the relative cluster category in
section~\ref{s: The Higgs category via Gorenstein projective dg modules},
we also need the following generalization of Theorem~\ref{theorem1}.
\redv{Here, we write $Y \mapsto C(Y)$ for the functor from $\cd(\Pi_3(\ca,\cb))$
to $\cd(kQ)$ induced by the functor from $\mpr(kQ)$ to the
category of complexes $\cc^b(\proj kQ)$ which takes a morphism
$P_1 \to P_0$ between projectives to its cone.}

\begin{theorem} \label{theorem2}
For all integers $-1\leq i\leq 1$, all objects $X$ of $\Pi_2(\proj(\bk Q))$ and 
all objects $Y$ of $\Pi_3(\ca,\cb)$,  we have canonical
isomorphisms
\begin{align}
\RHom(G_{-1} X,Y)  &\iso \RHom(X, C_0 Y) \\
\RHom(G_0 X,Y)  &\iso \RHom(X, C_1 Y ) \\
\RHom(G_1 X,Y)  &\iso \tau_{\leq 0} \Si^{-1}\RHom(X, C(Y) \lten_{kQ} \Pi_2) 
\end{align}
Here, we write $?\lten_{kQ} \Pi_2$ for the functor induced by the canonical 
functor $\proj(kQ) \to \Pi_2(\proj kQ)$. Moreover, we write $\RHom$ for the morphism
complex in the canonical dg enhancement of the derived category
of $\Pi_2(\proj(\bk Q))$ respectively $\Pi_3(\ca,\cb)$.
\end{theorem}
We will prove this theorem in section~\ref{ss: Proof of Thm 2}.

\section{Morphism complexes of the boundary dg category}
\label{section4}

\subsection{Easy isomorphisms} \label{ss: easy isomorphisms}
Let $\Phi: \cb \rightarrow \ca $ be the dg functor defined in section~\ref{ss: Relative 3-CY-completion}. 
For a dg $\ca$-bimodule $M \in \cd(\ca^e)$, we write 
\[
M^\vee = \RHom_{\ca^e}(M, \ca^e)
\]
for its bimodule dual viewed as an object in $\cd(\ca^e)$ via the canonical anti-isomorphism
$\ca^e \iso (\ca^e)^{op}$. The composition map $\ca \ten_\cb \ca \to \ca$ yields a canonical
morphism
\[
\begin{tikzcd}
\ca^\vee \arrow{r} & (\ca\lten_\cb\ca)^\vee.
\end{tikzcd}
\]
The {\em relative inverse dualizing bimodule} associated with $\Phi$ is defined as
the shifted homotopy fiber (=cocone)
\[
\red{\Omega_{\ca,\cb}} = \fib(\ca^\vee \to (\ca\lten_\cb\ca)^\vee).
\]
Let  $F = ?\lten_\ca \red{\red{\Si^2 \Omega_{\ca,\cb}}}$ be the associated derived tensor functor $\cd(\ca) \to \cd(\ca)$
\red{composed with $\Si^2$}.

Let $X$ and $Y$ be objects of $\Pi_2(\proj(\bk Q))$ and $i,j$ integers among $-1$, $0$ and $1$.
We will compute
\[
\RHom_{\Pi_3(\ca,\cb)}(G_i X , G_j Y)
\]
in Theorem~\ref{theorem1} using the following isomorphism
\begin{equation} \label{eq: orbit formula}
\RHom_{\Pi_3(\ca,\cb)} (G_i X, G_j Y) \iso \bigoplus_{p\geq 0} \RHom_\ca(D_i X, F^p D_j Y).
\end{equation}

Let us first describe the functor $F = ?\lten_\ca \Si^2 \Omega_{\ca,\cb}$
in more detail. The functor $?\lten_\ca \ca^\vee$ induces the inverse Serre
functor 
\begin{equation} \label{eq: induced Serre functor}
\red{\Se}^{-1}: \per(\ca) \to \per(\ca)
\end{equation}
(recall from section~\ref{ss: Relative 3-CY-completion} that $\pvd(\ca)$ and 
$\per(\ca)$ coincide). The category $\cb$ is smooth so that we have a 
canonical isomorphism
\[
\ca\lten_{\cb}\cb^{\vee}\lten_{\cb}\ca\iso(\ca\lten_{\cb}\cb\lten_{\cb}\ca)^{\vee} \iso(\ca\lten_{\cb}\ca)^{\vee},
\]
where we write $\cb^\vee$ for the derived dual of $\cb$ as a dg $\cb$-bimodule \red{and have used the fact
that $\cb$ is perfect over $\cb^e$ for the first isomorphism}. It follows
that the derived tensor product with this dg bimodule takes an object $X$ of
$\per(\ca)$ to
\begin{equation} \label{eq: sum of three}
(D_{-1} \red{\Se}^{-1} C_0 X) \oplus 
(D_0 \red{\Se}^{-1} C_1 X) \oplus
(D_1 \red{\Se}^{-1} C_2 X).
\end{equation}
By combining the descriptions (\ref{eq: induced Serre functor}) and (\ref{eq: sum of three}), we readily deduce that the functor $F=?\lten_\ca \red{\Si^2 \Omega_{\ca,\cb}}$
sends an object $X$ of $\per(\ca)$ to the homotopy fiber of the canonical
morphism
\begin{equation} \label{eq: fiber}
\begin{tikzcd}
\Si^2 \Se^{-1} X \arrow{r} & \Si (D_{-1} \tau^{-1} C_0 X) \oplus  \Si (D_0 \tau^{-1} C_1 X) \oplus \Si (D_1 \tau^{-1} C_2 X) \ko
\end{tikzcd}
\end{equation}
\red{where $\tau=\Si^{-1} \Se$ denotes the Auslander--Reiten translation of 
the derived category $\per(kQ)=\cd^b(\mod kQ)$}.
Let us illustrate this construction using the following example.
\begin{example} We continue Example~\ref{ex: A3}.
For $-1\leq i\leq 1$, let $\cb_i = \Im D_i$. 
Clearly, we have canonical equivalences
\[
\cb_{-1}\simeq \cb_{0}\simeq \cb_{1}\simeq \proj \bk Q, ~\cb =\cb_{-1}\amalg\cb_{0}\amalg \cb_{1}.
\]
The Auslander--Reiten quiver of the exact category $\ca$ is given in
(\ref{eq: AR-quiver of mprA3}). We number the indecomposables as indicated in the following diagram

\begin{equation} \label{eq: AR quiver of mpr(kA_3)}
\begin{tikzcd}[column sep=2em, row sep=3.5ex]
  & & 3 \arrow[rd] \arrow[rr, dashed, no head] && 7 \arrow[rd]  \\
  & 2 \arrow[ru] \arrow[rd] \arrow[rr, dashed, no head] && 6 \arrow[rr, dashed, no head] \arrow[rd] \arrow[ru] \arrow[rd] && 10 \arrow[rd]  \\
  1 \arrow[rd] \arrow[ru] \arrow[rr, dashed, no head] && 5 \arrow[ru] \arrow[rd] \arrow[rr, dashed, no head] && 9 \arrow[rr, dashed, no head] \arrow[ru] \arrow[rd] & &12 \\
  & 4 \arrow[ru] & & 8 \arrow[ru] & & 11 \arrow[ru] \\
\end{tikzcd}
\end{equation}

For $1\leq i \leq 12$, we denote by $P_{i}^{\ca}$ the representable $\ca$-module $\ca(?, i)$
(i.e. the indecomposable projective $\ca$-module corresponding to $i$). For example,
the projective-injective indecomposables correspond exactly to the vertices $4$, $8$ and $11$.
The quotient of $\ca$ by the ideal of morphisms factoring through a projective-injective
is equivalent to the full subcategory of the derived category $\cd^{b}(\mod \bk Q)$ whose
objects are the complexes with projective components concentrated in (cohomological)
degrees $-1$ and $0$. This is an extension-closed subcategory of the derived
category and thus an extriangulated category. Its Auslander--Reiten quiver
is the following subquiver of the Auslander--Reiten quiver of $\cd^{b}(\mod \bk Q)$

\begin{equation} \label{eq: AR quiver of 2-term complexes up to homotopy}
\begin{tikzpicture}[every node/.style={anchor=center}, font=\small, scale=0.9]
\matrix[matrix of math nodes, column sep=0.2em, row sep=9ex] (m) {
  && &  \quad P_3^{A_3} & & \quad \Sigma P_1^{A_3} & & \Sigma \tau^{-1} P_1^{A_3} &   \\
 & & \quad P_2^{A_3} & &\quad I_2^{A_3} & & \Sigma P_2^{A_3} & & \Sigma I_2^{A_3} &   \\
  &  P_1^{A_3} & & S_2^{A_3} & & S_3^{A_3} & & \Sigma P_3 & & \Sigma^{2} P_1^{A_3} &  \\
};
\draw[dashed] (m-1-4) -- (m-1-6);
\draw[dashed] (m-1-6) -- (m-1-8);
\draw[dashed] (m-2-3) -- (m-2-5);
\draw[dashed] (m-2-5) -- (m-2-7);
\draw[dashed] (m-2-7) -- (m-2-9);
\draw[dashed] (m-3-2) -- (m-3-4);
\draw[dashed] (m-3-4) -- (m-3-6);
\draw[dashed] (m-3-6) -- (m-3-8);
\draw[dashed] (m-3-8) -- (m-3-10);

\draw[->] (m-3-2) -- (m-2-3);
\draw[->] (m-3-4) -- (m-2-5);
\draw[->] (m-3-6) -- (m-2-7);
\draw[->] (m-3-8) -- (m-2-9);

\draw[->] (m-2-3) -- (m-1-4);
\draw[->] (m-2-5) -- (m-1-6);
\draw[->] (m-2-7) -- (m-1-8);

\draw[->] (m-1-4) -- (m-2-5);
\draw[->] (m-1-6) -- (m-2-7);
\draw[->] (m-1-8) -- (m-2-9);
\draw[->] (m-2-9) -- (m-3-10);

\draw[->] (m-2-3) -- (m-3-4);
\draw[->] (m-2-5) -- (m-3-6);
\draw[->] (m-2-7) -- (m-3-8);

\end{tikzpicture}
\end{equation}
With the notations of (\ref{eq: AR quiver of mpr(kA_3)}), we obtain that 
\begin{align*}
&F(P_{1}^{\ca})= (P_{4}^{\ca}\to P_5^{\ca}), 
&F(P_{2}^{\ca})= (P_{4}^{\ca}\to P_6^{\ca}), \\
&F(P_{3}^{\ca})= (P_{4}^{\ca}\to P_7^{\ca}), 
&F(P_{5}^{\ca})= (P_{4}^{\ca}\to P_9^{\ca}), \\
&F(P_{6}^{\ca})= (P_{4}^{\ca}\to P_{10}^{\ca}), 
&F(P_{8}^{\ca})= (P_{4}^{\ca}\to P_{11}^{\ca}), \\
&F(P_{9}^{\ca})= (P_{4}^{\ca}\to P_{12}^{\ca}).
\end{align*}
\end{example}

\begin{lemma} \label{lemma: CF-FD}
We have the following isomorphisms between functors linking the categories
$\per(kQ)$ and $\per(\ca)$:
 \begin{align}
 C_{-1} F &\iso 0 ,\\
C_{0}F  &\iso \tau^{-1}C_{0},\\
C_{1}F  &\iso \tau^{-1}C_{1},\\
\label{eq: FD_0} FD_{0}	&\iso D_0\tau^{-1},\\
\label{eq: FD_1} FD_{1} &\iso D_1\tau^{-1}.
\end{align}
\end{lemma}

\begin{proof}
For the vanishing of $C_{-1} F$, we use that $C_{-1} D_0$ vanishes (since its right adjoint
$C_1 D_{-1}$ vanishes), that $C_{-1} D_1$ vanishes, that $C_{-1} D_{-1}$ is the identity
and that $C_{-1} \red{\Se^{-1}}$ is isomorphic to $\red{\Se^{-1}} C_0$.
To prove the second isomorphism, we need to compute the homotopy fiber $\cF$ of the image under $C_0$ of
the morphism
\[
\begin{tikzcd}
\Si^2 \red{\Se}^{-1}  \arrow{r} & \Si^2 (D_{-1} \red{\Se}^{-1} C_0 ) \oplus  \Si^2 (D_0 \red{\Se}^{-1} C_1 ) \oplus \Si^2 (D_1 \red{\Se}^{-1} C_2 ).
\end{tikzcd}
\]
We have the isomorphisms $C_0 D_{-1} \iso \id$, $C_0 D_0 \iso \id$ and $C_0 D_1 \iso 0$. It follows that $\cF$ is
also the homotopy fiber of the morphism
\[
\begin{tikzcd}
\Si^2 C_0 \red{\Se}^{-1}  \arrow{r} & \Si^2 \red{\Se}^{-1} C_0  \oplus  \Si^2 \red{\Se}^{-1} C_1 .
\end{tikzcd}
\]
The second component of this morphism is induced by the canonical isomorphism
$C_0 \red{\Se}^{-1} \iso \red{\Se}^{-1} C_1$ of Lemma~\ref{Lemma: Adjoints from Serre functors}. Therefore, the homotopy fiber $\cF$ is
also the homotopy fiber of the morphism
\[
\begin{tikzcd}
0 \arrow{r} & \Si^2 \red{\Se}^{-1} C_0,
\end{tikzcd}
\]
which is clearly isomorphic to $\Si \red{\Se}^{-1} C_0 = \tau^{-1} C_0$. The proof of the third isomorphism is similar.
To prove the fourth isomorphism, we need to study the homotopy fiber $\cF'$ of the morphism
\[
\begin{tikzcd}
\Si^2 \red{\Se}^{-1}D_0  \arrow{r} & \Si^2 (D_{-1} \red{\Se}^{-1} C_0  D_0) \oplus  \Si^2 (D_0 \red{\Se}^{-1} C_1 D_0 ) \oplus 
\Si^2 (D_1 \red{\Se}^{-1} C_2  D_0).
\end{tikzcd}
\]
We have the isomorphisms $C_0 D_0 \iso \id$, $C_1 D_0 \iso \id$ and $C_2 D_0 \iso 0$. Therefore, 
the homotopy fiber $\cF'$ is also the homotopy fiber of the morphism
\[
\begin{tikzcd}
\Si^2 \red{\Se}^{-1}D_0  \arrow{r} & \Si^2 (D_{-1} \red{\Se}^{-1}) \oplus  \Si^2 (D_0 \red{\Se}^{-1} ).
\end{tikzcd}
\]
The first component of this morphism is induced by the canonical isomorphism $\red{\Se}^{-1} D_0 \iso D_{-1} \red{\Se}^{-1}$
of Lemma~\ref{Lemma: Adjoints from Serre functors}.
Therefore, the homotopy fiber $\cF'$ is also the homotopy fiber of the morphism
\[
\begin{tikzcd}
0 \arrow{r} &\Si^2 (D_0 \red{\Se}^{-1} ),
\end{tikzcd}
\]
which is clearly $D_0 \Si \red{\Se}^{-1} = D_0 \tau^{-1}$. The proof of the fifth isomorphism is similar.
\end{proof}

\begin{lemma}\label{upperhalfproof}
Let $p\geq 0$ be an integer. For objects $X$ and $Y$ of $\Pi_2(\proj(\bk Q))$, we have the canonical 
isomorphisms
\begin{align*}
\RHom(D_{-1} X, F^p D_{-1} ~Y) &= \RHom(X, \tau^{-p} Y),\\
 \RHom(D_{-1} X, F^p D_0~ Y)  &=   \RHom(X, \tau^{-p} Y),\\
 \RHom(D_{-1} X, F^p D_1 ~Y) &= 0,\\
 \RHom(D_0 X, F^p D_{-1}~ Y) &= 0, \\
\RHom(D_0 X, F^p D_0~ Y) &= \RHom(X, \tau^{-p} Y) ,\\
\RHom(D_0 X, F^p D_1 ~Y) &= \RHom(X, \tau^{-p} Y),\\
\RHom(D_1 X, F^p D_0 ~Y) &= 0,\\
\RHom(D_1 X, F^p D_1 ~Y) &= \RHom(X, \tau^{-p} Y).
\end{align*}
As in Theorem~\ref{theorem1}, we write $\RHom$ for the morphism
complex in the canonical dg enhancement of the derived category
of $\Pi_2(\proj(\bk Q))$ respectively $\Pi_3(\ca,\cb)$.
\end{lemma}

\begin{proof}
For the first three isomorphisms, we use the fact that $D_{-1}$ is left adjoint to $C_0$, the isomorphism
$C_0 F^p \iso \tau^{-p} C_0$ of Lemma~\ref{lemma: CF-FD} and finally the isomorphisms
$C_0 D_{-1} \iso \id$, $C_0 D_0 \iso \id$ and $C_0 D_1 \iso 0$. For the fourth, fifth and sixth 
isomorphisms we use the fact that $D_0$ is left adjoint to $C_1$, the isomorphism
$C_1 F^p \iso \tau^{-p} C_1$ and finally the isomorphisms $C_1 D_{-1} \iso 0$, $C_1 D_0 \iso \id$ and
$C_1 D_1 \iso \id$. For the last two isomorphisms, we use the fact that $D_1$ is left adjoint
to $C_2$, the isomorphisms $F^p D_i \iso D_i \tau^{-p}$ for $0\leq i \leq 1$ and finally the two 
isomorphisms $C_2 D_0 \iso 0$ and $C_2 D_1 \iso \id$.
\end{proof}

\subsection{Study of the functor \texorpdfstring{$F$}{}}  \label{ss: study of the functor F}
We keep the notations and assumptions of section~\ref{ss: easy isomorphisms}. 
Recall that the functor $F=?\lten_\ca \red{\Si^2 \Omega_{\ca,\cb}}$ from $\per(\ca)$ to itself
sends an object $X$ of $\per(\ca)$ to the homotopy fiber (\ref{eq: fiber}) of the canonical
morphism
\begin{equation}
\begin{tikzcd}
\Si^2 \red{\Se}^{-1} X \arrow{r} & \Si (D_{-1} \tau^{-1} C_0 X) \oplus  \Si (D_0 \tau^{-1} C_1 X) \oplus \Si (D_1 \tau^{-1} C_2 X).
\end{tikzcd}
\end{equation}

\begin{lemma} The right adjoint $F_\rho$ of $F$ sends an object $X$ to the homotopy cofiber
of the canonical morphism
\begin{equation}
\begin{tikzcd}
\Si^{-2}D_0\red{\Se}C_0 X \oplus \Si^{-2}D_1\red{\Se}C_1X  \oplus  \Si^{-2} D_{2}\red{\Se}C_2 X \arrow{r} & \Si^{-2}\red{\Se}X.
\end{tikzcd}
\end{equation}
\end{lemma}
We leave the proof to the reader.  

For an object $X=(P_1 \to P_0)$ of $\ca$ (i.e.~a 
morphism of $\proj(\bk Q)$), we denote by $P_X$ the corresponding projective
$\ca$-module $\ca(?,X)$. If $X$ is indecomposable, we write $S_X$ for its simple top. 
We denote by $\tau$ the Auslander--Reiten translation of the exact category $\ca$. By definition, 
it vanishes on the projectives of $\ca$. We write $\ca_m$ for the full subcategory of
$\ca$ whose objects lie in the kernel of $C_2 : \ca \to \proj(\bk Q)$. Notice that
the objects of $\ca_m$ are the injective morphisms $f: P_1 \to P_0$. The functor taking
$f$ to its cokernel induces an equivalence from $\ca_m/\Im D_0$ onto $\mod \bk Q$. 

\begin{lemma} \label{lemma: Frho} Let $X=(P_1 \to P_0)$ be an indecomposable object of $\ca$.
\begin{itemize}
\item[a)] If $X$ belongs to $\ca_m$ and is not projective-injective, then $F_\rho S_X = S_{\tau(X)}$.
\item[b)] If $X$ is injective and not projective, then the restriction of $F_\rho S_X$ to $\ca_m$ 
is isomorphic to $S_{\tau(X)}$.
\end{itemize}
\end{lemma}

\begin{proof} a) We first assume that $X$ is not projective in $\ca$. Then we have an
Auslander--Reiten conflation
\[
\begin{tikzcd}
0 \arrow{r} & \tau X \arrow{r} & E \arrow{r} & X \arrow{r} & 0.
\end{tikzcd}
\]
It yields a projective resolution of the simple $S_X$ as follows
\[
\begin{tikzcd}
0 \arrow{r} & P_{\tau X} \arrow{r} & P_E \arrow{r} & P_X \arrow{r} & S_X \arrow{r} & 0.
\end{tikzcd}
\]
It follows that the image of $S_X$ under the Serre functor is quasi-isomorphic to the
complex
\[
\begin{tikzcd}
\cdots \arrow{r} 0 \arrow{r} & I_{\tau X} \arrow{r} & I_E \arrow{r} & I_X \arrow{r} & 0 \arrow{r} & \cdots
\end{tikzcd}
\]
where $I_{\tau X}$ is in degree $-2$. We deduce a quasi-isomorphism between $\Si^{-2}\red{\Se}(S_X)$ and
$S_{\tau(X)}$. Since $X$ is non-projective and non-injective in $\ca$, we have 
\[
C_0(S_X)=0, \quad C_1(S_X)=0, \quad C_2(S_X)=0.
\]
It follows that $F_\rho(S_X)$ is isomorphic to $S_{\tau(X)}$. 

Now let us assume that $X$ is projective (and not projective-injective). 
Then $S_X$ belongs to the image of $D_{-1}$ and $F_\rho D_{-1}=0$ since it is
right adjoint to $C_{-1} F$, which vanishes by Lemma~\ref{lemma: CF-FD}.

b) As in part a), we obtain an isomorphism between $\Si^{-2}\red{\Se}(S_X)$ and $S_{\tau(X)}$.
Since $C_0(S_X)$ and $C_1(S_X)$ vanish, the object $F_\rho(S_X)$ is the cofiber of
the morphism
\[
\begin{tikzcd}
\Si^{-2} D_2 \red{\Se} C_2 S_X \arrow{r} & S_{\tau(X)}.
\end{tikzcd}
\]
Since $D_2$ is right adjoint to $C_2$, the image of $D_2$ equals the right
orthogonal of the kernel of $C_2$. Now the functor $C_2$ is induced by
the restriction from $\ca$ to its full subcategory $\cb_1 \iso \proj(\bk Q)$
and its kernel is generated by the projectives $P_X$ associated
with the indecomposable objects $X$ of $\ca$ which do not lie
in $\cb_1$. These are precisely the objects in $\ca_m$.
Thus, the restriction of $\Si^{-2}D_2 \red{\Se}C_2 S_X$ to $\ca_m$ vanishes.
\end{proof}

Let $i$ be a vertex of $Q$ and $P_i$ the
associated indecomposable projective $\bk Q$-module. 
We write $P_i^\ca$ for the indecomposable projective
$\ca$-module $D_{-1}P_i$.
We write $\tau$ for the Auslander--Reiten translation 
of the derived category $\per(kQ)$.
Let $e_i\geq 0$ be the unique integer such that the object
$\tau^{-e_i} P_i$ of the derived category of $Q$ lies in $\Si \proj(\bk Q)$.
Thus, we have
\[
\tau^{-e_i} P_i=\Si P_{i^{*}}
\]
for a unique vertex $i^*$ of $Q$. 
Notice that the bijection $i\mapsto i^*$ is of order $1$ or $2$ depending
on the Dynkin type of $Q$. 
For $0\leq p\leq e_i$, we write $P_{\tau^{-p}(i)}^\ca$ for
the indecomposable projective $\ca$-module
$\ca(?,X)$, where $X$ is a minimal projective
resolution of the $kQ$-module $\tau^{-p}(P_i)$ for $p<e_i$ and $X=(P_{i^*} \to 0)$
for $p=e_i$.

By the isomorphism~(\ref{eq: FD_0}), the functor $F: \per(\ca) \to \per(\ca)$ takes the 
subcategory $\Im D_0$ to itself. We still denote by $F$ the functor induced by
$F$ in the quotient $\per(\ca)/\Im D_0$. The following lemma is inspired
by Prop.~8.6 in \cite{Wu23}.

\begin{lemma}\label{lemma: FpY} Let  $i$ be a vertex of $Q$.
\begin{itemize}
\item[a)] For each integer $0 \leq p \leq e_i$,  we have an isomorphism
\[
\begin{tikzcd}
F^{p} P_i^{\ca } \arrow{r}{_\sim} & P^\ca_ {\tau^{-p}(i)}
\end{tikzcd}
\]
in $\per(\ca) /\Im D_0$.
\item[b)] For each integer $q\geq 0$, we have an isomorphism
\[
\begin{tikzcd}
F^{e_i + q} D_{-1} P_i \arrow{r}{_\sim} & F^q D_1 P_{i^*}
\end{tikzcd}
\]
in $\per(\ca) /\Im D_0$.
\end{itemize}
\end{lemma}
\begin{proof} a) We proceed by induction on $0 \leq p \leq e_i$. Clearly the statement holds for $p=0$. Suppose
that we have proved it for some $p\geq 0$ such that $p<e_i$. We wish to show it for $p+1$. 
We write $S_j$ for the simple $\ca$-module associated with an indecomposable $j$ of $\ca$.
To prove that we have an isomorphism
\[
F^{p+1} P_i^{\ca }\simeq P^\ca_ {\tau^{-(p+1)}(i)}
\]
in the quotient category $\per(\ca)/\Im D_0$, it suffices to show that the space
\[
\Hom_{\per(\ca)}(F^{p+1}P^\ca_i, \Sigma^{-l} S_j)
\]
is one-dimensional if $P_j^\ca=P_{\tau^{-(p+1)}(i)}^\ca$ and $l=0$ and vanishes 
for all the other vertices $j$ not lying in $\cb_0$
and all integers $l\neq 0$. By the adjunction between $F$ and $F_\rho$ and the induction
hypothesis, we have
\[
\Hom(F^{p+1}P^\ca_i, \Sigma^{-l} S_j) = \Hom(F^{p}P^\ca_i, \Sigma^{-l} F_\rho S_j)
=\Hom(P^\ca_{\tau^{-p}(i)}, \Sigma^{-l} F_\rho S_j).
\]
Since the indecomposable $\tau^{-p}(P_i)$ belongs to $\ca_m$, it follows from Lemma~\ref{lemma: Frho} that
this space is isomorphic to
\[
\Hom(P^\ca_{\tau^{-p}(i)}, \Sigma^{-l} S_{\tau(j)}).
\]
Since the first argument is a projective $\ca$-module, 
this space vanishes for all $l\neq 0$. If we have $l=0$, 
it vanishes except if $\tau^{-p}(P_i)=\tau(P_j)$, in which case
it is one-dimensional. This proves the claim. 

b) By part a), the object $F^{e_i} D_{-1}P_i = F^{e_i} P_i^\ca$ is isomorphic to $P_{\tau^{-e_i}(i)} = D_1P_{i^*}$.
By the isomorphism (\ref{eq: FD_1}), this implies the claim for each $q\geq 0$. 
\end{proof}

\subsection{The cone functor and its use} \label{ss: The cone functor and its use} 
We keep the notations and assumptions of the preceding section.
We have a canonical morphism $C_1 \to C_0$ of functors from $\ca$ to $\proj(kQ)$. It induces a morphism
of functors at the level of complexes of modules and for a complex $X$ of $\ca$-modules, we define
the complex of $kQ$-modules $C(X)$ to be the cone over the morphism $C_1 X \to C_0 X$. Clearly, the cone
functor $X \mapsto C(X)$ induces a functor $\per(\ca)/\Im D_0 \to \per(kQ)$. 

\begin{lemma} \label{lemma: cone functor} The cone functor $C: \per(\ca)/\Im D_0 \to \per(kQ)$ is right
adjoint to the functor $D_{-1}: \per(kQ) \to \per(\ca)/\Im D_0$. Moreover, we have $C F \iso \tau^{-1} C$.
\end{lemma}

\begin{proof} Let $X$ and $Y$ be objects of $\per(\ca)$. We form a triangle
\[
\begin{tikzcd}
D_0 C_1 Y \arrow{r} & Y \arrow{r} & Y' \arrow{r} & \Si D_0 C_1 Y
\end{tikzcd}
\]
over the adjunction morphism $D_0 C_1 Y \to Y$. We have 
canonical bijections
\[
\Hom_{\per(\ca)/\Im D_0}(D_{-1} X, Y) \iso \Hom_{\per(\ca)}(D_{-1} X, Y') \iso \Hom_{\per(kQ)}(X, C_0 Y').
\]
since $Y'$ is right orthogonal to $\Im D_0$ and $D_{-1}$ is left adjoint to $C_0$. If we apply $C_0$ to
the triangle defining $Y'$, we obtain a triangle
\[
\begin{tikzcd}
C_1 Y \arrow{r} & C_0 Y \arrow{r} & C_0Y' \arrow{r} & \Si C_1 Y
\end{tikzcd}
\]
since $C_0 D_0$ is isomorphic to the identity. It follows that $C_0 Y'$ is canonically isomorphic to $C(Y)$.
The second statement follows from the fact that we have canonical isomorphisms $C_i F \iso \tau^{-1} C_i$
for $i=0,1$.
\end{proof}

\begin{proposition}\label{prop: Right adjoint to D1} Let us denote by $P$ an object of $\proj(\bk Q)$, 
by $Y$ an object of $\ca$ and by $p$ an integer $\geq 0$. We have a canonical isomorphism
\[
\begin{tikzcd}
\RHom_\ca(D_1 P, F^p Y^\we) \arrow{r}{_\sim} &  \tau_{\leq 0} \Si^{-1}\RHom_{\bk Q}(P, \tau^{-p} C(Y)).
\end{tikzcd}
\]
\end{proposition}

\begin{proof} Recall the functors $D_i : \proj(\bk Q) \to \ca$ for $-1\leq i \leq 1$. Explicitly, 
we have
\[
D_{-1} P = (0 \to P), \quad
D_0 P = (\id_P : P \to P), \quad
D_1 P = (P \to 0).
\]
Since these functors are fully
faithful, the chain of adjoints $D_{-1} \dashv  C_0 \dashv  D_0 \dashv  C_1 \dashv  D_1$ yields
canonical morphisms 
\[
\begin{tikzcd}
D_{-1} \arrow{r}{\phi} & D_0 \arrow{r}{\psi} &  D_1,
\end{tikzcd}
\]
which are easy to make explicit and whose composition vanishes. The same holds
for the induced functors $D_i : \cc^b(\proj \bk Q) \to \cc^b(\proj \ca)$.  Thus, for
$X$ in $\cc^b(\proj \bk Q)$, we obtain a canonical morphism $\alpha X: C(\phi X) \to D_1 X$
given by its components $\psi X$ and $0$ with respect to the decomposition
\[
C(\phi X) = (D_0 X)\oplus \Si (D_{-1} X)
\]
of $C(\phi X)$ as a graded $\ca$-module.  The canonical triangle
\[
\begin{tikzcd}
	D_{-1} X \arrow{r}{\phi X} &  D_{0} X \arrow{r}{} & C(\phi X)  \arrow{r}{\omega X}& \Si D_{-1} X. 
\end{tikzcd}
\]
shows that $\omega X$ becomes invertible in the Verdier quotient $\per(\ca)/\Im D_0$. Thus,
we have a canonical morphism
\[
\begin{tikzcd}
(\alpha X)(\omega X)^{-1}: \Si D_{-1} X \arrow{r} &  D_1 X
\end{tikzcd}
\]
in $\per(\ca)/\Im D_0$. For any object $Z$ of $\per(\ca)$,
this morphism induces a canonical morphism
\[
\begin{tikzcd}
\RHom_{\ca/ \Im D_0}(D_1 X, Z) \arrow{r} & \RHom_{\ca/ \Im D_0}(\Si D_{-1} X, Z).
\end{tikzcd}
\]
Using Lemma~\ref{lemma: cone functor}, we obtain a canonical morphism
\[
\begin{tikzcd}
\RHom_{\ca/ \Im D_0}(D_1 X, Z) \arrow{r} & \RHom_{kQ}(\Si  X, C(Z)).
\end{tikzcd}
\]
Now assume that $X=P^\we$ for the fixed projective $\bk Q$-module $P$.
Then $D_1 P^\we$ is left orthogonal to $\Im D_0$ so that the left hand side becomes
$\RHom_\ca(D_1 P^\we, Z)$. Now assume that $Z$ is connective. 
Then the complex $\RHom_\ca(D_1 P^\we, Z)$ is connective so that
we obtain an induced morphism
\[
\begin{tikzcd}
\RHom_\ca(D_1 P^\we, Z) \arrow{r} &  \tau_{\leq 0} \RHom_{kQ}(\Si P, C(Z)),
\end{tikzcd}
\]
which is functorial in $Z\in \per(\ca)/\Im D_0$. Our aim is to show that this morphism 
is an isomorphism if $Z=F^p Y^\we$. 
For this, it suffices to show that it induces isomorphisms
\begin{equation} \label{eq: key}
\begin{tikzcd}
\Hom_{\per(\ca)}(D_1 P^\we, \Si^{-n} F^p Y^\we) \arrow{r}{_\sim} & \Hom_{\per(kQ)}(\Si P,  \Si^{-n} \tau^{-p}C(Y))
\end{tikzcd}
\end{equation}
for all $n\geq 0$. We may and will assume that $Y$ is indecomposable in $\ca$. Then, by part a) of
Lemma~\ref{lemma: FpY}, in the 
quotient $\per(\ca)/\Im D_0$, the object $Y^\we$ is isomorphic to $F^q D_{-1} P_j$ for a vertex 
$j$ of $Q$ and an integer $q\geq 0$.  Let $e_j$ be the unique integer such that
$\tau^{-e_j} P_j$ lies in $\Si \proj(\bk Q)$. Let $j^*$ be the vertex of $Q$ such that $\tau^{-e_j} P_j$
is isomorphic to $\Si P_{j^*}$.  Let us first assume that $0 \leq p+q < e_j$.
We know from Lemma~\ref{lemma: FpY}, that for $p+q< e_j$, in $\per(\ca)/\Im D_0$,
the object $F^{p+q} D_{-1} P_j^\we$ is isomorphic to $P_{\tau^{-(p+q)}(j)}^\ca$. Thus,
this object and all of its shifts are right orthogonal to $D_1 P^\we$ so that
the left hand side vanishes. The right hand side is isomorphic to 
\[
\Hom_{\per(\bk Q)}(\Si P, \Si^{-n} \tau^{-p-q} P_j).
\]
Now for $n\geq 0$ and $0\leq p+q < e_j$, the object $\Si^{-n} \tau^{-p-q} P_j$ lies
in the right aisle $\cd_{\geq 0}$ of the canonical $t$-structure on $\per(kQ)$,
whereas $\Si P$ lies in the shifted left aisle. Hence this space also vanishes.

Let us now assume that $p+q=e_j +r$ for some $r\geq 0$. By part b) of
Lemma~\ref{lemma: FpY}, in $\per(\ca)/\Im D_0$, we have 
\[
F^p Y^\we \iso F^{p+q} D_{-1} P_j \iso F^r F^{e_j} D_{-1} P_j \iso F^r D_1 P_{j^*}. 
\]
Thus, the left hand side of (\ref{eq: key}) is isomorphic to
\[
\Hom(D_1 P^\we, \Si^{-n} F^r D_1 P_{j^*}) \iso \Hom(D_1 P^\we, \Si^{-n} D_1 \tau^{-r} P_{j^*}) \iso
\Hom(P, \Si^{-n} \tau^{-r} P_{j^*})
\]
and this is also isomorphic to the right hand side.
\end{proof}

\subsection{Proof of Theorem~\ref{theorem1} }\label{ss: Proof of Thm 1}
We abbreviate $\Pi=\Pi_2(\proj kQ)$ and $\Ga=\Pi_3(\ca,\cb)$. 
Let $X_0$ and $Y_0$ be objects of $\proj(\bk Q)$ and $X$ and $Y$
their images under the induction functor
$\per(\bk Q) \to \per(\Pi)$. We have 
\[
\begin{tikzcd}
\RHom_{\Ga}(G_i X, G_j Y) \arrow{r}{_\sim} &  \bigoplus_{p\ge 0} \RHom_{\ca}(D_iX_0, F^p D_j Y_0)
\end{tikzcd}
\]
and 
\[
\begin{tikzcd}
\RHom_{\Pi}(X, Y) \arrow{r}{_\sim} &  \bigoplus_{p\ge 0}\RHom_{\bk Q}(X_0, \tau^{-p}Y_0).
\end{tikzcd}
\]
For the case where $i=1$ and $j=-1$, it follows from Proposition~\ref{prop: Right adjoint to D1} that the space 
\[
\begin{tikzcd}
\RHom_\Ga(G_1 X, G_{-1} Y) \arrow{r}{_\sim} & \bigoplus_{p\ge 0}\RHom_{\ca}(D_1 X_0 , F^p D_{-1} Y_0)
\end{tikzcd}
\]
is isomorphic to the space
\[
 \bigoplus_{p\ge 0}\tau_{\leq 0} \RHom_{\bk Q}(\Si X_0, \tau^{-p}Y_0)= 
 \bigoplus_{p\ge 0} \tau_{\leq 0} \Si^{-1} \RHom_{\bk Q}(X_0, \tau^{-p}Y_0),
\]
and thus isomorphic to 
\[
\tau_{\le 0}\Si^{-1}\RHom_\Pi(X, Y).
\]
The other parts of  Theorem~\ref{theorem1}  follow from Lemma~\ref{upperhalfproof} in a similar
manner.

\subsection{Proof of Theorem~\ref{theorem2}} \label{ss: Proof of Thm 2} We keep the notations and assumptions
of the previous section. The first two isomorphisms in Theorem~\ref{theorem2} follow from 
Lemma~\ref{lemma: CF-FD}. Indeed, for $i=-1, 0$, we have
\begin{align*}
\RHom_\Ga(G_i X, Y) &= \bigoplus_{p\geq 0} \RHom_\ca(D_i X, F^p Y) \\
&= \bigoplus_{p\geq 0} \RHom_{kQ}(X, C_{i+1} F^p Y) \\
&= \bigoplus_{p\geq 0} \RHom_{kQ}(X, \tau^{-p} C_{i+1} Y) \\
&= \RHom_\Pi(X, C_{i+1} Y).
\end{align*}
Here, the third isomorphism holds thanks to Lemma~\ref{lemma: CF-FD}. 
The third isomorphism in Theorem~\ref{theorem2} follows from Proposition~\ref{prop: right adjoint for G_i}
below. The rest of this section is devoted to the proof of this proposition. 
We refer to section~\ref{ss: The cone functor and its use} for the definition of the 
cone functor $\per(\ca) \to \per(kQ)$.

\begin{lemma} \label{lemma: cone} Let $p\geq 0$ be an integer and $X$ and $Y$ objects of $\ca$. 
Then the cone functor induces isomorphisms
\begin{equation} \label{eq: C1}
\begin{tikzcd}
\RHom_{\ca / \Im D_0}(X^\we, F^p Y^\we) \arrow{r} & \tau_{\leq 0} 
\RHom_{\bk Q}(C(X), \tau^{-p}C(Y))
\end{tikzcd}
\end{equation} 
\end{lemma}

\begin{proof}
We first prove this for the case where $p=0$. 
Let $\ul{\ca}_{ex}$ denote the exact dg category obtained by endowing $\ca / \Im D_0$
with the exact structure induced from the triangulated category $\per(\bk Q)$ via the
canonical embedding of 
\[
H^0(\ca / \Im D_0) = \ch^{[-1,0]}(\proj(\bk Q))
\]
into $\per(\bk Q)$. By Theorem~A of \cite{Chen24b}, the quotient functor induces an isomorphism
\begin{equation} \label{eq: C2}
\begin{tikzcd}
({\ca / \Im D_0})(X, Y) \arrow{r} & \tau_{\leq 0} 
\RHom_{\cd^b_{dg}(\ul{\ca}_{ex})}(X, Y).
\end{tikzcd}
\end{equation}
Now it is not hard to see that the natural inclusion of exact dg categories
\[
\ul{\ca}_{ex} \subset \cd^b_{dg}(\mod(\bk Q))
\]
extends to an equivalence
\[
\cd^b_{dg}(\ul{\ca}_{ex}) \iso \cd^b_{dg}(\mod(\bk Q))
\]
induced by the cone functor. Thus, for $p=0$, the isomorphism (\ref{eq: C1})
follows from (\ref{eq: C2}). 

We now prove the isomorphism (\ref{eq: C1}) for $p>0$. Suppose that, in the notations
of Lemma~\ref{lemma: FpY}, we have $Y=P_{\tau^{-l}(j)}$ for some $0\leq p+l \leq e_j$. 
Then, by Lemma~\ref{lemma: FpY}, the object $F^p Y = F^p P_{\tau^{-l}(j)}$ is isomorphic
in $\per(\ca)/\Im D_0$ to $P_{\tau^{-l-p}(j)}$ and thus belongs to $\ca/\Im D_0$. So the
claim holds in this case by what we have just shown. Now suppose that 
 $Y=P_{\tau^{-l}(j)}$ and $p+l = e_j+q$ for some $q\geq 0$. By part b) of
 Lemma~\ref{lemma: FpY}, we have
 \[
F^{p} Y=F^{q}D_{1}P_{j^{*}}^{\bk Q}=D_{1}(\tau^{-q}P_{j^{*}}^{\bk Q}).
\]
Thus, we have 
\begin{align*}
\RHom_{\ca/\Im D_0}(X, F^{p} Y)&\cong  \tau_{\leq 0} \RHom_{\bk Q}(C(X), C(D_{1}(\tau^{-q}P_{j^{*}}^{\bk Q})))\\
&= \tau_{\leq 0} \RHom_{\bk Q}(C(X), C(\tau^{-q}P_{j^{*}}^{\bk Q}\to 0))\\
&\cong \tau_{\leq 0} \RHom_{\bk Q}(C(X), \tau^{-q}\Si P_{j^{*}}^{\bk Q})\\
&= \tau_{\leq 0} \RHom_{\bk Q}(C(X), \tau^{-q}\tau^{-e_{j}}P_{j}^{\bk Q})\\
&= \tau_{\leq 0} \RHom_{\bk Q}(C(X), \tau^{-p}\tau^{-l}P_{j}^{\bk Q})\\
&= \tau_{\leq 0} \RHom_{\bk Q}(C(X), \tau^{-p}CD_{-1}(\tau^{-l}P_{j}^{\bk Q}))\\
&= \tau_{\leq 0} \RHom_{\bk Q}(C(X), \tau^{-p}C(P_{\tau^{-l}(j)}))\\
&= \tau_{\leq 0} \RHom_{\bk Q}(C(X), \tau^{-p}C(Y)).\\
\end{align*}
\end{proof}

\begin{remark} \label{rk: toDminus}
Let $X \in \ca$ and $P\in \proj(\bk Q)$ be indecomposable and let $p\geq 0$ be an integer.
It follows from the above lemma that we have an isomorphism
\[
\begin{tikzcd}
\RHom_\ca(X^\we, F^p D_{-1} P^\we) \arrow{r}{_\sim}  & \RHom_{\bk Q}(C(X), \tau^{-p} P).
\end{tikzcd}
\]
Indeed, the object $F^p D_{-1} P^\we$ belongs to the right orthogonal of $\Im D_0$ so
that the left hand side is isomorphic to
\[
\RHom_{\ca/\Im D_0}(X^\we, F^p D_{-1} P^\we).
\]
Now the claim is clear since the cone of $F^p D_{-1} P^\we$ is $\tau^{-p} C(D_{-1} P^\we) = \tau^{-p} P$.
\end{remark}

\begin{proposition} \label{prop: right adjoint for G_i} For $-1 \leq i \leq 1$, the fully faithful functor
\[
\begin{tikzcd}
G_i: \per(\Pi_2(\proj kQ)) \arrow{r} & \per(\Pi_3(\ca,\cb))
\end{tikzcd} 
\]
admits a right adjoint $C_{i+1}$. Moreover, for an object $Y=(Y_1 \to Y_0)$ in $\ca$, we have
\[
C_0(Y)=Y_0\lten_{kQ} \Pi \ko 
C_1(Y)=Y_1\lten_{kQ} \Pi \ko
C_2(Y^\we)=\tau_{\leq 0}(\Si^{-1} C(Y) \lten_{kQ} \Pi)\ko
\]
where we write $?\lten_{kQ} \Pi$ for the induction from
$\per(kQ)$ to $\per(\Pi_2(\proj(kQ)))$.
\end{proposition}

\begin{remark} \label{remark: perfection of restriction}
It follows that the restriction functors $C_i: \cd(\Ga) \to \cd(\Pi)$, $-1 \leq i\leq 1$, take perfect 
objects to perfect objects.
\end{remark}

\begin{proof} Let us abbreviate $\Pi_2=\Pi_2(\proj kQ)$ and $\Ga=\Pi_3(\ca,\cb)$.
We need to show that for each $Z$ in $\per(\Ga)$, the functor
taking $X$ in $\Pi_2$ to $\Hom_{\per(\Ga)}(G_i X, Z)$ is representable.
It suffices to check this for a system of generators $Z$ of the triangulated
category $\per(\Ga)$. So we may assume that $Z=Y^\we$ for an indecomposable
object $Y=(Y_1 \to Y_0)$ of $\ca$. Let $P\in \proj(kQ)$. We have
\[
\RHom_\Ga(G_i P, Y^\we) = \bigoplus_{p\geq 0} \RHom_\ca(D_i P, F^p Y^\we).
\]
For $i=-1$ or $i=0$, by Lemma~\ref{lemma: CF-FD}, we have
\[
\begin{tikzcd}
\RHom_\ca(D_i P, F^p Y^\we) \arrow{r}{_\sim} & \RHom_{kQ}(P, C_{i+1} F^p Y^\we) \arrow{r}{_\sim} & \RHom_{kQ}(P, \tau^{-p} C_{i+1}(Y)) \ko
\end{tikzcd}
\]
which shows the first two isomorphisms of the claim. For $i=1$, 
by Proposition~\ref{prop: Right adjoint to D1}, the right hand side is isomorphic to
\[
\RHom_{kQ}(P, \tau_{\leq 0}(\bigoplus_{p\geq 0} \tau^{-p} \Si^{-1} C(Y))).
\]
The coproduct on the right hand side is taken over the $\tau^{-1}$-orbit
of the indecomposable object $\Si^{-1} C(Y)$. Clearly, the truncation via $\tau_{\leq 0}$
of the $\tau^{-1}$-orbit of an indecomposable object in the derived
category of a Dynkin quiver is still the $\tau^{-1}$-object of an
indecomposable. Thus, the right hand side is representable as
a functor of $P$. 
\end{proof}

\section{Morphism complexes in the relative Calabi--Yau completion}

\subsection{Linking \texorpdfstring{$\per(\ca)$}{} to \texorpdfstring{$\per(\cp_\ca)$}{}}
\label{ss: Linking per(A) to per(P_A)}
We use the notations and assumptions of section~\ref{ss: The morphism category}.
We denote by $\cp_\ca$ the full subcategory of the projective objects of the
exact category $\ca$. Thus, the objects of $\cp_\ca$ are direct sums of
morphisms $0 \to P$ and $\id_P : P \to P$, where $P$ belongs to $\proj(\bk Q)$. 
Notice that we have an equivalence
\[
\begin{tikzcd}
\proj((kA_2)^{op} \ten kQ) \arrow{r}{_\sim} & \cp_\ca
\end{tikzcd}
\]
taking $0 \to P$ to $P_0' \ten P$ and $\id_P: P \to P$ to $P_1' \ten P$, where
$P_i'$ denotes the projective left $kA_2$-module whose head is the simple
concentrated at the vertex $i$ for $0\leq i\leq 1$.
Let us denote the restriction along
the inclusion $\cp_\ca \subseteq \ca$ by 
\[
\begin{tikzcd}
R : \cd(\ca) \arrow{r} &  \cd(\cp_\ca)
\end{tikzcd}
\]
and its left adjoint by $L$. Then the adjunction morphism $LR \to \id_{\cd(\ca)}$ induces isomorphisms
$LRD_{-1} \iso D_{-1}$ and $LR D_0 \iso D_0$. 
Recall that the functor 
\[
\begin{tikzcd}
F: \per(\ca) \arrow{r} & \per(\ca)
\end{tikzcd}
\]
was defined in section~\ref{ss: easy isomorphisms}.

\begin{proposition} \label{prop: isomAP}
For any objects $X$ and $Y$ of $\ca$ and any integer $p\geq 0$, 
the functor $R$ induces an isomorphism
\[
\begin{tikzcd}
\RHom_\ca(X^\we, F^p Y^\we) \arrow{r} & \tau_{\leq 0} \RHom_{\cp_\ca}(R X^\we, R F^p Y^\we).
\end{tikzcd}
\]
\end{proposition} 

\begin{proof} For ease of notation, we will sometimes abbreviate $X^\we$ by $X$ and similarly
for $Y^\we$. We will prove the statement in three steps.

{\em First step: We prove that the morphism
\[
\begin{tikzcd}
\RHom_\ca(X^\we, F^p Y^\we) \arrow{r} & \RHom_{\cp_\ca}(R X^\we, R F^p Y^\we)
\end{tikzcd}
\]
is invertible when $X\in \ca$ belongs to the image of $D_0$.} Then it is of the
form $D_0 P$ for some $P\in \proj(\bk Q)$. Using the adjunction between $D_0$ and
$C_1$ as well as Lemma~\ref{lemma: CF-FD} we obtain isomorphisms
\[
\begin{tikzcd}
\RHom_\ca(D_0 P, F^p Y) = \RHom_{\bk Q} (P, C_1 F^p Y) = \RHom_{\bk Q}(P , \tau^{-p} C_1 Y).
\end{tikzcd}
\]
On the other hand, we have
\begin{align*}
 \RHom_{\cp_\ca}(R D_0 P, R F^p Y)&\cong \RHom_{\ca}(LRD_0 P, F^p Y) \\
 &\cong \RHom_{\ca}(D_0 P, F^p Y) \\
 &\cong \RHom_{\bk Q}(P, C_{1}F^{p}Y) \\
 &\cong \RHom_{\bk Q}(P, \tau^{-p}C_1 Y).
 \end{align*}
 
 {\em Second step: We prove the statement `modulo the image of $D_0$'}.
 The inclusion $\cp_\ca \to \ca$ induces a quasi-fully faithful dg functor
 \[
\begin{tikzcd}
\cp_\ca /  \Im D_0 \arrow{r} &	\ca / \Im D_0
\end{tikzcd}
\]
and $R$ induces the restriction
\[
\begin{tikzcd}
\per(\ca / \Im D_0) \arrow{r} & \per(\cp_\ca /  \Im D_0)
\end{tikzcd}
\]
along this functor. Let us show that it induces isomorphisms
\[
\begin{tikzcd}
\RHom_{\ca / \Im D_0}(X^\we, F^p Y^\we) \arrow{r} & \tau_{\leq 0} 
\RHom_{\cp_\ca /  \Im D_0}(R X^\we, R F^p Y^\we)
\end{tikzcd}
\] 
for all indecomposables $X$ and $Y$ of $\ca$ and all $p\geq 0$. 
The dg functor 
\[
D_{-1}: \proj(\bk Q) \to \cp_\ca /  \Im D_0
\] 
induces an equivalence
\[
\per(\bk Q) \iso \per( \cp_\ca /  \Im D_0)
\]
because $\per(\cp_\ca)$ has a semi-orthogonal
decomposition witnessed by the triangles
\[
\begin{tikzcd} 
	U_1 \arrow{r} & U_0 \arrow{r} &  C(f) \\
	U_1 \arrow{r}{1_{U_1}} \arrow[u, "{1_{U_1}}" ]& U_1 \arrow{u}{f} \arrow{r} &   0\arrow{u}
\end{tikzcd}
\]
for $U\in \per(\cp_\ca)$. Clearly, the quasi-inverse takes $f: U_1 \to U_0$ to $C(f)$. 
By Lemma~\ref{lemma: cone}, the cone functor induces isomorphisms
\begin{equation} 
\begin{tikzcd}
\RHom_{\ca / \Im D_0}(X^\we, F^p Y^\we) \arrow{r} & \tau_{\leq 0} 
\RHom_{\bk Q}(C(X), \tau^{-p}C(Y)).
\end{tikzcd}
\end{equation} 
This yields the claim.

{\em Third step: We prove that the first and the second step together imply the claim.}
We still denote by $D_{i}$ and $C_{i}$ the functors between the categories 
$\per (\bk Q)$ and $\per (\ca)$ (resp. $\per(\cp_{\ca})$). 
We have the following recollement
\begin{equation} \label{recollement1}
\begin{tikzcd}
\per (\bk Q) \arrow{rr}[description]{D_0} &&\per (\ca)
\arrow[yshift=-1.5ex]{ll}{C_1}
\arrow[yshift=1.5ex]{ll}[swap]{C_0}
\arrow{rr}[description]{p} &&\per (\ca)/ \rm{Im} D_0\, 
\arrow[yshift=-1.5ex]{ll}{p_\rho}
\arrow[yshift=1.5ex]{ll}[swap]{p_\lambda}.
\end{tikzcd}
\end{equation}
Recall that $X$ is an object of $\ca$. By abuse of notation, we also write $X$ for
the associated representable dg module $X^\we$ in $\per(\ca)$. 
We have the functorial triangle
\[
 p_\lambda pX \longrightarrow X \longrightarrow D_0 C_0 X \longrightarrow \Si p_\lambda p X.
 \]
 Let us abbreviate $X'=p_\lambda pX$, $X''= D_0 C_0 X$ and $Z=F^p Y^{\we}$. 
 For objects $U$ and $V$ of $\per(\ca)$, we abbreviate
 \[
 _\ca(U,V) = \RHom_\ca(U,V) \quad\mbox{and}\quad
 (RU, RV) = \RHom_{\cp_\ca}(RU, RV). 
 \]
 With these notations, the above triangle yields the
 following commutative diagram, whose rows give rise to long exact sequences in homology
 \[
 \begin{tikzcd}[column sep = small]
 _\ca(\Si X', Z) \arrow{d}{\psi_1} \arrow{r} & 
 _\ca(X'',Z) \arrow{d}{\psi_2} \arrow{r} & 
 _\ca(X,Z)  \arrow{d}{\psi_3} \arrow{r} &
 _\ca(X',Z) \arrow{d}{\psi_4} \arrow{r} &
 _\ca(\Si^{-1} X'', Z) \arrow{d}{\psi_5} \\
 (\Si RX', RZ) \arrow{r} & 
 (RX'', RZ) \arrow{r} &
 (RX,RZ) \arrow{r} &
 (RX', RZ) \arrow{r} &
 (\Si^{-1} RX'',RZ).
 \end{tikzcd}
 \]
 Since $X'$ is left orthogonal to the image of $D_0$, we have canonical isomorphisms
 \[
 _\ca(X',Z) \iso \RHom_{\ca/ \Im D_0}(X,Z) \quad\mbox{and}\quad
 (RX', RZ) \iso \RHom_{\cp_\ca/ \Im D_0}(RX, RZ).
 \]
 Thus, by the second step, the map $H^n(\psi_4)$ is an isomorphism for
 all $n\leq 0$ and similarly for $H^n(\psi_1)$. Since $X''$ belongs to the
 image of $D_0$, the maps $H^n(\psi_2)$ and $H^n(\psi_5)$ are
 isomorphisms for all integers $n\in \Z$. By the five-lemma, we obtain
 that $H^n(\psi_3)$ is an isomorphism for all $n\leq 0$, which shows
 the claim. 
 \end{proof}
  
 \begin{corollary} \label{cor: Hom in per(A)} Let $X$ and $Y$ be objects of $\ca$ and $n$ an integer.
 There is an integer $M$ such that the space
 \[
 \Hom_{\cd(\ca)}(X^\we, \Si^n F^p Y^\we)
 \]
 vanishes for all $p>M$. 
 \end{corollary}
 
 \begin{proof} By the above Proposition~\ref{prop: isomAP}, it suffices to prove the
 corresponding claim for the space
 \[
 \Hom_{\cd(\cp_\ca)}(RX^\we, \Si^n RF^pY^\we).
 \]
 Now the object $RX^\we$ belongs to $\Mod(\cp_\ca)$ and $\cp_\ca$ is
 of global dimension $\leq 2$. So it suffices to show that the object
 $\Si^n RF^p Y^\we$ is concentrated in cohomological degrees $\leq -3$ for
 all sufficiently large $p$. Now we have isomorphisms
 \[
 C_i(RF^p Y^\we) \iso \tau^{-p} C_i Y 
 \]
 in $\cd(\bk Q)$ for $i=0,1$. Since $Q$ is a Dynkin quiver and
 the $C_i Y$ lie in $\Mod \bk Q$, their iterated negative AR-translates
 $\tau^{-p} C_i Y$ lie in the shifted aisle $\cd^{\leq -4}(\bk Q)$ for
 all $p$ greater than or equal to twice the Coxeter number $h$
 of the Dynkin diagram underlying $Q$. Thus, it suffices to
 choose $M$ to be $h(3-n)$. 
 \end{proof}

 \subsection{The Higgs category associated with \texorpdfstring{$\cb\subset\ca$}{}} 
 \label{ss: The Higgs category}
 We keep the notations and assumptions of the previous section.
  Recall that for objects $X$ and $Y$ of $\ca$, we have a 
 canonical isomorphism
 \[
 \Pi_3(\ca, \cb)(X, Y) \iso \bigoplus_{p\geq 0} \RHom_\ca(X^\we, F^p Y^\we).
 \]
 By the above Corollary~\ref{cor: Hom in per(A)}, the spaces $H^n(\Pi_3(\ca,\cb)(X,Y))$ are
 finite-dimensional for all integers $n$. Since $Q$ is a Dynkin quiver, the space
 $H^n(\Pi_2(\cb)(P, P'))$ is also finite-dimensional for
 all $P$, $P'$ in $\cb$ and all integers $n$. Thus, the Ginzburg morphism of the
 relative $3$-Calabi--Yau completion
 \[
 \Pi_2(\cb) \to \Pi_3(\ca,\cb)
 \]
 is Morita-equivalent to a morphism between smooth connective
 dg algebras with finite-dimensional homologies.
 Let us put $\Ga=\Pi_3(\ca,\cb)$. Then $\Ga$ is smooth 
 and its perfect derived category $\per(\Ga)$ is
 Hom-finite. We write $\cc=\cc^{rel}_\Ga$ for the
 corresponding relative cluster category and
 $\ch=\ch_\Ga$ for the corresponding Higgs
 category in the sense of \cite{Wu23}. It 
 contains the canonical cluster-tilting object
 $T$ obtained as the direct sum of the images $T_X$ of
 the $X^\we$, where $X$ ranges through the
 indecomposable objects of $\ca$. By Lemma~3.30 
 of \cite{Chen24b}, we have canonical
 isomorphisms
 \[
 \Ga(X^\we, Y^\we) \iso \tau_{\leq 0} \RHom_{\cc}(T_X, T_Y).
 \]
 The projective-injective objects of the Higgs category
 $\ch$ are the direct sums of objects $T_X$, where
 $X$ is projective or injective in $\ca$. We denote by
 $\cp\subset\ch$ the full subcategory of the projective-injective
 objects in $\ch$ and by $\cR\subset\cp$ the subcategory
 of the objects $T_X$, where $X$ is projective in $\ca$. 
 The quotient functor $\per(\Ga) \to \cc$ induces an
 equivalence $H^0(\cp_{dg}) \iso \cp$, where
 $\cp_{dg}$ is the boundary dg category defined
 in section~\ref{ss: Relative 3-CY-completion}. 
 We often identify $H^0(\cp_{dg})$ and $\cp$
 using this equivalence. We write $\cp_i$ for the
 full subcategory of $\cp$ given by the image
 of $G_i$, $-1\leq i \leq 1$, and we write
 $\cp_{i,dg}$ for its canonical dg enhancement.
 
 Recall that $\cp_\ca \subset \ca$ denotes the
 full subcategory of the projective objects of $\ca$ and that $\cb\subset\ca$ 
 is the additive subcategory generated by the projective and the injective objects
 of $\ca$. Thus, the canonical functor $X \mapsto T_X$
 induces essentially surjective functors $\cp_\ca \to \cR$,
 $\cb \to \cp$ and $\ca \to \add(T)$. We write $\cR_{dg}$
 for the dg enhancement of $\cR$ induced by that of $\cc$.
 We sum up the notations in the diagram
  \[
 \begin{tikzcd}
 \cp_\ca \arrow{d} \arrow{r} & \cb \arrow{d} \arrow{r} &  \ca \arrow{d}          &                       & \\
 \cR \arrow{r}                       & \cp \arrow{r}                &  \add(T)  \arrow{r}   & \ch \arrow{r}  & \cc.
 \end{tikzcd}
 \]
 
 \begin{corollary} \label{cor: Perfection of restrictions} For any object $U$ of the relative
 cluster category $\cc$ and any $-1\leq i\leq 1$, the restriction of $U$ to $\cp_{i,dg}$ is perfect. 
 \end{corollary}
 
 \begin{proof} Since $\cc$ is a Verdier quotient of $\per(\Ga)$, it is generated as a triangulated
 category by the $T_X$, where $X$ is indecomposable in $\ca$. The restriction of
 $\RHom_\Ga(?,T_X)$ to $\cp_{i,dg}$ is perfect by Proposition~\ref{prop: right adjoint for G_i}.
 \end{proof}

\subsection{The categories \texorpdfstring{$\cp_\ca$}{} and \texorpdfstring{$\cR$}{}} \label{ss: The categories P_A and R}
Recall from section~\ref{ss: Linking per(A) to per(P_A)} that we have the canonical equivalence
\[
\begin{tikzcd}
\proj((\bk A_2)^{op} \ten (\bk Q)) \arrow{r}{_\sim} &  \cp_\ca
\end{tikzcd}
\]
taking $P'_0 \ten P$ to $0 \to P$ and $P'_1 \ten P$ to $\id_P : P \to P$.
In particular, the category $\cp_\ca$ is Morita equivalent to a
finite-dimensional algebra of global dimension $2$. 
Via the above equivalence, we will often identify the
two categories. In particular, we will identify
a $\cp_\ca$-module $M$ with a morphism $M_1 \to M_0$  of
$kQ$-modules.  

We denote by $R: \per(\ca) \to \per(\cp_\ca)$ the restriction along the
inclusion $\cp_\ca \to \ca$. It takes an object $M$ of $\per(\ca)$ to
the object of the perfect derived category of $\cp_\ca$ given by
the morphism $C_1 M \to C_0 M$ of complexes of $kQ$-modules.
Thus, the functor $R$ is a localization functor whose
kernel is the intersection of the kernels of $C_1$ and $C_0$.
Therefore, by Lemma~\ref{lemma: CF-FD}, the kernel
of $R$ is stable under $F$ so that $F$ induces a functor
$\per(\cp_\ca) \to \per(\cp_\ca)$, which we still denote by $F$.

\begin{proposition}  \label{prop: RA}
\begin{itemize}
\item[a)] We have the following square which is commutative
up to isomorphism
\[
\begin{tikzcd} 
\per((\bk A_2)^{op} \ten (\bk Q)) \arrow{r}{_\sim} \arrow[d, "\id\ten \tau^{-1}"'] &  \per(\cp_\ca) \arrow{d}{F} \\
\per((\bk A_2)^{op} \ten (\bk Q)) \arrow{r}{_\sim} &  \per(\cp_\ca) ,
\end{tikzcd}
\]
where we denote by $\id\ten \tau^{-1}$ the derived functor of tensoring
with the bimodule $(kA_2)^{op} \ten \Si \Theta$ and $\Theta$ is the inverse
dualizing bimodule $\RHom_{(kQ)^e}(kQ, (kQ)^e)$.
\item[b)] The functor $\cp_\ca \to \cR_{dg}$ is
essentially surjective and induces a canonical isomorphism
\[
\bigoplus_{p\geq 0} \RHom_{\cp_\ca}(U, (\id\ten \tau^{-1})^p V) \iso \RHom_{\cR_{dg}}(U,V)
\]
for all objects $U$ and $V$ of $\per(\cp_\ca)$.
 \end{itemize}
 \end{proposition}

 \begin{proof}  a) This follows from Lemma~\ref{lemma: CF-FD}. 
 
b)  This follows from Proposition~\ref{prop: isomAP} by choosing for
 $X$ and $Y$ projective objects of $\ca$.
 \end{proof}
 
 \begin{remark} Part b) of the proposition shows that the canonical
 functor $\cp_\ca \to \cR_{dg}$ induces an equivalence
 \[
 \begin{tikzcd}
\per(\cp_\ca)/_{ll} \,(\id\ten \tau^{-1})^\N  \arrow{r}{_\sim} & \per(\cR_{dg}) \ko
 \end{tikzcd}
 \]
 where the term on the left denotes the left lax quotient in
 the sense of \cite{FanKellerQiu24}.
 \end{remark}

\begin{corollary} \label{cor: R as tensor product}
The functor $\cp_\ca \to \cR_{dg}$ induces an equivalence
\[
\begin{tikzcd}
\add((\bk A_2)^{op} \ten \Pi_2(\bk Q)) \arrow{r} &  \cR ,
\end{tikzcd}
\]
where the category on the left is viewed as a full subcategory of $\rep(\bk A_2, \Pi_2(\bk Q))$.
 \end{corollary}
 
 \begin{proof} This follows from part b) of Proposition~\ref{prop: RA}.
 \end{proof}

 We notice that the restriction $R: \per(\Ga) \to \per(\cR_{dg})$ vanishes on
 the simples $S_X$ associated with indecomposables $X$ in $\ca$ which
 are neither projective nor injective. Thus it induces a functor
 $\cc \to \per(\cR_{dg})$. Let us sum up the notations in the following diagram
 \[
 \begin{tikzcd}
 \cp_\ca \arrow{d} \arrow{r} & \ca \arrow{d} \arrow{r} &  \per(\ca) \arrow{d} \arrow{r}{R}  & \per(\cp_\ca) \arrow{d} \\
 \cR \arrow{r}                       & \add(T)  \arrow{r}        & \cc \arrow{r}{R}                               & \per(\cR_{dg}).
 \end{tikzcd}
 \]

\begin{corollary}  \label{cor: R for T}
For any objects $X$ and $Y$ of $\ca$, the functor $R$ induces a quasi-isomorphism 
\[
\begin{tikzcd}
\tau_{\leq 0} \RHom_{\cc}(T_X, T_Y) \arrow{r} & \tau_{\leq 0} \RHom_{\cR_{dg}}(R X^\we, R Y^\we).
\end{tikzcd}
\]
\end{corollary}
\begin{proof} We may and will assume that $X$ and $Y$ are indecomposable.
By Lemma~3.30 of \cite{Chen24b}, we have a canonical isomorphism
\[
\begin{tikzcd}
\RHom_\Ga(X,Y) \arrow{r}{_\sim} & \tau_{\leq 0} \RHom_{\cc}(T_X, T_Y).
\end{tikzcd}
\]
By definition, the complex $\RHom_\Ga(X,Y)$ is equal to the left hand side
of the following isomorphism obtained from Proposition~\ref{prop: isomAP}
\[
\begin{tikzcd}
\bigoplus_{p\geq 0} \RHom_\ca(X^\we, F^p Y^\we) \arrow{r}{_\sim} &
\tau_{\leq 0} \left( \bigoplus_{p\geq 0} \RHom_{\cp_\ca}(RX^\we, RF^pY^\we) \right) .
\end{tikzcd}
\]
By Proposition~\ref{prop: RA}, part a), we have $RF^p Y^\we \iso (\id\ten\tau^{-1})^p(RY^\we)$
so that the claim follows from part b) of that Proposition.
\end{proof}

\subsection{The Higgs category versus the cosingularity category}
\label{ss: Higgs category and Cosingularity category}
Recall the functor $R: \cc \to \per(\cR_{dg})$ from section~\ref{ss: The categories P_A and R}. Clearly the 
composition
\[
\begin{tikzcd}
\per(\Ga) \arrow{r} & \cc \arrow{r}{R} & \per(\cR_{dg}) \arrow{r} & \cosg(\cR_{dg}) 
\end{tikzcd}
\]
vanishes on $\pvd(\Ga)$ so that the functor $R$
induces a functor
\[
\begin{tikzcd}
\cosg(\Ga) \arrow{r}{R} & \cosg(\cR_{dg}).
\end{tikzcd}
\]
\begin{proposition}[Christ \cite{Christ25b}] \label{prop: Christ} 
The functor $R: \cosg(\Ga) \to \cosg(\cR_{dg})$  is an equivalence.
\end{proposition}

\begin{proof} Recall from section~\ref{ss: Relative 3-CY-completion} that $\cp_{dg}$
denotes the boundary dg category. We may view $\cR_{dg}$ as a full subcategory
of $\cp_{dg}$ so that we obtain a restriction functor
\[
\begin{tikzcd}
R: \per(\Ga) \arrow{r} & \per(\cR_{dg}).
\end{tikzcd}
\]
The inclusion of $\cR_{dg}$ into $\per_{dg}(\Ga)$ yields a fully faithful left
adjoint $L$ so that $RL$ is isomorphic to the identity. For an object $X$ of
$\per(\Ga)$, let us consider the triangle
\[
\begin{tikzcd}
LRX \arrow{r} & X \arrow{r} & X' \arrow{r} & \Si LR X.
\end{tikzcd}
\]
Then $X'$ is perfect and belongs to the subcategory $\cv$ formed by
the objects right orthogonal to the $X^\we$, $X\in \cR$. Now this
subcategory is equivalent to the perfect derived category of
$\Ga'=\Pi_3(\ca', \cb')$, where $\ca'$ is the (dg) quotient of
$\ca$ by all projective objects and similarly for $\cb'$.
By Wu's theorem (Prop.~8.14 and Example~8.19 of \cite{Wu23}), the dg algebra $\Pi_3(\ca', \cb')$
is concentrated in degree $0$, finite-dimensional and of
global dimension $3$. Thus, it is smooth and proper and
\red{$\per(\Ga')$} coincides with $\pvd(\Ga')$. We conclude
that the object $X'$ lies in the kernel of the functor
$\per(\Ga) \to \cosg(\Ga)$. It follows that the adjoint
pair $(L,R)$ induces quasi-inverse equivalences
between $\cosg(\Ga)$ and $\cosg(\cR)$.
\end{proof}

Recall that the {\em $1$-cluster category} of the Dynkin quiver $Q$ is defined
as the orbit category
\[
\cc_{Q}^{(1)} = \per(kQ)/\tau^{\Z} \ko
\]
cf.~\cite{Keller05}. It is canonically equivalent to the cosingularity category
$\cosg(\Pi_2)$. We write $(\cc_{Q}^{(1)} )_{dg}$ for the canonical dg
enhancement of the $1$-cluster category.

\begin{lemma}\label{lemma: The Cosingularity category  and  1-cluster category}
	We have a canonical equivalence
\[
\begin{tikzcd}
\cosg_{dg}(\cR_{dg}) \arrow{r}{_\sim} & \rep_{dg}(\bk A_2,  (\cc_{Q}^{(1)})_{dg}).
\end{tikzcd}
\]
\end{lemma}
\begin{proof} By Cor.~\ref{cor: R as tensor product}, we have an equivalence
\[
kA_2^{op} \ten \Pi_2(kQ) \iso \cR_{dg}.
\]
By Lemma~\ref{lemma: cosg of tensor product}, we have a canonical isomorphism
in the Morita homotopy category $\Hmo$
\[
kA_2^{op} \ten \cosg_{dg}(\Pi_2(kQ)) \iso \cosg_{dg}(kA_2^{op} \ten \Pi_2(kQ)).
\]
By Lemma~\ref{lemma: inner Hom}, we also have a canonical isomorphism in $\Hmo$
\[
kA_2^{op} \ten \cosg_{dg}(\Pi_2(kQ)) \iso \rep_{dg}(kA_2, \cosg_{dg}(\Pi_2(kQ))).
\]
\end{proof}

\begin{corollary} \label{cor: cosg(Ga) and the dg cat of triangles} We have a canonical equivalence
\[
\begin{tikzcd}
\cosg_{dg}(\Ga) \arrow{r}{_\sim} & \rep_{dg}(\bk A_2,  (\cc_{Q}^{(1)})_{dg}).
\end{tikzcd}
\]
\end{corollary}

\begin{proof} This is immediate from Proposition~\ref{prop: Christ} and Lemma~\ref{lemma: The Cosingularity category  and  1-cluster category}.
\end{proof}

\begin{conjecture}[Christ \cite{Christ24}] \label{conj: Christ}
The restriction $\Phi$ of the quotient functor 
\[
\cc \to \cosg(\Ga)
\]
to the subcategory $\ch$ is an equivalence of $k$-linear categories $\ch \to \cosg(\Ga)$. Moreover, it induces
bijections
\[
\begin{tikzcd}
\Hom_\cc(X, \Si^{-n} Y) \arrow{r}{_\sim} & \Hom_{\cosg(\Ga)}(X, \Si^{-n} Y)
\end{tikzcd}
\]
for all objects $X$ and $Y$ of $\ch$ and all integers $n\geq 0$.
\end{conjecture}

\begin{remark} For a quiver $Q$ of type $A_1$, the conjecture is easy to check directly.
In~\cite{Christ22a}, Christ proved that it also holds more generally for the relative
Ginzburg algebras associated with triangulations of marked surfaces.
\end{remark}

\begin{theorem} \label{conj true}
Conjecture~\ref{conj: Christ} holds true.
\end{theorem}

\begin{proof}
We prove the conjecture in the following sections~\ref{ss: Essential surjectivity of Phi}  and
\ref{ss: Full faithfulness of Phi} using the results of 
sections~\ref{s: The Higgs category via Gorenstein projective dg modules} and 
\ref{s: Comparison with the cosingularity category} below.
\end{proof}

\subsection{Essential surjectivity of \texorpdfstring{$\Phi$}{}}
\label{ss: Essential surjectivity of Phi} In this section, we will show that the functor
$\Phi: \ch \to \cosg(\Ga)$ is essentially surjective. Thanks to the equivalence
\[
\begin{tikzcd}
\cosg(\Ga) \arrow{r}{_\sim} & \rep_{dg}(\bk A_2,  (\cc_{Q}^{(1)})_{dg}).
\end{tikzcd}
\]
of Cor.~\ref{cor: cosg(Ga) and the dg cat of triangles}, we see that the functor
taking an object $X$ of $\cosg(\Ga)$ to the morphism
\[
\begin{tikzcd}
C_1(X) \arrow{r}{\phi(X)} & C_0(X)
\end{tikzcd} 
\]
of the $1$-cluster category $\cc_{Q}^{(1)}$ is an epivalence, cf.~section~\ref{ss: Representations up to homotopy}.
Now the functor $H^0: \cc_{Q}^{(1)} \to \proj(\La)$, \red{where $\La=\La_Q=H^0(\Pi_2(kQ))$ is the classical preprojective
algebra of $Q$}, is an equivalence of $k$-linear categories
and so the functor taking $X$ to 
\[
\begin{tikzcd}
H^0 C_1(X) \arrow{r}{H^0 \phi(X)} & H^0 C_0(X)
\end{tikzcd} 
\]
is an epivalence from $\cosg(\Ga)$ to the category of morphisms
of the category $\proj(\La)$. So let $u: U_1 \to U_0$ be a morphism
between finitely generated projective
$\La$-modules. We would like to show that it is isomorphic
to a morphism of the above form for some object $X$ of $\ch$.
This is easy if $u$ is an isomorphism or if we have
$U_1=0$ or $U_0=0$. It is also clear that it suffices to
show this if $u$ is indecomposable as a morphism. 
We may thus assume that $u$ is radical and does
not have direct factors of the form $P \to 0$ or
$0 \to P$. In this case, the morphism $u$ is
the minimal projective presentation of its
cokernel $\cok(u)$ and is therefore determined
by $\cok(u)$ up to isomorphism in the category
of morphisms. As shown in \cite{GeissLeclercSchroeer07c},
the category $\mod(\La)$ contains a canonical
cluster-tilting object $T^\La$ associated with $Q$, namely
the direct sum of the cokernels of all morphisms
\[
\begin{tikzcd}
P_1 \ten_{kQ} \La \arrow{r}{f\ten \La} &  P_0\ten_{kQ} \La \ko
\end{tikzcd}
\]
where $f: P_1 \to P_0$ ranges through the minimal
projective resolutions of the indecomposable
$kQ$-modules (up to isomorphism). It follows that there
is a short exact sequence of $\La$-modules
\begin{equation} \label{eq: conflation for cok(u)}
\begin{tikzcd}
0 \arrow{r} & T^\La_1 \arrow{r}{g} & T^\La_0 \arrow{r} & \cok(u) \arrow{r} & 0 \ko
\end{tikzcd}
\end{equation}
where the $T_i^\La$ belong to $\add(T^\La)$. If we apply
Proposition~\ref{prop: Relative CY-completions and localizations} 
to the additive subcategory $\cn = \cb_0 \amalg \cb_1$ of $\cb$ (cf.~section~\ref{ss: Relative 3-CY-completion})
and use Theorem 6.2 of \cite{Wu23},  we see that we have an equivalence of extriangulated categories
\[
\begin{tikzcd}
\ch/(\cp_0, \cp_1) \arrow{r}{_\sim} & \mod (\La)
\end{tikzcd}
\]
from the quotient of $\ch$ by the ideal generated by the identities
of the objects in $\cp_0$ and $\cp_1$ to the category $\mod(\La)$
which sends the canonical cluster-tilting object $T$ of $\ch$ to $T^\La$. 
Since $\ch$ is a Frobenius extriangulated category and $\cp_0$ and
$\cp_1$ consist of projective-injective objects, we can lift the
conflation~\ref{eq: conflation for cok(u)} to a conflation
\[
\begin{tikzcd} 
T_1 \arrow{r}{h} & T_0 \arrow{r} & X'
\end{tikzcd}
\]
of $\ch$. We claim that $X'$ has a direct factor $X$ whose image in
$\cosg(\Ga)$ yields a morphism  $H^0 C_1(X) \to H^0 C_0 (X)$
which  is isomorphic to $u$. Indeed, if we apply the functor $H^0$ to 
the morphism of triangles
\[
\begin{tikzcd}
C_1 T_1 \arrow{r} \arrow{d} & C_1 T_0 \arrow{r} \arrow{d} & C_1 X' \arrow{r} \arrow{d} & \Si C_1 T_1 \arrow{d} \\
C_0 T_1 \arrow{r}                & C_0 T_0 \arrow{r}                 & C_0 X' \arrow{r}                & \Si C_0 T_1
\end{tikzcd}
\]
we obtain the following commutative diagram (since all objects are connective), whose
third row is isomorphic to the short exact sequence~\ref{eq: conflation for cok(u)} so that
$\cok(H^0(\phi X'))$ is isomorphic to $\cok(u)$. 
\[
\begin{tikzcd}
H^0 C_1 T_1 \arrow{d} \arrow{r} & H^0 C_1 T_0 \arrow{d} \arrow{r} & H^0 C_1 X' \arrow{d} \arrow{r} & 0 \\
H^0 C_0 T_1 \arrow{d} \arrow{r} & H^0 C_0 T_0 \arrow{d} \arrow{r} & H^0 C_0 X' \arrow{d} \arrow{r} & 0 \\
\cok(H^0(\phi T_1)) \arrow{d} \arrow{r} & \cok(H^0(\phi T_0))  \arrow{d} \arrow{r} & \cok(H^0(\phi X')) \arrow{d} \arrow{r} & 0 \\
0                                                           &   0                                                           & 0 
\end{tikzcd}
\]
It follows that $H^0 C_1 X' \to H^0 C_0 X'$ is the direct sum of $u$ and a split
epimorphism $p$. It is clear that $p$ lifts to a direct factor of $X'$ in $\ch$ 
and that the quotient $X$ of $X'$ by this direct factor lifts $u$.

\subsection{Full faithfulness of \texorpdfstring{$\Phi$}{}} \label{ss: Full faithfulness of Phi}
Let us prove that $\Phi$ is fully faithful
using the results of sections~\ref{s: The Higgs category via Gorenstein projective dg modules} and 
\ref{s: Comparison with the cosingularity category} below. Consider the diagram
\[
\begin{tikzcd}
\ch \arrow[d,"\Phi"'] \arrow{r} & \cd(\cp_{dg}) \arrow{r} & \Cosg(\cp_{dg}) \\
\cosg(\Ga) \arrow{r}{_\sim} & \cosg(\cR_{dg}) \arrow{r} & \Cosg(\cR_{dg}), \arrow[u,"\Psi"']
\end{tikzcd}
\]
where $\Cosg(\cp_{dg})$ is the quotient of the unbounded derived category $\cd(\cp_{dg})$ 
by its localizing subcategory generated by $\pvd(\cp_{dg})$ and similarly for $\Cosg(\cR_{dg})$.
The composition of the two top horizontal functors is fully faithful by Prop.~\ref{prop: from H to D(P)/N}
below.  The functor $\Psi$ is an equivalence by Lemma~\ref{lemma: cosg(R) vs D(P)/N} below.
It follows that $\Phi$ is fully faithful. By the same argument, $\Phi$ induces
isomorphisms in the functors $\tau_{\leq 0}\RHom$.

\section{The Higgs category via Gorenstein projective dg modules}
\label{s: The Higgs category via Gorenstein projective dg modules}

\subsection{Projective domination} \label{ss: Projective domination}
Let $\ce$ be an exact dg category in the sense of \cite{Chen24b}.
We assume that $\ce$ is connective.
Now suppose that $\ce$ has enough projectives
(i.e.~the extriangulated \cite{NakaokaPalu19} category
$H^0(\ce)$ has enough projectives) and let $\cp\subseteq \ce$ be the
full dg subcategory on the projective objects. We have a canonical
exact functor
\[
\begin{tikzcd}
\ce \arrow{r} & \cd_{dg}(\cp)
\end{tikzcd}
\]
taking an object $X$ of $\ce$ to the restriction of the functor
$\ce(?,X)$ to the subcategory $\cp\subseteq\ce$. Loosely speaking,
this functor takes each object $X$ of $\ce$ to its projective resolution. Since the
target of this functor is pretriangulated, by the universal property
of the dg derived category of $\ce$ proved in \cite{Chen24b},
the functor extends to an exact dg functor
\[
\begin{tikzcd}
\can: \cd^b_{dg}(\ce) \arrow{r} & \cd_{dg}(\cp).
\end{tikzcd}
\]
The exact dg category $\ce$ is {\em projectively dominated}\footnote{We thank Merlin Christ for
suggesting the terminology.} in the sense 
of \cite{Ding25} if this functor is quasi-fully faithful. For example, an exact
dg category concentrated in degree $0$ (i.e.~an exact category in
the sense of Quillen) is projectively dominated if and only if it
has enough projectives. On the other hand, if $H^0(\ce)$ is
triangulated, then $\ce$ has enough projectives and all
projectives are contractible (i.e.~become zero objects in
$H^0(\ce)$). So in this case, the dg category $\cd_{dg}(\cp)$ is
quasi-equivalent to the zero category and $\ce$
cannot be projectively dominated unless it is itself
quasi-equivalent to the zero category. 

\begin{lemma}[\cite{Ding25}]  \label{lemma: from projective domination}
Suppose that $\ce$ is projectively dominated.
\begin{itemize}
\item[a)] The image of $\ce$ in $\cd(\cp)$ is formed by connective pseudocoherent dg modules.
\item[b)] If $P$ is projective-injective, then for each object $X$ of $\ce$, we have
\[
\Ext^n_{\cp}(\can(X), P^\we)=0
\]
for all $n>0$.
\end{itemize}
\end{lemma}

\begin{proof} \red{We refer to \cite{Ding25} for the proof.}
\end{proof}

Let $\cp_0\subseteq \cp$ be a full additive dg subcategory consisting
of projective--injective objects. Let us suppose that the inclusion
\[
\begin{tikzcd}
\per(\cp_0) \arrow{r}{I} & \cd^b(\ce)
\end{tikzcd}
\]
has a left adjoint $I_\lambda$ and a right adjoint $I_\rho$.
Recall from theorem~3.23 of \cite{Chen24b} that the dg
quotient $\ce/ \cp_0$ inherits a canonical exact
structure from $\ce$. 

\begin{prop} \label{prop: projective domination and quotients}
The exact dg category $\ce$ is projectively dominated
if and only if the dg quotient $\ce/ \cp_0$ is projectively 
dominated.
\end{prop}

\begin{proof} The dg quotient $\ce/\cp_0$ still has enough
projectives and the corresponding subcategory identifies with
$\cp/ \cp_0$. Let us write
$\Phi: \cd^b(\ce) \to \cd(\cp)$ for the canonical functor and use
the symbol $\Phi_0$ for the induced functor $\per(\cp_0) \to \cd(\cp_0)$
and the symbol $\Phi_1$ for the induced functor in the
quotients. It is not hard to check that we have
a morphism of recollements
\[
\begin{tikzcd}
\per (\cp_0) \arrow{d} \arrow{rr}[description]{I} &&\cd^b(\ce) \arrow{d}{\Phi}
\arrow[yshift=-1.5ex]{ll}{I_\rho}
\arrow[yshift=1.5ex]{ll}[swap]{I_\lambda}
\arrow{rr}[description]{Q} &&\cd^b(\ce/ \cp_0) \arrow{d}
\arrow[yshift=-1.5ex]{ll}{Q_\rho}
\arrow[yshift=1.5ex]{ll}[swap]{Q_\lambda}. \\
\cd (\cp_0) \arrow{rr}[description]{I} &&\cd(\cp) 
\arrow[yshift=-1.5ex]{ll}{I_\rho}
\arrow[yshift=1.5ex]{ll}[swap]{I_\lambda}
\arrow{rr}[description]{Q} &&\cd(\cp/ \cp_0) 
\arrow[yshift=-1.5ex]{ll}{Q_\rho}
\arrow[yshift=1.5ex]{ll}[swap]{Q_\lambda}.
\end{tikzcd}
\]
Now let $L$ and $M$ be objects of $\cd^b(\ce)$. We have the
canonical triangle
\[
\begin{tikzcd}
M' \arrow{r} & M \arrow{r} & M'' \arrow{r} & \Si M'
\end{tikzcd}
\]
where $M'=I I_\rho M$ and $M''$ is the cone over the
counit of the adjunction. We write $\RHom$ for the
morphism complexes in the respective canonical
dg enhancements. The above triangle yields a triangle
\[
\begin{tikzcd}
\RHom(L, M') \arrow{r} & \RHom(L,M) \arrow{r} & \RHom(L,M'') \arrow{r} & \Si \RHom(L, M').
\end{tikzcd}
\]
The adjoint pair $(I_\lambda, I)$ yields the isomorphism
\[
\begin{tikzcd}
\RHom(I I_\lambda L,M') \arrow{r}{_\sim} & \RHom(L,M')
\end{tikzcd}
\]
and $\Phi$ clearly induces an isomorphism from the left hand side to
\[
\RHom(\Phi_0 I I_\lambda L, \Phi_0 M').
\] 
By the adjunction between
$Q_\rho$ and $Q$, we have the isomorphism
\[
\begin{tikzcd}
\RHom(L'', M'') \arrow{r}{_\sim} & \RHom(L, M'').
\end{tikzcd}
\]
We conclude that $\Phi$ induces a morphism of triangles
\[
\begin{tikzcd} 
\RHom(I I_\lambda L, M') \arrow{r} \arrow{d}{\Phi_0} & \RHom(L,M) \arrow{d}{\Phi} \arrow{r} &
\RHom(L'', M'') \arrow{d}{\Phi_1}  \arrow{r} & \mbox{ }\\
\RHom(\Phi_0 II_\lambda L, \Phi_0 M') \arrow{r} &  \RHom(\Phi L, \Phi M) \arrow{r} & 
\RHom(\Phi_1 L'', \Phi_1 M'') \arrow{r} & \mbox{ }
\end{tikzcd}
\]
Clearly the left vertical arrow induced by $\Phi_0$ is an isomorphism.
It follows that $\Phi$ induces an isomorphism if and only if
$\Phi_1$ induces an isomorphism. This implies the claim.
\end{proof}

\subsection{Projective domination for the Higgs category} \label{ss: Projective domination for the Higgs category}
We keep the notations and assumptions of section~\ref{ss: The Higgs category}.
We write $\ch_{dg}$ for the connective cover of the canonical dg enhancement
of the Higgs category $\ch$. We endow it with its canonical exact
structure so that it becomes a Frobenius exact dg category, cf.~section~3.6.4
of \cite{Chen24b}, where it is shown that the bounded derived category
of $\ch_{dg}$ is canonically equivalent to the cluster category $\cc$.

\begin{corollary}  \label{cor: The Higgs category is projectively dominated}
The dg Higgs category $\ch_{dg}$ is projectively dominated.
\end{corollary}

\begin{proof} Let $\cp_{dg}$ be the full dg subcategory of 
$\ch_{dg}$ whose objects are the projective--injectives.
Let $\cp_{dg}^0$ be the full dg subcategory
of $\cp_{dg}$ whose objects are those in the essential
image of $G_0$. By Prop.~\ref{prop: right adjoint for G_i},
the adjoints $C_0$ and $C_1$ of $D_0$
yield adjoints for the inclusion $\per(\cp_{dg}^0)  \to \per(\cp_{dg})$.
By Proposition~\ref{prop: projective domination and quotients},
the Higgs category $\ch_{dg}$ is projectively dominated if
and only if its quotient $\ch^1_{dg}=\ch_{dg}/ \cp_{dg}^0$ is.
The derived category of $\ch^1_{dg}$ is canonically
equivalent to the Verdier quotient $\cc^1$ of $\cc$ by its
full subcategory $\per(\cp_{dg}^0)$.
Now let $\cp_{dg}^1$ be the full subcategory of $\ch^1_{dg}$
whose objects are those in the essential image of $G_1$.
Since $\cp_{dg}^0$ is right orthogonal to
$\cp_{dg}^1$, the right adjoint $C_2$ of $G_1$ obtained
in Proposition~\ref{prop: right adjoint for G_i} vanishes on
$\cp_{dg}^0$ and induces a right adjoint for 
the functor $\per(\cp_{dg}^1) \to \cc^1$.
Since the functor $G_0$ admits a left adjoint (namely $C_0$,
cf. Prop.~\ref{prop: right adjoint for G_i} ), the inclusion of
the left orthogonal of the image of $G_0$ admits a 
right adjoint $X \mapsto pX$. Then it is easy to see that
the functor $X \mapsto C_1(pX)$ induces a left
adjoint to the inclusion of the image of $G_1$ in $\cc^1$.
By Prop.~\ref{prop: projective domination and quotients}, it follows that
the quotient $\ch_{dg}^2= \ch_{dg}^1 / \cp_{dg}^1$
is projectively dominated if and only if this holds for
$\ch_{dg}^1$. Now by Cor.~\ref{cor: localized cluster category}, 
the exact dg category  $\ch_{dg}^2$ is equivalent to the dg Higgs category
associated with the relative $3$-Calabi-Yau completion
of the canonical embedding
\[
\begin{tikzcd}
\proj(kQ) \arrow{r} & \mod(kQ).
\end{tikzcd}
\]
As shown in \cite{Wu23}, this Higgs category is concentrated
in degree $0$ and equivalent to the category $\mod \Lambda$
of finite-dimensional modules over the preprojective algebra
$\Lambda = H^0(\Pi_2(kQ))$ of the Dynkin quiver $Q$. Since
it is a Quillen exact Frobenius category, it is projectively
dominated and this therefore also holds for
$\ch_{dg}^1$ and $\ch_{dg}$.
\end{proof}

\subsection{Reflexivity} \label{ss: Reflexivity}
Let $\ca$ be a connective dg category. Let $M$ be a dg $\ca$-module.
Recall from section~\ref{ss: notations} that the {\em $\ca$-dual} of $M$ is the 
dg $\ca^{op}$-module
\[
M^\vee = \RHom_\ca(M,\ca) \ko
\]
and that $M$ is {\em reflexive} if the canonical
morphism
\begin{equation} \label{eq: double dual}
\begin{tikzcd} 
M \arrow{r} & (M^\vee)^\vee
\end{tikzcd}
\end{equation}
is invertible in $\cd(\ca)$. 
\red{Let us recall how to interpret the canonical morphism~\ref{eq: double dual} as an
adjunction morphism.} For this, let us write $D_\ca$ for the duality functor
\[
\begin{tikzcd}
\cd(\ca) \arrow{r} & \cd(\ca^{op})^{op}
\end{tikzcd}
\]
taking $M$ to $M^\vee$. By a slight abuse of notation, we write $D_{\ca^{op}}$
for the analogous functor
\[
\begin{tikzcd}
\cd(\ca^{op})^{op} \arrow{r} & \cd(\ca).
\end{tikzcd}
\]
\red{The morphism~\ref{eq: double dual} yields a} canonical morphism of functors
\[
\begin{tikzcd}
\eta: \id_{\cd(\ca)} \arrow{r} & D_{\ca^{op}} \circ D_\ca
\end{tikzcd}
\]
and this morphism exhibits $D_{\ca^{op}}$ as a right
adjoint to $D_\ca$ (since the non derived functors
$D^{nd}_\ca$ and $D^{nd}_{\ca^{op}}$ form an adjoint
pair of functors between the categories up to 
homotopy of dg modules over $\ca$ and $\ca^{op}$). 

\begin{lemma} \label{lemma: Faithfulness implies reflexivity}
A dg $\ca$-module $M$ is reflexive if and
only if the canonical morphism
\[
\begin{tikzcd}
\RHom_\ca(X, M) \arrow{r} & \RHom_{\ca^{op}}( M^\vee, X^\vee)
\end{tikzcd}
\]
is invertible for each representable dg $\ca$-module $X$.
\end{lemma}

\begin{proof}
\red{Let $M$ be a dg $\ca$-module.
Since $\cd(\ca)$ is compactly generated by the representables, the adjunction morphism}
\[
\begin{tikzcd}
\eta M: M \arrow{r} & D_{\ca^{op}}(D_\ca M)
\end{tikzcd}
\]
\red{is invertible if and only if the morphism}
\[
\begin{tikzcd}
\RHom_\ca(X,\eta M): \RHom_\ca(X,M) \ar[r] & \RHom_\ca(X, D_{\ca^{op}}(D_\ca M))
\end{tikzcd}
\]
\red{is invertible for each representable dg $\ca$-module $X$. The claim follows 
because of the adjunction isomorphism}
\[
\begin{tikzcd}
\RHom_\ca(X, D_{\ca^{op}}(D_\ca M)) \arrow{r}{_\sim} & \RHom_{\ca^{op}}(M^\vee, X^\vee).
\end{tikzcd}
\]
\end{proof}

Now let $Q: \ca \to \cb$ be a dg localization such that
$\cb(Q?, B)$ is a perfect $\ca$-module for each $B\in \cb$.
In other words, we suppose that the induced functor 
\[
Q: \per(\ca) \to \per(\cb)
\]
admits a right adjoint. \red{Notice that, since $Q$ is a dg localization, such a
right adjoint is automatically fully faithful.}

\begin{lemma} \label{lemma: Q preserves reflexivity} We have a canonical isomorphism 
\[
\begin{tikzcd}
Q^{op}\circ D_\ca \arrow{r}{_\sim} & D_\cb \circ Q.
\end{tikzcd}
\]
If moreover $\cb(B, Q?)$ is a perfect left $\ca$-module for each
$B$ in $\cb$, then the functor $Q: \cd(\ca) \to \cd(\cb)$ 
preserves reflexivity.
\end{lemma}

\begin{proof} Let $X$ be an object of $\cd\ca$ and $B$ an object of $\cb$.
We have canonical isomorphisms
\begin{align*}
(Q^{op} \circ D_\ca(X))(B) & \iso Q^{op}(\RHom_\ca(X,\ca))(B) \\
& \iso  \cb(?,B)\lten_\ca \RHom_\ca(X, \ca) \\
& \iso \RHom_\ca(X, \cb(Q?, B)) \\
& \iso \RHom_\ca(QX, \cb(?, B)) \\
& \iso ((D_\cb \circ Q)(X))(B).
\end{align*}
Here, for the third isomorphism, we have used that $\cb(Q?, B)$ is a perfect
dg $\ca$-module and for the fourth isomorphism, we have used that
$Q:\cd\ca \to \cd\cb$ is left adjoint to the restriction along $Q$. 
Under the hypothesis of the second statement, we also have
an isomorphism
\[
\begin{tikzcd}
Q \circ D_{\ca^{op}} \arrow{r}{_\sim} & D_{\cb^{op}} \circ Q
\end{tikzcd}
\]
so that if $X\in \cd(\ca)$ is reflexive, then so is $QX$ in $\cd(\cb)$. 
\end{proof}

\subsection{Reflexive objects from Frobenius categories}
Now let $\ce$ be a (connective) exact dg category as in
section~\ref{ss: Projective domination for the Higgs category}.
The category $\ce$ is {\em injectively submitted} if its opposite
$\ce^{op}$ is projectively dominated. Let us spell this out
explicitly: Let $\ci\subset \ce$ be the subcategory of the injective
objects of $\ce$. In analogy with section~\ref{ss: Projective domination}
we have a canonical dg functor
\[
\begin{tikzcd}
\cd^b_{dg}(\ce)  \arrow{r} & \cd(\ci^{op})^{op}
\end{tikzcd}
\]
taking an object $X$ of $\ce$ to the restriction of the representable
$\ce(X,?)$ to $\ci$. The exact dg category $\ce$ is
injectively submitted if this functor is quasi-fully faithful.

Now suppose that $\ce$ is Frobenius (i.e.~it has enough projectives and
enough injectives and these two classes coincide). \red{Let $\cp \subseteq \ce$
be the full dg subcategory of the projective-injectives of $\ce$}.

\begin{lemma}[\cite{Ding25}] \label{lemma: from reflexivity} If $\ce$ is both projectively dominated and
injectively submitted, then the canonical functor
\[
\begin{tikzcd}
R: \cd^b(\ce) \arrow{r} & \cd(\cp)
\end{tikzcd}
\]
takes each object $X$ of $\cd^b(\ce)$ to a reflexive dg $\cp$-module.
\end{lemma}

\begin{proof} Let us denote the canonical functor 
\[
\begin{tikzcd}
\cd^b(\ce) \arrow{r} & \cd(\ci^{op})^{op}
\end{tikzcd}
\]
by $R'$ (recall that $\cp=\ci$). Let $X$ be an object of $\cd^b(\ce)$.
Since $R$ is fully faithful, the canonical morphism
\[
\begin{tikzcd}
\Hom_{\cd^b_{dg}(\ce)}(X,P) \arrow{r} & \RHom_\cp(RX, RP) = (RX)^\vee(P)
\end{tikzcd}
\]
is invertible for each $P$ in $\cp$. This means that the canonical
morphism 
\[
\begin{tikzcd}
R' \arrow{r} & D_\cp \circ R
\end{tikzcd}
\]
is invertible. By our assumption, $R'$ and $R$ are quasi-fully 
faithful. Therefore, the functor $D_\cp$ restricted to the
image of $R$ is fully faithful. By the above Lemma~\ref{lemma: Faithfulness implies reflexivity},
we deduce that $RX$ is reflexive for each object $X$ of $\cd^b(\ce)$.
\end{proof}

\begin{corollary} \label{cor: Higgs objects are reflexive}
With the notations of section~\ref{ss: Projective domination for the Higgs category},
the dg Higgs category $\ch_{dg}$ is projectively dominated and injectively submitted and
the canonical functor $\cc  \to \cd(\cp)$ takes each object of the cluster category $\cc$
to a reflexive dg $\cp$-module.
\end{corollary}

\begin{proof} The opposite category $\ch^{op}$ is canonically equivalent to the Higgs
category constructed from the opposite quiver $Q^{op}$. Thus, the category $\ch^{op}$
is also projectively dominated and the claim follows from the preceding lemma.
\end{proof}

\subsection{Gorenstein projective modules} \label{ss: Gorenstein projective modules}
Let $\ca$ be a connective dg category. A dg $\ca$-module $M$ is \emph{pseudocoherent}
if $M$ is homologically right bounded and the complex 
\[
M \lten_\ca H^0(\ca)
\]
is quasi-isomorphic to a right bounded complex of finitely generated projective $H^0(\ca)$-modules. 

\begin{remark} \label{remark: Pseudocoherence} 
\begin{enumerate}
\item In the above definition, the condition that the $\ca$-module $M$ is homologically right bounded
is not implied by the other condition, as shown by the following example: 
Let $A$ be the dg algebra $k[t]$, where the generator $t$ is in degree $-1$ and $d(t)=0$.
Let $M$ be the dg $A$-module  $k[t, t^{-1}]$ with the vanishing differential. 
Then $H^0(A)$ is the ground field $k$ and we have
\[
M \lten_A H^0(A) = M \lten_A k =M\otimes_A \cone (\Si k[t]\xrightarrow{t} k[t])=0.
\]
So $M \lten_A H^0(A)$ is quasi-isomorphic to a right bounded complex of finitely generated projective $H^0(A)$-modules (the zero complex). But $M$ is not homologically right bounded. 
\item If $\ca$ is a connective dg algebra $A$ with finite-dimensional homologies 
$H^p(A)$, $p\in \Z$, then a dg $A$-module $M$ is pseudocoherent if and only if the spaces
$H^p(M)$ vanish for all sufficiently large integers $p$ and are finite-dimensional for all
integers $p$. This follows from Theorem~3c) in \cite{Keller94}.
\end{enumerate}
\end{remark}

A dg $\ca$-module $M$ is \emph{Gorenstein projective} in the sense of \cite{Ding25} if 
\begin{itemize}
\item[a)] $M$ and its dual $M^\vee=\RHom_\ca(M,\ca)$ are pseudocoherent,
\item[b)] $M$ is reflexive, i.e.~the canonical morphism $M \to (M^\vee)^\vee$ is invertible in $\cd(\ca)$ and
\item[c)] $M$ and $M^\vee$ are connective. 
\end{itemize}
We define the \emph{category of Gorenstein projective $\ca$-modules $\gpr(\ca)$} to be
the full subcategory of $\cd(\ca)$ whose objects are the Gorenstein projective dg $\ca$-modules.

\begin{lemma}[\cite{Ding25}] If $A$ is  a smooth connective dg algebra with finite-dimensional $H^0(A)$, 
then the subcategory $\gpr(A)$ of $\cd(A)$ is the closure $\add(A)\subset \cd(A)$ under finite direct
sums and summands of the free $A$-module $A_A$.
\end{lemma}

\begin{proof} Let $X$ be an object of $\gpr(A)$. Since $A$ is connective and $H^0(A)$
is finite-dimensional, in order to conclude that $X$ belongs to $\add(A)$, it suffices
to show that we have
\[
\Hom_{\cd(A)}(X, \Si^p V)=0
\]
for each simple $H^0(A)$-module $V$ and all $p>0$. Since $A$ is smooth,
the category $\pvd(A)$ is contained in $\per(A)$. So $V$ is a direct factor
of a finite iterated extension of objects $\Si^i A$ for $i\geq 0$. This 
implies the desired vanishing of $\Hom_{\cd(A)}(X, \Si^p V)$ since
$X$ is Gorenstein projective.
\end{proof}

\begin{example} 
We claim that the dg category $\cp_{dg}$ is smooth only if $Q$ is of type $A_1$.
Indeed, otherwise the Higgs category contains non-projective objects so that, by
Cor.~\ref{cor: The Higgs category is projectively dominated},  the category
$\gpr(\cp_{dg})$ contains objects which are not in $\add \cp$ in contradiction with
the lemma.
\end{example}

A dg $\ca$-module $X$ is {\em derived Gorenstein projective} if it is reflexive and both
$X$ and $X^\vee$ are pseudocoherent. We write $\dgp(\ca)$ for the full subcategory
of $\cd(\ca)$ whose objects are the derived Gorenstein projective dg $\ca$-modules.
Clearly each perfect dg $\ca$-module is derived Gorenstein projective. 
Moreover, each Gorenstein projective dg $\ca$-module is derived Gorenstein
projective.

\begin{lemma} \label{lemma: gpr generates dgp}
The triangulated subcategory of $\cd(\ca)$ generated by $\gpr(\ca)$ equals $\dgp(\ca)$.
\end{lemma}

\begin{proof} Let $X$ be an object of $\dgp(\ca)$. Since $X^\vee$ is pseudocoherent, 
its homologies vanish in all degrees $>N$ for some $N\gg 0$. After replacing $X$ with $\Si^{-N}X$, 
we may assume that $X^\vee$ is connective. Since $\ca$ is connective, the derived
category $\cd(\ca)$ has a canonical weight structure. We write $\sigma^{\geq p}$ and $\sigma^{\leq p}$
for the corresponding truncation operators. Thus, we have a triangle
\[
\begin{tikzcd}
\sigma^{\geq 1} (X) \arrow{r} & X \arrow{r} & \sigma^{\leq 0}(X) \arrow{r} & \Si \sigma^{\geq 1}(X).
\end{tikzcd}
\]
Notice that $\sigma^{\geq 1}(X)$ is perfect hence belongs to $\dgp(\ca)$. It follows
that $\sigma^{\leq 0}(X)$ also belongs to $\dgp(\ca)$. 
We claim that it even belongs to $\gpr(\ca)$. Indeed, this object is connective. To check
whether its dual is connective, let us consider the dualized triangle
\[
\begin{tikzcd}
\Si^{-1}(\sigma^{\geq 1}(X))^\vee \arrow{r} & \sigma^{\leq 0}(X)^\vee \arrow{r} & X^\vee \arrow{r} & (\sigma^{\geq 1}(X))^\vee.
\end{tikzcd}
\]
Clearly the object $\Si^{-1}(\sigma^{\geq 1}(X))^\vee$ 
is connective. Moreover, the object $X^\vee$ is
connective by construction. Thus, the object $\sigma^{\leq 0}(X)^\vee$ is connective as an extension
of connective objects. 
\end{proof}

\begin{lemma} \label{lemma: per=dgp}
If $A$ is a smooth connective dg algebra with finite-dimensional $H^0(A)$, then
we have $\per(A)=\dgp(A)$.
\end{lemma}

\begin{proof} Let $X$ be a derived Gorenstein projective dg $A$-module. Since $X^\vee$
is pseudocoherent, there is an integer $N\geq 0$ 
such that $\Hom_{\cd A}(X, \Si^n A)$ vanishes for each $n\geq N$. Then, for each
perfect and connective dg $A$-module $P$, we also have $\Hom_{\cd A}(X, \Si^n P)=0$
for each $n\geq N$. Since $A$ is smooth, each simple $H^0(A)$-module $E$ is
perfect and it is clearly connective. So for each simple $H^0(A)$-module $E$,
we have $\Hom_{\cd A}(X, \Si^n E)=0$ for all $n\geq N$. Since $X$ is 
pseudocoherent, 
this implies that $X$ is perfect as we see using a
minimal cofibrant resolution of $X$.
\end{proof}

\begin{prop} \label{prop: quotient in gpr} Let $Q:\ca \to \cb$ be a dg localization of connective dg categories
such that the induced functor $Q: \per(\ca) \to \per(\cb)$ admits a left adjoint and
a right adjoint \red{(which are automatically fully faithful).}
\begin{itemize}
\item[a)]  The functor $Q:\cd(\ca) \to \cd(\cb)$ induces functors 
\[
\begin{tikzcd}
\dgp(\ca) \arrow{r} & \dgp(\cb) & \mbox{and} & \gpr(\ca) \arrow{r} & \gpr(\cb).
\end{tikzcd}
\]
\item[b)] Suppose moreover that $Q:\ca\to\cb$ is the dg localization
at a full dg subcategory $\cn$ of $\ca$. Let $X$ be an object
of $\gpr(\ca)$ such that the restriction of $X^\vee$ to $\cn$ is a perfect
left dg $\cn$-module. Then, for all objects $Y$ of $\gpr(\ca)$, 
the map
\[
\Hom_{\cd\ca}(X,Y) \to \Hom_{\cd\cb}(QX, QY)
\]
is surjective and its kernel is formed by the morphisms factoring through a finite
direct sum of representables $N^\we$, $N\in \cn$.
\end{itemize}
\end{prop}

\begin{remark} The proposition could be used to generalize the main result
of \cite{Chen14}.
\end{remark}

\begin{proof} a)  Since $\cb$ is connective, the functor $Q:\cd(\ca) \to \cd(\cb)$ takes connective pseudocoherent
dg $\ca$-modules to connective pseudocoherent $\cb$-modules. 
Moreover, the same holds for $Q^{op}: \cd(\ca^{op}) \to \cd(\cb^{op})$.
By Lemma~\ref{lemma: Q preserves reflexivity}, the functor $Q$
preserves reflexivity. Thus, it induces a functor from $\dgp(\ca)$ 
to $\dgp(\cb)$. If $X \in \cd(\ca)$ is Gorenstein projective, we
have
\[
\Hom_{\cd(\cb)}(QX, \Si^p Q A^\we) = \Hom_{\cd(\ca)}(X, \Si^p Q_\rho Q A^\we).
\]
The functor $Q_\rho$ also preserves connectivity since it is
just the restriction along $Q$. So the object $Q_\rho Q A^\we$ is
connective and perfect. Since $X$ is Gorenstein projective,
the right hand side above vanishes and $QX$ is still
Gorenstein projective.

b) Let $X$ and $Y$ be objects
of $\gpr(\ca)$ as in the statement. \redv{We compute morphisms
in the localization $\cd(\cb) \iso \cd(\ca)/\cd(\cn)$ using right fractions}.
Thus, we consider the category of morphisms $X \to N$, where
$N$ belongs to $\cd(\cn)$. Since $X$ is connective, the subcategory
of morphisms $X \to N$ with connective $N$ is cofinal.
So let $f: X \to N$ be a morphism with connective $N$. Since
$\cn$ is connective, the category $\cd(\cn)$ has a 
standard weight structure. We write $\sigma^{\geq -p} (N)$,
$p\geq 0$, for the corresponding truncations of $N$.
Then the object $N$ is the homotopy colimit of the
system formed by the $\sigma^{\geq -p}(N)$, $p\geq 0$.
Since the restriction of $X^\vee$ to $\cn$ is a perfect
left $\cn$-module, we have the isomorphism
\[
\begin{tikzcd}
\Hom_{\cd(\ca)}(X, \colim\, \sigma^{\geq -p}(N))  \arrow{r}{_\sim} & 
\colim\, \Hom_{\cd(\ca)}(X, \sigma^{\geq -p}(N)).
\end{tikzcd}
\]
Thus, the morphism $f$ factors through some $N'=\sigma^{\geq -p}(N)$.
Consider the triangle
\[
\begin{tikzcd}
\sigma^{\geq 0}(N') \arrow{r} & N' \arrow{r} & \sigma^{<0}(N') \arrow{r} & \Si \sigma^{\geq 0}(N').
\end{tikzcd}
\]
Since $X$ lies in $\gpr(\ca)$ and $\sigma^{<0}(N')$ is a {\em finite} extension of
objects $\Si^p A^\we$, $p>0$, $A\in \ca$, there are no non-zero
morphisms from $X$ to
$\sigma^{<0}(N')$. Thus, the morphism $f$ factors through $\sigma^{\geq 0}(N')$,
which lies in $\cn$ itself. Therefore, any right fraction from $X$ to $Y$ is
equivalent to a right fraction with denominator $s$ fitting into a triangle
\[
\begin{tikzcd}
\Si^{-1} N \arrow{r} & X' \arrow{r}{s} & X \arrow{r} & N
\end{tikzcd}
\]
where $N$ belongs to $\cn$. Since $Y$ is connective and $N$ belongs
to $\cn\subset\ca$, there are no non-zero morphisms from $\Si^{-1} N$
to $Y$. Thus, any right fraction from $X$ to $Y$ is equivalent to
the image of a morphism $X \to Y$ in $\gpr(\ca)$. A morphism
$g:X \to Y$ in $\cd(\ca)$ belongs to the kernel if and only if its composition
with some morphism $s$ as above vanishes. But this means that
$g$ factors through the object $N$.

\end{proof}

\begin{lemma}[\cite{Ding25}] Let $\ce$ be a Frobenius exact connective dg category and
$\cp\subset \ce$ its subcategory of projective-injectives.
Suppose that $\ce$  is both projectively dominated and injectively submitted.
Then the canonical functor
\[
\begin{tikzcd}
R: \cd^b(\ce) \arrow{r} & \cd(\cp)
\end{tikzcd}
\]
is fully faithful and induces functors
\[
\begin{tikzcd}
\cd^b(\ce) \arrow{r} & \dgp(\cp) & \mbox{and} & \ce \arrow{r} & \gpr(\cp).
\end{tikzcd}
\]
\end{lemma}

\begin{proof} Since $\ce$ is projectively dominated and injectively submitted,
the functor $R$ takes the objects of $\ce$ to objects which are 
pseudocoherent and connective (by Lemma~\ref{lemma: from projective domination}) as well as reflexive 
(Lemma~\ref{lemma: from reflexivity}). Thus, it takes objects of $\ce$ to
objects in $\gpr(\cp)$ and therefore objects in $\cd^b(\ce)$ to objects
in $\dgp(\cp)$. Moreover, it is fully faithful by
the definition of projective domination. 
\end{proof}

\begin{theorem}  Let $\ce$ be a Frobenius exact connective dg category and
$\cp\subset \ce$ its subcategory of projective-injectives.
Suppose that $\ce$  is both projectively dominated and injectively submitted.
Let $\cp_0 \subseteq \cp$ be an additive dg subcategory such that
\begin{itemize}
\item[a)] $\cp_0$ is smooth and $H^0(\cp_0)$ is Morita equivalent to a finite-dimensional algebra.
\item[b)] the inclusion $I: \per(\cp_0) \to \cd^b(\ce)$ admits a left adjoint $I_\lambda$
and two successive right adjoints $I_\rho$ and $I_{\rho\rho}$. 
\end{itemize}
Then the functors
\[
\begin{tikzcd}
\cd^b(\ce) \arrow{r} & \dgp(\cp) & \mbox{and} & \ce \arrow{r} & \gpr(\cp).
\end{tikzcd}
\]
are equivalences if the functors
\[
\begin{tikzcd}
\cd^b(\ce/\cp_0) \arrow{r} & \dgp(\cp/\cp_0) & \mbox{and} & \ce/\cp_0 \arrow{r} & \gpr(\cp/\cp_0)
\end{tikzcd}
\]
are equivalences.
\end{theorem}

\begin{proof}
Since the inclusion $I: \per(\cp_0) \to \cd^b(\ce)$ admits a left and a right
adjoint so does the inclusion $I: \per(\cp_0) \to \per(\cp)$ and thus, so does
the quotient functor $\per(\cp) \to \per(\cp/ \cp_0)$. Therefore, by
Prop.~\ref{prop: quotient in gpr}, the quotient functor induces
well-defined functors $\gpr(\cp) \to \gpr(\cp/ \cp_0)$ and $\dgp(\cp) \to \dgp(\cp/\cp_0)$. So we obtain
a commutative square
\[
\begin{tikzcd} 
\cd^b(\ce) \ar[d, "\pi_1"'] \arrow{r} & \dgp(\cp) \arrow{d}{\pi_2} \\
\cd^b(\ce/ \cp_0) \arrow{r}{_\sim} & \dgp(\cp/ \cp_0).
\end{tikzcd}
\]
where both horizontal arrows correspond to fully faithful functors. By assumption,
the bottom horizontal arrow represents an equivalence. Let $Y$ be an object of
$\dgp(\cp)$. By inspecting the diagram, we see that there is an object
$X$ of $\cd^b(\ce)$ such that $\pi_2(Y)$ becomes isomorphic to $\pi_2(RX)$
in $\dgp(\cp/\cp_0)$.  Now the restriction of $\RHom_\cp(RX,?)$ to $\cp_0$ is
perfect since the inclusion of $\per(\cp_0)$ into $\cd^b(\ce)$ admits a left
adjoint. It follows from part b) of Prop.~\ref{prop: quotient in gpr} that a given
isomorphism $\pi_2(RX) \iso \pi_2(Y)$ can be lifted to a morphism
$f: RX \to Y$ of $\dgp(\cp)$. Let us form a triangle
\[
\begin{tikzcd}
RX \arrow{r}{f} & Y \arrow{r} & N \arrow{r} & \Si RX.
\end{tikzcd}
\]
Here the object $N$ lies in $\cd(\cp_0)\subseteq \cd(\cp)$. Since $RX$ 
and $Y$ lie in $\dgp(\cp)$, so does the object $N$. In particular it is
reflexive in $\cd(\cp)$. Since $I_\rho: \per(\cp) \to \per(\cp_0)$ is a
dg quotient admitting a left adjoint (namely $I$) and a right adjoint
(namely $I_{\rho\rho}$), it induces a functor $I_\rho: \cd(\cp) \to \cd(\cp_0)$
preserving reflexivity. It follows that $N \iso I_\rho(N)$ is also reflexive
in $\cd(\cp_0)$. Thus, the object $N$ belongs to $\dgp(\cp_0)$. 
Since $\cp_0$ is smooth, the category $\dgp(\cp_0)$ coincides
with $\per(\cp_0)$ by Lemma~\ref{lemma: per=dgp}. Thus, the
object $N$ lies in $\per(\cp_0)$. Now clearly, the subcategory
$\per(\cp_0)$ is contained in the image of $\cd^b(\ce)$ under
$R$. Since $RX$ is also contained in the image of $\cd^b(\ce)$
and $R$ is fully faithful, it follows that $Y$ belongs to the image
of $\cd^b(\ce)$.

It remains to be shown that the functor $\ce \to \gpr(\cp)$ is
essentially surjective. Let $Y$ be an object of $\gpr(\cp)$.
By what we have already shown, there is an object
$X$ in $\cd^b(\ce)$ such that $RX$ is isomorphic to $Y$
and it is easy to see that we may choose $X$ connective.
Since the $\cp$-dual of $Y$ is connective, we have
\[
\Ext^p_\ce(X,I)=0
\]
for all injectives $I$ of $\ce$ and all $p>0$. By Lemma~\ref{lemma: Characterizing objects in degree 0},
we conclude that $X$ lies in $\ce$.
\end{proof}

\subsection{The Higgs category and Gorenstein projective dg modules}
With the notations of section~\ref{ss: Projective domination for the Higgs category}, recall
that the canonical dg functor
\[
\begin{tikzcd}
\cd^b_{dg}(\ch_{dg}) \arrow{r} &  \cc_{dg}
\end{tikzcd}
\]
is an equivalence. We use it to identify the two dg categories. In particular, the canonical
functor
\[
\begin{tikzcd}
\cd^b_{dg}(\ch_{dg}) \arrow{r} & \cd(\cp_{dg})
\end{tikzcd}
\]
sending $X\in\ch$ to the restriction of the functor $\ch_{dg}(?,X)$ to $\cp_{dg}$ corresponds
to a functor 
\[
\begin{tikzcd}
R: \cc \arrow{r} & \cd(\cp_{dg}).
\end{tikzcd}
\]
\begin{theorem} \label{thm: Higgs=gpr}
The functor $R$ induces 
\begin{itemize}
\item[a)]  an equivalence from $\ch$ onto the subcategory $\gpr(\cp_{dg})$ of 
Gorenstein projective dg modules and
\item[b)] an equivalence from $\cc$ onto the full subcategory $\dgp(\cp_{dg})$ of $\cd(\cp_{dg})$.
\end{itemize}
\end{theorem}

\begin{remark} \red{By the theorem~6.6.1, we have a commutative square}
\[
\begin{tikzcd} 
\ch \arrow{r}{_\sim} \arrow{d} & \gpr(\cp_{dg}) \arrow{d} \\
\cc \arrow{r}{_\sim}                 & \dgp(\cp_{dg}) \ko
\end{tikzcd}
\]
\red{where the horizontal equivalences take an object $X$ to the restriction
of $\RHom_\cc(?,X)$ to $\cp_{dg}$. }
\end{remark}

\begin{proof} We know from Corollary~\ref{cor: Higgs objects are reflexive} that the functor
\[
\begin{tikzcd}
R: \cd^b(\ch_{dg}) \arrow{r} & \cd(\cp_{dg})
\end{tikzcd}
\]
is fully faithful and that the objects in its image are reflexive. 
Since $\cp$ consists of projective-injective objects of $\ch$, if $X$ lies in $\ch$, then 
$RX$ is connective and so is its dual $(RX)^\vee$. It follows from Remark~\ref{remark: Pseudocoherence} that
$RX$ and $(RX)^\vee$ are pseudocoherent. We conclude that $R$ induces a 
well-defined functor from $\ch$ to $\gpr(\cp_{dg})$.  Since $\cc$ is generated by $\ch$ as
a triangulated category, it follows that $R$ induces a fully faithful functor from $\cc$
to $\dgp(\cp_{dg})$. Since $\dgp(\cp_{dg})$ is generated by $\gpr(\cp_{dg})$ by Lemma~\ref{lemma: gpr generates dgp},
part b) will follow once 
we show a). 
Proceeding as in the proof of Corollary~\ref{cor: The Higgs category is projectively dominated}
we see that we have a commutative square
\[
\begin{tikzcd} 
\ch \arrow{r}{R} \ar[d, "\pi_1"'] & \gpr(\cp_{dg}) \ar{d}{\pi_2}\\
\mod \La \arrow{r}{R_\La} & \gpr(\proj(\La)) ,
\end{tikzcd}
\]
where the functor $\pi_1$ is the quotient by the ideal of
morphisms factoring through a sum of objects in
$\cp^0$ and $\cp^1$ and the functor $\pi_2$ induces
a fully faithful functor from the quotient of $\gpr(\cp_{dg})$
by the ideal generated by the image under $R$ of
$\cp^0$ and $\cp^1$. 
Notice that the restriction of the
functor $R$ to the additive subcategory generated by $\cp^0$ and
$\cp^1$ is an equivalence onto its image. Since the
functor $R_\La$ is an equivalence, it follows that
the functor $R$ is essentially surjective.
We already know that it is fully faithful so it
is an equivalence. Thus, we have shown a). 
\end{proof}

\section{Comparison with the cosingularity category}
\label{s: Comparison with the cosingularity category}

\subsection{Auxiliary results} \label{ss: Auxiliary results}
We use the notations and assumptions of section~\ref{ss: The Higgs category}.
We abbreviate $\Pi=\Pi_2(\proj kQ)$ and $\Ga=\Pi_3(\ca,\cb)$. Let $\cn\subset \cd(\cp_{dg})$
be the localizing subcategory generated by $\pvd(\cp_{dg})$ and let $\Cosg(\cp_{dg})$ be
the quotient $\cd(\cp_{dg})/\cn$. Clearly, the subcategory $\cn$
contains the dg modules $D_i V$, where $V$ is a simple dg $\Pi$-module
and $-1\leq i\leq 1$. The following lemma shows in particular that $\cn$ is
also the localizing subcategory of $\cd(\cp_{dg})$ generated by these modules
$D_i V$.

\begin{lemma} Each simple dg $\cp_{dg}$-module is in the closure under extensions and 
arbitrary coproducts of the dg modules $\Si^p D_i V$, where we have $p\geq 0$, $-1\leq i \leq 1$
and $V$ is a simple dg $\Pi$-module.
\end{lemma}

\begin{proof} We write $\La=H^0(\Pi(kQ))$ for the preprojective algebra of $Q$.
Since $Q$ is a Dynkin quiver, it is selfinjective. We write $\nu$ for its
Nakayama automorphism.  By Theorem~\ref{theorem1}, 
the category $\cp = H^0(\cp_{dg})$ is equivalent to 
the category of the finitely generated projective modules over the algebra $T_Q$ given
by the matrices with entries in $\La$ and $\mbox{}_\nu \La$ as follows
\[
\left[ 
\begin{array}{ccc}
 \La & 0 & \mbox{}_\nu \La \\
 \La & \La & 0 \\
 0 & \La & \La
 \end{array} \right].
 \]
 The algebra $T_Q$ is selfinjective (its Nakayama automorphism is of
 order $6$). We endow the category $\mod T_Q$ of finite-dimensional
 $T_Q$-modules with the exact structure whose conflations are the
 exact sequences of modules
 \[
 \begin{tikzcd}
 0 \arrow{r} & L \arrow{r} & M \arrow{r} & N \arrow{r} & 0
 \end{tikzcd}
 \]
 such that $(LE_{ii}, ME_{ii}, NE_{ii})$ is a split exact sequence of $\La$-modules for each integer $1\leq i\leq 3$.
 In this way, the category $\mod T_Q$ becomes a Frobenius exact category whose projectives
 are the direct sums of modules $D_i V$, where $V$ is a simple $\La$-module and $-1\leq i\leq 1$. Here,
in matrix notation, we have
\[
D_{-1}V = [V, 0, V_\nu]\ko\quad D_0 V= [V,V,0]\ko \quad D_1 V=[0,V,V].
\]
In particular, each simple $T_Q$-module $E$ admits a projective resolution $P$. Clearly,
in the derived category $\cd(T_Q)$, the resolution $P$ is in the closure under extensions and arbitrary
coproducts of the modules $\Si^p D_i V$, where $p\geq 0$, $-1\leq i\leq 1$ and $V$ is
a simple $\La$-module. Since $\cp_{dg}$ is connective and $H^0(\cp_{dg})$ is equivalent
to $\proj(T_Q)$, we have a restriction functor from $\cd(T_Q)$ to $\cd(\cp_{dg})$ which takes
the simple $T_Q$-modules to the simple dg $\cp_{dg}$-modules and commutes with
shifts, extensions and arbitrary coproducts. The claim follows.
 \end{proof}

\begin{proposition} \label{prop: bijection in Hom}
Let $L$ be a connective object of $\cd(\cp_{dg})$ and $M$ an object of $\cd(\cp_{dg})$
such that $C_i M$ is the $\Pi$-dual of a connective left dg $\Pi$-module for $0\leq i\leq 2$. 
Then the map
\[
\begin{tikzcd}
\Hom_{\cd(\cp_{dg})}(L,\Si^{-p}M) \arrow{r} & \Hom_{\cd(\cp_{dg})/\cn}(L,\Si^{-p} M)
\end{tikzcd}
\]
is bijective for any $p\geq 0$. 
\end{proposition}

\begin{proof} Notice first that for $p\geq 0$, the object $C_i\Si^{-p}M$ is still the dual
of a connective left dg $\Pi$-module for $0\leq i\leq 2$. Thus, we may and will assume
that we have $p=0$. 
We compute the space of morphisms from $L$ to $M$ in the
Verdier quotient $\cd(\cp_{dg})/\cn$ using right fractions. For an object
$N\in\cn$, since $L$ is connective, each morphism $L \to N$ in
$\cd(\cp_{dg})$ uniquely factors through the morphism $\tau_{\leq 0} N \to N$
and the object $\tau_{\leq 0} N$ still belongs to $\cn$. So the 
morphisms $L \to N$ with connective $N$ in $\cn$ form a cofinal
subcategory in the category of all morphisms from $L$ with
target in $\cn$. Thus, a morphism $L \to M$ in the quotient
category $\cd(\cp_{dg})/\cn$ is represented by a fraction
\[
\begin{tikzcd}
L & L' \arrow[l, "s"'] \arrow{r}{f} & M
\end{tikzcd}
\]
where the cone over $s$ is a connective object $N$ of $\cn$. In order
to conclude that $s$ induces a bijection
\[
\begin{tikzcd}
\Hom_{\cd(\cp_{dg})}(L,M) \arrow{r}{_\sim} & \Hom_{\cd(\cp_{dg})}(L',M),
\end{tikzcd}
\]
it suffices therefore to show that we have 
\[
\Hom_{\cd(\cp_{dg})}(N', M)=0
\]
for each object $N'$ of $\cn$ whose homologies are concentrated in
degrees $\leq 1$. By the lemma, the subcategory of such objects
$N'$ is the closure under extensions and arbitrary coproducts
of the objects $\Si^p D_i V$, where $p\geq -1$, $-1\leq i \leq 1$
and $V$ is a simple $\Pi$-module. Therefore, it suffices to show that
for these objects, we have
\[
\Hom_{\cd(\cp_{dg})}(\Si^p D_i V, M)=0.
\]
Indeed, we have
\[
\Hom_{\cd(\cp_{dg})}(\Si^p D_i V, M) = \Hom_{\cd(\Pi)}(\Si^p V, C_{i+1} M).
\]
By our assumption, the dg module $C_{i+1} M$ is of the form
$\RHom_{\Pi^{op}}(U,\Pi)$ for a connective left dg $\Pi$-module $U$. So we have
\redv{
\begin{align*}
 \Hom_{\cd(\Pi)}(\Si^p V, C_{i+1} M) &= \Hom_{\cd(\Pi)}(\Si^p V, \RHom_{\Pi^{op}}(U,\Pi))  \\
                                                          &= \Hom_{\cd(\Pi^{op})}(U, \RHom_\Pi(\Si^p V, \Pi)).
\end{align*}
}
Since $U$ is connective, it is in the closure of the $\Si^p P$, $p\geq 0$, $P\in \Pi$, under
extensions and arbitrary coproducts. So we may assume that $U=P$ for some
$P\in \Pi$ and are reduced to showing that we have
\[
\Hom_{\cd \Pi}(\Si^p V, P)=0
\]
for all $p\geq -1$. Since $V$ lies in $\pvd(\Pi)$ and $\Pi$ is $2$-Calabi--Yau,
we have
\[
D \Hom_{\cd \Pi}(\Si^p V, P) = \Hom_{\cd \Pi}(P, \Si^{p+2} V) = H^{p+2}(V(P)) \ko
\]
where $D$ denotes the $k$-dual space. Since $V$ is concentrated in
degree $0$, the space $H^{p+2}(V(P))$ vanishes indeed for all $p\geq -1$. 
\end{proof}

\begin{lemma} \label{lemma: dual of connective}
If $M$ belongs to $\ch$, then $C_i RM$ is the $\Pi$-dual of a connective
left dg $\Pi$-module for each $0\leq i \leq 2$.
\end{lemma}

\begin{proof} We know that $C_i RM$ is perfect over $\Pi$ and hence reflexive. So we need
to show that $\RHom_\Pi(C_i RM, \Pi)$ is connective. Now by adjunction, for an object
$P$ of $\Pi$, we have
\[
\RHom_\Pi( C_i RM, P) = \RHom_{\cc}(M, G_{i-1} P).
\]
Since $M$ belongs to $\ch$ and $G_{i-1} P$ is projective-injective in $\ch$, this complex
has its homologies concentrated in degrees $\leq 0$, as claimed.
\end{proof}

\subsection{The main theorem} We keep the notations and assumptions of the preceding section.

\begin{proposition} \label{prop: from H to D(P)/N} The composed functor 
\[
\begin{tikzcd}
\ch \arrow{r}{R} & \cd(\cp_{dg}) \arrow{r}{\can}  & \cd(\cp_{dg})/\cn
\end{tikzcd}
\]
is fully faithful. More precisely, it induces isomorphisms in the complexes $\tau_{\leq 0}\RHom$.
\end{proposition}

\begin{proof} This follows from Lemma~\ref{lemma: dual of connective} and Proposition~\ref{prop: bijection in Hom}.
\end{proof}

Let us recall from section~\ref{ss: The Higgs category} that $\cR$ is the full subcategory
of $\cp\subseteq \ch$ whose objects are the direct sums of objects $G_{-1} P$ and
$G_0 P'$, where $P$ and $P'$ belong to $\Pi_2(\proj kQ)$.
Recall that the large cosingularity category of $\cR$ is the quotient
\[
\Cosg(\cR) = \cd(\cR_{dg})/\Pvd(\cR_{dg}) \ko
\]
where $\Pvd(\cR_{dg})$ is the localizing subcategory of $\cd(\cR_{dg})$ generated
by $\pvd(\cR_{dg})$. 

\begin{lemma} \label{lemma: cosg(R) vs D(P)/N} The inclusion $\cR_{dg} \subset \cp_{dg}$  induces 
an equivalence
\[
\begin{tikzcd}
\Psi: \Cosg(\cR_{dg}) \arrow{r}{_\sim} & \Cosg(\cp_{dg}).
\end{tikzcd}
\]
\end{lemma}

\begin{proof} By construction, the category $\cp\subseteq \ch$ is the additive
closure (i.e. the closure under finite direct sums and direct summands) 
of the objects $G_i P$ with $P\in \add(\Pi)$ and $-1 \leq i \leq 1$.
It follows that the category $\cd(\cp_{dg})$ is compactly generated
by the corresponding objects $R G_i P$. Let $P$ be an object of
$\add(\Pi)$.  Let us show that the object $G_1 P$ is in the triangulated 
subcategory generated by $G_{-1} P$ and $G_0 P$. Recall that
in the proof of Prop.~\ref{prop: Right adjoint to D1}, for each
object $X$ of $\cc^b(\proj kQ)$, we have constructed a
diagram
\[
\begin{tikzcd}
D_{-1} X \arrow{r}{\phi X} & D_0 X \arrow{r} \arrow[rd, "\psi X"'] & C(\phi X) \arrow{d}{\alpha X} \arrow{r} & \Si D_{-1} X \\
 & & D_1 X
\end{tikzcd}
\]
in the category $\per(\ca)$.  In the relative cluster category $\cc$, for $X$ in $\Pi$, this gives rise to a diagram
\begin{equation} \label{eq: cone}
\begin{tikzcd}
G_{-1} X \arrow{r}{\phi X} & G_0 X \arrow{r} \arrow[rd, "\psi X"'] & C(\phi X) \arrow{d}{\alpha X} \arrow{r} & \Si G_{-1} X. \\
 & & G_1 X
\end{tikzcd}
\end{equation}
We claim that the image under $R$ of the cone over the morphism $\alpha X$ lies in $\cn$ 
and thus $\alpha_X$ becomes invertible in $\cd(\cp_{dg})/\cn$. Indeed, using Theorem~\ref{theorem1} 
we easily compute that the functors $C_0$ and $C_1$ take $\alpha X$ to an isomorphism
and that the functor $C_2$ (the right adjoint to $G_1$, cf. Prop.~\ref{prop: right adjoint for G_i}) 
takes $\alpha X$ to the canonical morphism
\[
\begin{tikzcd}
\tau_{\leq -1} X \arrow{r} & X \ko
\end{tikzcd}
\]
whose cone $H^0 X$ lies in $\pvd(\Pi)$.  It follows that the objects $G_{-1} X$ and
$G_0 X$ compactly generate $\cd(\cp_{dg})/\cn$. One easily checks that for
$X$ and $Y$ in $\Pi$ and $-1\leq i,j \leq 0$, we have isomorphisms
\[
\begin{tikzcd}
\RHom_{\cosg(\cR_{dg})}(G_i X, G_j Y) \arrow{r}{_\sim} & \RHom_{\cosg(\cp_{dg})}(G_i X, G_j Y).
\end{tikzcd}
\]
The claim follows by Lemma~4.2 of \cite{Keller94}.
\end{proof}

\section{Group actions}

\subsection{The cyclic group action} \label{ss: The cyclic group action} We use the assumptions and notations
of sections~\ref{ss: The Higgs category} and \ref{s: Comparison with the cosingularity category}. 
In particular, we have the Ginzburg functor $\Pi_2(\cb) \to \Pi_3(\ca,\cb)$ and we abbreviate $\Ga=\Pi_3(\ca,\cb)$.
Since the Ginzburg functor has a relative $3$-Calabi--Yau structure, the cone over
the morphism
\[
\Si^{-3}(\Ga \lten_{\Pi_2(\cb)} \Ga \to \Ga)
\]
is isomorphic to the inverse dualizing bimodule of $\Ga$. By Remark~\ref{remark: perfection of restriction},
the restriction functor from $\cd(\Ga)$ to $\cd(\Pi_2(\cb))$ takes perfect objects to perfect objects. 
It follows that the bimodule $\Ga \lten_{\Pi_2(\cb)} \Ga$ is right perfect. By replacing the original quiver
$Q$ with $Q^{op}$ we obtain that it is also left perfect. If $M$ belongs to $\pvd(\Ga)$, then
the tensor product 
\[
M \lten_\Ga(\Ga \lten_{\Pi_2(\cb)} \Ga)
\]
still belongs to $\pvd(\Ga)$ because for any simple dg $\Ga$-module $S_i$, the complex
\[
\RHom_\Ga(M \lten_\Ga(\Ga \lten_{\Pi_2(\cb)} \Ga), S_i) \iso \RHom_{\Pi_2(\cb)}(M, S_i)
\]
is perfect, since $\Pi_2(\cb)$ is smooth.  Moreover, we know that $\Ga$ is smooth and
connective and $H^0(\Ga)$ is Morita equivalent to a finite-dimensional algebra. By 
Lemma~\ref{lemma: Serre functor equivalence}, it follows that the endofunctor \red{$\Se'$} given by
\[
X \mapsto \Si^{-3}\cone(X \lten_{\Pi_2(\cb)} \Ga \to X)
\]
is an autoequivalence of $\cd(\Ga)$. 
We put $\Omega= \Si^2 \Se'$. For $P\in \Pi_2(\proj kQ)$, we compute that we have
\begin{equation}
\label{eq: description of Omega on P}
\Om G_1 P = G_0 P \ko \quad \Om G_0 P = G_{-1} P\ko \quad \Om G_{-1} P = G_0 \nu P \ko
\end{equation}
where $\nu P = \tau_{\leq 0} \Si^{-1} P$. Notice that $\nu$ induces the Nakayama
functor in the category $H^0(\Pi_2(\proj kQ)) \iso \proj(\La)$, where $\La$ is the (classical)
preprojective algebra of $Q$. It takes the indecomposable projective $P^\La_i$ to 
$P^\La_{i^*}$, where $i^*$ is the unique vertex such that $\Si P_i^Q$ lies in the
$\tau$-orbit of $P_{i^*}^Q$ (here $P_i^Q$ denotes the indecomposable projective
$kQ$-module with simple head $S_i$). We have $\nu^2 \iso \id$.

Recall that $\cp\subseteq \per(\Ga)$ is the image of the
Ginzburg functor $\Pi_2(\cb) \to \Ga$. Thus, the functor $\Om$ takes $\cp$ to itself and induces
autoequivalences of $\cp$ and of $\cp_{dg}$. We see from the description~\ref{eq: description of Omega on P}
that this autoequivalence of $\cp_{dg}$ is of order dividing~$6$. 
By definition, the kernel of the projection functor $\per(\Ga) \to \cc$ is the right orthogonal
of the thick subcategory generated by $\cp$. Thus, the functor $\Om$ induces an
autoequivalence of $\cc$, which we will still denote by $\Om$. 
By Theorem~\ref{thm: Higgs=gpr}, we have a square
\[
\begin{tikzcd} 
\ch \arrow{r}{_\sim} \arrow{d} & \gpr(\cp_{dg}) \arrow{d} \\
\cc \arrow{r}{_\sim}                 & \dgp(\cp_{dg}) \ko
\end{tikzcd}
\]
where the horizontal equivalences take an object $X$ to the restriction
of $\RHom_\cc(?,X)$ to $\cp_{dg}$. Since $\Om$
induces an autoequivalence of $\cp_{dg}$, we obtain that the
autoequivalence $\Om$ of $\cc$ induces an autoequivalence in $\ch$,
which is also of order dividing $6$. By definition, this autoequivalence takes an 
object $M\in \ch$ to the homotopy fiber $\Om M$ of the deflation 
\[
P(M) \to M
\]
where $P(M)=M\lten_{\Pi_2(\cb)} \Ga$ is projective-injective in $\ch$.

Notice that the autoequivalence $\Om$ of $\cd(\Ga)$ is {\em not exact}
with respect to the canonical $t$-structure since we have
\[
\Om(S_X) = \Si^{-1} S_X
\]
where $S_X$ is a `non frozen simple', i.e. the simple
quotient of a dg module $H^0(X^\wedge)$, where $X \in \ca$ is
indecomposable and neither projective nor injective. This also
shows that the autoequivalence $\Omega: \cd(\Ga) \iso \cd(\Ga)$
is of infinite order. We now define another autoequivalence
$\Om'$ of $\cd(\Ga)$ which is of order dividing $6$ and
induces an autoequivalence isomorphic to $\Om$ in $\cc$. The
autoequivalence $\Om'$ thus yields the required cyclic group action.
Let $\cF\subset\per(\Ga)$ be the relative fundamental domain.
Recall that the projection functor $\pi: \per(\Ga) \to \cc$ induces
a $k$-linear equivalence
\[
\begin{tikzcd}
\cF  \arrow{r}{_\sim} & \ch \ko
\end{tikzcd}
\]
which also induces isomorphisms in $\tau_{\leq 0}\RHom$.
We define $\Om'(\Ga)\subset \cF$ to be the full 
subcategory whose image under $\pi: \per(\Ga) \to \cc$
is the subcategory $\Om(\pi(\Ga))$ of $\ch$ so that
we have a diagram
\[
\begin{tikzcd}
\Ga \arrow{d}{\pi} & \Om'(\Ga) \arrow{d}{\pi}\\
\pi(\Ga) \arrow{r}{\Om} & \Om(\pi(\Ga)) \ko
\end{tikzcd}
\]
where all arrows represent equivalences which lift
to quasi-equivalences between the connective
covers of the corresponding dg enhancements (i.e. their
$\tau_{\leq 0}$-truncations). In particular, we obtain
a quasi-equivalence $\Ga \to \Om'(\Ga)$
and thus an autoequivalence $\Om': \cd(\Ga) \to \cd(\Ga)$
taking $\Ga$ to $\Om'(\Ga)$ and inducing an autoequivalence
isomorphic to $\Om$ in the relative cluster category $\cc$.

Since $\Om: \cd(\Ga) \to \cd(\Ga)$ is an autoequivalence, it induces 
autoequivalences of $\per(\Ga)$ and $\pvd(\Ga)$ and similarly for $\Om'$.
Thus, these functors induce autoequivalences of the cosingularity category $\cosg(\Ga)$. 
In the cosingularity category
\[
\begin{tikzcd}
\cosg(\Ga)  \arrow{r}{_\sim} & \rep(kA_2, \cosg(\Pi_2(\proj kQ)))
\end{tikzcd}
\]
the autoequivalence $\Om$ (and $\Om'$) induces the functor $\red{\Se}^*$
which takes a bimodule $X$ considered as a functor
$\per(kA_2) \to \cosg(\Pi_2(\proj kQ))$ to its composition
with the Serre functor \red{$\Se$} of $\per(kA_2)$.

\begin{prop} The autoequivalence $\Om' : \per(\Ga) \to \per(\Ga)$ is algebraic
and 
\begin{itemize}
\item[a)] takes $\cF\subseteq \per(\Ga)$ to itself;
\item[b)] induces $\red{\Se}^*$ in the cosingularity category $\cosg(\Ga)$.
\end{itemize}
Up to isomorphism, it is the unique algebraic autoequivalence of $\per(\Ga)$
with these properties. It is of order $6$ if the involution $i\mapsto i^*$ is non-trivial 
and of order $3$ otherwise.
\end{prop}

\begin{proof}  We have already seen that $\Om'$ satisfies a) and b). We have the diagram
\[
\begin{tikzcd}
\cF \arrow[hook]{d} \arrow{r}{_\sim} &\ch \arrow[hook]{d} \arrow{r}{_\sim} & \cosg(\Ga) \arrow[equals]{d} \\
\per(\Ga) \arrow{r} & \cc \arrow{r} & \cosg(\Ga) \ko
\end{tikzcd}
\]
where the functors in the bottom row are the canonical projections and the
functors in the top row are $k$-linear equivalences inducing isomorphisms
in $\tau_{\leq 0}\RHom$. This easily implies the claimed uniqueness.
\end{proof}

\subsection{The braid group action} \label{ss: The braid group action} Let $B_\Delta$ be the braid group of type $\Delta$,
where $\Delta$ is the underlying Dynkin diagram of $Q$. Let $W_\Delta$ be the corresponding
Weyl group. Let us write $s_i$ (resp.~$\si_i$) for the canonical generators of $W_\Delta$
(resp.~$B_\Delta$).  Recall that the canonical morphism $B_\Delta \to W_\Delta$ admits a canonical
set-theoretic section $w \mapsto \tilde{w}$ which maps an element given by a reduced
expression $s_{i_1} \cdots s_{i_l}$ to $\si_{i_1} \cdots \si_{i_l}$. In particular, the longest element
$w_0$ of $W$ lifts to a canonical element $\tilde{w}_0$ of $B_\Delta$. The
square $\tilde{w}_0^2$ is central in $B_\Delta$ so that the conjugation
$u \mapsto u^*=\tilde{w}_0 u \tilde{w}^{-1}_0$ is an involution. It takes the braid generator
$\si_i$ associated with the vertex $i$ of $\Delta$ to $\si_{i^*}$, where
$i \mapsto i^*$ is the involution described in the preceding section. 
Following section~2.1.6 of \cite{GoncharovShen19}, we define
$B_\Delta^*$ to be the subgroup of $B_\Delta$ fixed by $u\mapsto u^*$. 

We will construct an action of $B_\Delta^*$ on the derived category
$\cd(\cp_{dg})$ which will induce an action on the cluster category $\cc$ 
(which will {\em not} leave the Higgs category $\ch\subseteq \cc$ stable). 
As a first step, to each element $u$ of $B_\Delta^*$, we will assign  an additive 
silting subcategory of $\cd(\cp_{dg})$ in the sense of the following definition:
Recall that a {\em silting subcategory} $\cs$ of a compactly generated triangulated category
$\ct$ is a subcategory of compact objects that generates $\ct$ (as a triangulated category with 
arbitrary coproducts) and such that $\Hom(X, \Si^p Y)$ vanishes for all $X$, $Y$ in $\cs$
and all $p>0$. If $\ca$ is a connective dg category, the representables
$A^\we$, $A\in \ca$, form the {\em standard silting subcategory} of $\cd(\ca)$,
which we simply denote by $H^0(\ca) \subset \cd(\ca)$ (we identify
$H^0(\ca)$ with its image under the Yoneda functor). 
By an {\em additive} silting subcategory, we mean a silting subcategory
stable under taking direct factors and finite direct sums. If $\ca$ is a 
connective dg category such that $H^0(\ca)$ is additive and idempotent-complete
(for example $\Pi_2=\Pi_2(\proj kQ)$), then
the standard silting subcategory $H^0(\ca) \subset \cd(\ca)$ is additive.

\begin{remark} \label{rk: can_S}
For a silting subcategory $\cs \subset \cd(\ca)$,
we write $\cs_{dg}$ for the full dg subcategory of the dg derived category
$\cd_{dg}(\ca)$ whose objects are those of $\cs$. By definition, we
have $H^0(\cs_{dg})=\cs$.  Since $\cs$ is in
particular a set of compact generators of $\cd(\ca)$, by the main
result of \cite{Keller94}, the inclusion $\cs \subset \cd(\ca)$ extends
canonically to an equivalence
\[
\begin{tikzcd}
\can_\cs: \cd(\cs_{dg}) \arrow{r}{_\sim} & \cd(\ca).
\end{tikzcd}
\]
\end{remark}

Let us now recall Mizuno--Yang's classification \cite{MizunoYang24} of the additive
silting subcategories of $\cd(\Pi_2)$: Since $\Pi_2=\Pi_2(\proj kQ)$ is smooth and $2$-Calabi--Yau as a dg category,
the braid group $B_\Delta$ acts on $\cd(\Pi_2)$ (and $\per(\Pi_2)$) by spherical twist functors,
cf.~\cite{SeidelThomas01}. More precisely, the braid generator $\si_i$ sends an 
object $X$ to the cone $\si_i(X) = \tw_{S_i}(X)$ in the triangle
\[
\begin{tikzcd}
\RHom(S_i, X) \ten S_i \arrow{r} & X \arrow{r} & \tw_{S_i}(X) \arrow{r} & \Si \RHom(S_i, X) \ten S_i \ko
\end{tikzcd}
\]
where $S_i$ is the simple $\Pi_2$-module associated with the vertex $i$. Notice that
in the cosingularity category $\Cosg(\Pi_2)=\cd(\Pi_2)/\Pvd(\Pi_2)$, the  morphism
$X \to \tw_{S_i}(X)$ becomes invertible since its cone is a direct sum of shifted copies
of $S_i$. In particular, we see that the action of $B_\Delta$ on $\cd(\Pi_2)$ induces
the trivial action in $\Cosg(\Pi_2)$.

\begin{theorem}[Mizuno--Yang \cite{MizunoYang24}] \label{thm: Mizuno-Yang} The map $u \mapsto u(H^0(\Pi_2))$
is a bijection from $B_\Delta$ onto the set of additive silting subcategories of $\cd(\Pi_2)$.
\end{theorem}

For $P$ and $P'$ in $\Pi_2$, we define a $\Pi_2$-bimodule $\nu$ by
\[
\nu(P, P') = \tau_{\leq 0} \RHom_{\Pi_2}(P, \Si^{-1} P').
\]
We simply write $\nu(?)$ for the derived tensor product $?\lten_{\Pi_2} \nu$.
If $P_i$ is an indecomposable projective $kQ$-module, we have
$\nu(P_i\ten_{kQ} \Pi_2) = P_{i^*} \ten_{kQ} \Pi_2$ and $\nu^2$ is
isomorphic to the identity functor. Using Mizuno--Yang's bijection
(Theorem~\ref{thm: Mizuno-Yang}) 
one sees that, for an element $u$ of $B_{\Delta}$,
the additive silting subcategory $u(H^0(\Pi_2))$ is invariant under
$\nu$ if and only if $u$ belongs to $B_\Delta^*$. For such
a $\nu$-invariant silting subcategory $\cs=\nu(\cs)$, we
define its {\em triangular extension $T(\cs)$} to be the
additive subcategory of $\cd(\cp_{dg})$ generated by
the objects $G_{-1} P$, $G_0 P$ and $G_1 P$, where
$P$ ranges through $\cs$. One easily shows the
following lemma.

\begin{lemma} If $\cs$ is a $\nu$-invariant additive
silting subcategory of $\cd(\Pi_2)$, the subcategory $T(\cs)$
is an additive silting subcategory of $\cd(\cp_{dg})$.
\end{lemma}

Now let $u$ be an element of $B^*_\Delta$. \redv{Let $u \cp$ 
be the triangular extension to $\cd(\cp_{dg})$ of 
the image under $u$ of the standard silting subcategory $H^0(\Pi_2) \subset \cd(\Pi_2)$.}
Since the action of $B_\Delta$ on $\cd(\Pi_2)$ induces the
trivial action in $\Cosg(\Pi_2)$, the image of $u(H^0(\Pi_2))$
in $\Cosg(\Pi_2)$ equals that of the standard silting subcategory.                                           
It is not hard to check that the functors $G_i$, $-1\leq i \leq 1$, take $\Pvd(\Pi_2)$ to
$\Pvd(\cp_{dg})$. It follows that we have the equality
\[
\pi(u \cp) = \pi(\cp) \ko
\]
of  subcategories of $\Cosg(\cp_{dg}) = \cd(\cp_{dg})/\Pvd(\cp_{dg})$,
where $\pi$ denotes the quotient functor 
\[
\begin{tikzcd}
\cd(\cp_{dg}) \arrow{r} & \Cosg(\cp_{dg}).
\end{tikzcd}
\]

\begin{prop} \label{prop: Phi_u} There is a unique quasi-equivalence
\[
\begin{tikzcd} 
\Phi_u :\cp_{dg}  \arrow{r}{_\sim} & u\cp_{dg}
\end{tikzcd}
\]
making the following diagram commutative (in the homotopy category of dg categories)
\[
\begin{tikzcd}
\cp_{dg} \arrow{d} \arrow{r}{\Phi_u} & u\cp_{dg} \arrow{d} \arrow[hook]{r} & \cd_{dg}(\cp_{dg}) \arrow{d}{\pi}\\
\pi(\cp_{dg}) \arrow[equals]{r} & \pi(u\cp_{dg}) \arrow[hook]{r} & \Cosg_{dg}(\cp_{dg}).
\end{tikzcd}
\]
\end{prop}

\begin{proof} It follows from Prop.~\ref{prop: from H to D(P)/N} that the functor $\pi$ restricted
to $\cp_{dg}$ induces isomorphisms in $\tau_{\leq 0} \RHom$. Using the same arguments,
one proves the analogous proposition for the restriction to
$u \cp_{dg}$ of the quotient functor $\cd_{dg}(u \cp_{dg}) \to \Cosg_{dg}(u \cp_{dg})$. Since $\cp_{dg}$ and
$u\cp_{dg}$ are connective, it follows that we have the required quasi-equivalence
$\Phi_u$.
\end{proof}

For $u\in B_\Delta^*$, we now define the action of $u$ on $\cd(\cp_{dg})$ as the composition
$\Psi_u$ of the equivalences
\[
\begin{tikzcd}
\cd(\cp_{dg}) \arrow{r}{\Phi_u} & \cd(u \cp_{dg} ) \arrow{r}{\can_{u\cp}} & \cd(\cp_{dg}).
\end{tikzcd}
\]
obtained from Prop.~\ref{prop: Phi_u} and Remark~\ref{rk: can_S}.

\begin{lemma} 
\begin{itemize}
\item[a)] For $u,v\in B^*_\Delta$, we have $\Psi_{uv} = \Psi_u \circ \Psi_v$ (in the homotopy category
of dg categories).
\item[b)] For $u \in B^*_\Delta$, the auto-equivalence $\Psi_u$ induces an auto-equivalence
of the category $\dgp(\cp_{dg})$ of derived Gorenstein projective dg modules
and thus of the cluster category $\cc \iso \dgp(\cp_{dg})$. 
\end{itemize}
\end{lemma}

We refer to \cite{Liu25} for the details of the (easy) proof.

\def\cprime{$'$} \def\cprime{$'$}
\providecommand{\bysame}{\leavevmode\hbox to3em{\hrulefill}\thinspace}
\providecommand{\MR}{\relax\ifhmode\unskip\space\fi MR }
\providecommand{\MRhref}[2]{%
  \href{http://www.ams.org/mathscinet-getitem?mr=#1}{#2}
}
\providecommand{\href}[2]{#2}



\begin{thebibliography}{10}

\bibitem{Amiot09}
Claire Amiot, \emph{Cluster categories for algebras of global dimension $2$ and
  quivers with potential}, Annales de l'institut {F}ourier \textbf{59} (2009),
  no.~6, 2525--2590.

\bibitem{Auslander71}
Maurice Auslander, \emph{Representation dimension of {A}rtin algebras}, Queen
  Mary College, University of London, 1971.

\bibitem{BuanMarshReinekeReitenTodorov06}
Aslak Bakke~Buan, Robert~J. Marsh, Markus Reineke, Idun Reiten, and Gordana
  Todorov, \emph{Tilting theory and cluster combinatorics}, Advances in
  Mathematics \textbf{204 (2)} (2006), 572--618.

\bibitem{BravDyckerhoff19}
Christopher Brav and Tobias Dyckerhoff, \emph{Relative {C}alabi-{Y}au
  structures}, Compos. Math. \textbf{155} (2019), no.~2, 372--412.

\bibitem{BrownSridhar25}
Michael~K. Brown and Prashanth Sridhar, \emph{Orlov's theorem for dg-algebras},
  Adv. Math. \textbf{460} (2025), Paper No. 110035, 32.

\bibitem{CasalsKellerWilliams23}
Roger Casals, Bernhard Keller, and Lauren Williams, \emph{Arbeitsgemeinschaft:
  Cluster algebras}, Workshops 2023, Oberwolfach Reports, vol.~20,
  Mathematisches Forschungsinstitut Oberwolfach, 2023, pp.~1--44.

\bibitem{Chen14}
Xiao-Wu Chen, \emph{Singular equivalences induced by homological epimorphisms},
  Proc. Amer. Math. Soc. \textbf{142} (2014), no.~8, 2633--2640.

\bibitem{Chen24a}
Xiaofa Chen, \emph{Exact dg categories {I}: Foundations}, arXiv:2402.10694.

\bibitem{Chen24b}
\bysame, \emph{Exact dg categories {II} : {T}he embedding theorem},
  arXiv:2406.11226.

\bibitem{Chen23}
\bysame, \emph{On exact dg categories}, Ph.~D. thesis, Universit\'e Paris
  Cit\'e, arXiv:2306.08231.

\bibitem{Christ22a}
Merlin Christ, \emph{Cluster theory of topological {F}ukaya categories},
  arXiv:2209.06595.

\bibitem{Christ25b}
\bysame, \emph{Cluster theory of topological {F}ukaya categories. {P}art {II}:
  {H}igher {T}eichmüller theory}, arXiv:2510.05925 [math.RT].

\bibitem{Christ24}
\bysame, \emph{Gluing of cluster tilting objects on surfaces}, Talk at the
  meeting `Cat\'egories amass\'ees et sym\'etrie miroir', Strasbourg, May 30,
  2024.

\bibitem{Christ25a}
\bysame, \emph{Induction in perverse schobers and cluster-tilting theory},
  arXiv:2509.01689 [math.RT].

\bibitem{DerksenWeymanZelevinsky08}
Harm Derksen, Jerzy Weyman, and Andrei Zelevinsky, \emph{Quivers with
  potentials and their representations {I}: {Mutations}}, Selecta Mathematica
  \textbf{14} (2008), 59--119.

\bibitem{DerksenWeymanZelevinsky10}
\bysame, \emph{Quivers with potentials and their representations {II}:
  {Applications to cluster algebras}}, J.~Amer.~Math.~Soc. \textbf{23} (2010),
  749--790.

\bibitem{Ding25}
Zhenhui Ding, \emph{On {G}orenstein projective dg modules}, Ph.~D.~thesis in
  preparation, Universit\'e Paris Cit\'e, 2025.

\bibitem{FanKellerQiu24}
Li~Fan, Bernhard Keller, and Yu~Qiu, \emph{Dg enhanced orbit categories and
  applications}, arXiv:2405.00093.

\bibitem{Fei17}
Jiarui Fei, \emph{Cluster algebras and semi-invariant rings {I}. {T}riple
  flags}, Proc. Lond. Math. Soc. (3) \textbf{115} (2017), no.~1, 1--32.

\bibitem{Fei17a}
\bysame, \emph{Cluster algebras and semi-invariant rings {II}: projections},
  Math. Z. \textbf{285} (2017), no.~3-4, 939--966.

\bibitem{Fei17b}
\bysame, \emph{Cluster algebras, invariant theory, and {K}ronecker coefficients
  {I}}, Adv. Math. \textbf{310} (2017), 1064--1112.

\bibitem{Fei19}
\bysame, \emph{Cluster algebras, invariant theory, and {K}ronecker coefficients
  {II}}, Adv. Math. \textbf{341} (2019), 536--582.

\bibitem{Fei21}
\bysame, \emph{Tensor product multiplicities via upper cluster algebras}, Ann.
  Sci. \'Ec. Norm. Sup\'er. (4) \textbf{54} (2021), no.~6, 1415--1464.

\bibitem{FockGoncharov06a}
Vladimir~V. Fock and Alexander~B. Goncharov, \emph{Moduli spaces of local
  systems and higher {T}eichm\"uller theory}, Publ. Math. Inst. Hautes \'Etudes
  Sci. (2006), no.~103, 1--211.

\bibitem{FominPortal}
Sergey Fomin, \emph{Cluster algebras portal}, available at Sergey Fomin's home
  page.

\bibitem{FominZelevinsky02}
Sergey Fomin and Andrei Zelevinsky, \emph{Cluster algebras. {I}.
  {F}oundations}, J. Amer. Math. Soc. \textbf{15} (2002), no.~2, 497--529
  (electronic).

\bibitem{FrankildIyengarJoergensen03}
Anders Frankild, Srikanth Iyengar, and Peter J\o~rgensen, \emph{Dualizing
  differential graded modules and {G}orenstein differential graded algebras},
  J. London Math. Soc. (2) \textbf{68} (2003), no.~2, 288--306.

\bibitem{FrankildJoergensen03}
Anders Frankild and Peter J\o~rgensen, \emph{Gorenstein differential graded
  algebras}, Israel J. Math. \textbf{135} (2003), 327--353.

\bibitem{GeissLeclercSchroeer13}
Ch. Geiss, B.~Leclerc, and J.~Schr\"oer, \emph{Cluster algebras in algebraic
  {L}ie theory}, Transform. Groups \textbf{18} (2013), no.~1, 149--178.

\bibitem{GeissLeclercSchroeer05}
Christof Gei\ss, Bernard Leclerc, and Jan Schr{\"o}er, \emph{Semicanonical
  bases and preprojective algebras}, Ann. Sci. \'Ecole Norm. Sup. (4)
  \textbf{38} (2005), no.~2, 193--253.

\bibitem{GeissLeclercSchroeer06}
\bysame, \emph{Rigid modules over preprojective algebras}, Invent. Math.
  \textbf{165} (2006), no.~3, 589--632.

\bibitem{GeissLeclercSchroeer07c}
Christof Geiss, Bernard Leclerc, and Jan Schr\"oer, \emph{Auslander algebras
  and initial seeds for cluster algebras}, J. Lond. Math. Soc. (2) \textbf{75}
  (2007), no.~3, 718--740.

\bibitem{GeissLeclercSchroeer08a}
Christof Gei\ss, Bernard Leclerc, and Jan Schr{\"o}er, \emph{Preprojective
  algebras and cluster algebras}, Trends in representation theory of algebras
  and related topics, EMS Ser. Congr. Rep., Eur. Math. Soc., Z\"urich, 2008,
  pp.~253--283.

\bibitem{GeissLeclercSchroer12}
Christof Geiss, Bernard Leclerc, and Jan Schr\"{o}er, \emph{Generic bases for
  cluster algebras and the {C}hamber ansatz}, J. Amer. Math. Soc. \textbf{25}
  (2012), no.~1, 21--76.

\bibitem{GoncharovShen19}
Alexander Goncharov and Linhui Shen, \emph{Quantum geometry of moduli spaces of
  local systems and representation theory}, arXiv 1904.10491.

\bibitem{JensenKingSu16}
Bernt~Tore Jensen, Alastair~D. King, and Xiuping Su, \emph{A categorification
  of {G}rassmannian cluster algebras}, Proc. Lond. Math. Soc. (3) \textbf{113}
  (2016), no.~2, 185--212.

\bibitem{Jin20}
Haibo Jin, \emph{Cohen-{M}acaulay differential graded modules and negative
  {C}alabi-{Y}au configurations}, Adv. Math. \textbf{374} (2020), 107338, 59.

\bibitem{IyamaKalckWemyssYang15}
Martin Kalck, Osamu Iyama, Michael Wemyss, and Dong Yang, \emph{Frobenius
  categories, {G}orenstein algebras and rational surface singularities},
  Compos. Math. \textbf{151} (2015), no.~3, 502--534.

\bibitem{KalckYang16}
Martin Kalck and Dong Yang, \emph{Relative singularity categories {I}:
  {A}uslander resolutions}, Adv. Math. \textbf{301} (2016), 973--1021.

\bibitem{Keller06}
Bernhard Keller, \emph{Tilting and derived categories}, Contribution to the
  Handbook of Tilting Theory, edited by L. Angeleri, D. Happel and H. Krause,
  to appear.

\bibitem{Keller94}
Bernhard Keller, \emph{Deriving {D}{G} categories}, Ann. Sci. {\'E}cole Norm.
  Sup. (4) \textbf{27} (1994), no.~1, 63--102.

\bibitem{Keller05}
\bysame, \emph{{On triangulated orbit categories}}, Doc. Math. \textbf{10}
  (2005), 551--581.

\bibitem{Keller06d}
\bysame, \emph{On differential graded categories}, International Congress of
  Mathematicians. Vol. II, Eur. Math. Soc., Z\"urich, 2006, pp.~151--190.

\bibitem{Keller08d}
\bysame, \emph{Triangulated {C}alabi-{Y}au categories}, Trends in
  Representation Theory of Algebras (Zurich) (A.~Skowro\'nski, ed.), European
  Mathematical Society, 2008, pp.~467--489.

\bibitem{Keller11b}
\bysame, \emph{Deformed {C}alabi--{Y}au completions}, Journal f{\"u}r die reine
  und angewandte Mathematik (Crelles Journal) \textbf{654} (2011), 125--180,
  with an appendix by Michel~Van den Bergh.

\bibitem{KellerNicolas12}
Bernhard Keller and Pedro Nicol\'as, \emph{{Weight Structures and Simple dg
  Modules for Positive dg Algebras}}, International Mathematics Research
  Notices (2012), doi:10.1093/imrn/rns009.

\bibitem{KellerWang23}
Bernhard Keller and Yu~Wang, \emph{An introduction to relative {C}alabi-{Y}au
  structures}, Representations of algebras and related structures, EMS Ser.
  Congr. Rep., EMS Press, Berlin, [2023] \copyright 2023, pp.~279--304.

\bibitem{KellerWu23}
Bernhard Keller and Yilin Wu, \emph{Relative cluster categories and higgs
  categories with infinite-dimensional morphism spaces}, with an appendix by
  Chris Fraser and Bernhard Keller, arXiv:2307.12279.

\bibitem{Klapproth22}
Carlo Klapproth, \emph{$n$-{E}xtension closed subcategories of $n$-exangulated
  categories}, arXiv:2209.01128.

\bibitem{Le19a}
Ian Le, \emph{An approach to higher {T}eichm\"uller spaces for general groups},
  Int. Math. Res. Not. IMRN (2019), no.~16, 4899--4949.

\bibitem{Le19}
\bysame, \emph{Cluster structures on higher {T}eichmuller spaces for classical
  groups}, Forum Math. Sigma \textbf{7} (2019), Paper No. e13, 165.

\bibitem{Leclerc10}
Bernard Leclerc, \emph{Cluster algebras and representation theory}, Proceedings
  of the {I}nternational {C}ongress of {M}athematicians. {V}olume {IV} (New
  Delhi), Hindustan Book Agency, 2010, pp.~2471--2488.

\bibitem{Liu25}
Miantao Liu, \emph{On the categorification of the cluster variety of triples of
  flags}, \redv{Ph.~D.~thesis, December 2025}, Universit\'e Paris
  Cit\'e.

\bibitem{MizunoYang24}
Yuya Mizuno and Dong Yang, \emph{Derived preprojective algebras and spherical
  twist functors}, arXiv:2407.02725.

\bibitem{NakaokaPalu19}
Hiroyuki Nakaoka and Yann Palu, \emph{Extriangulated categories, {H}ovey twin
  cotorsion pairs and model structures}, Cah. Topol. G\'eom. Diff\'er. Cat\'eg.
  \textbf{60} (2019), no.~2, 117--193.

\bibitem{Neeman91}
Amnon Neeman, \emph{Some new axioms for triangulated categories}, J. Algebra
  \textbf{139} (1991), 221--255.

\bibitem{Orlov09}
Dmitri Orlov, \emph{Derived categories of coherent sheaves and triangulated
  categories of singularities}, Algebra, arithmetic, and geometry: in honor of
  {Y}u. {I}. {M}anin. {V}ol. {II}, Progr. Math., vol. 270, Birkh\"auser Boston,
  Boston, MA, 2009, pp.~503--531.

\bibitem{Palu08a}
\redv{Yann Palu, \emph{Cluster characters for 2-{C}alabi-{Y}au triangulated
  categories}, Ann. Inst. Fourier (Grenoble) \textbf{58} (2008), no.~6,
  2221--2248.}

\bibitem{Plamondon11a}
Pierre-Guy Plamondon, \emph{Cluster algebras via cluster categories with
  infinite-dimensional morphism spaces}, Compositio Mathematica \textbf{147}
  (2011), 1921--1954.

\bibitem{Plamondon11}
\bysame, \emph{Cluster characters for cluster categories with
  infinite-dimensional morphism spaces}, Adv. Math. \textbf{227} (2011), no.~1,
  1--39.

\bibitem{Bautista04}
Bautista Raymundo, \emph{The category of morphisms between projective modules},
  Comm. Algebra \textbf{32} (2004), no.~11, 4303--4331.

\bibitem{SeidelThomas01}
Paul Seidel and Richard Thomas, \emph{Braid group actions on derived categories
  of coherent sheaves}, Duke Math. J. \textbf{108} (2001), no.~1, 37--108.

\bibitem{Toen07}
Bertrand To{\"e}n, \emph{The homotopy theory of {$dg$}-categories and derived
  {M}orita theory}, Invent. Math. \textbf{167} (2007), no.~3, 615--667.

\bibitem{Wu21}
Yilin Wu, \emph{Relative {C}alabi--{Y}au structures in representation theory},
  Ph.~D.~thesis, Universit\'e Paris Cit\'e, December 2021.

\bibitem{Wu23a}
\bysame, \emph{Categorification of ice quiver mutation}, Math. Z. \textbf{304}
  (2023), no.~1, Paper No. 11, 42.

\bibitem{Wu23}
\bysame, \emph{Relative cluster categories and {H}iggs categories}, Adv. Math.
  \textbf{424} (2023), Paper No. 109040, 112.

\bibitem{Yeung16}
Wai-kit Yeung, \emph{Relative {C}alabi--{Y}au completions}, arXiv:1612.06352.

\end{thebibliography}

\end{document}